\documentclass[11pt,reqno]{amsart}
\usepackage{graphicx}
\usepackage{amsmath}
\usepackage{amsfonts}
\usepackage{amssymb} 
\usepackage{amsthm}
\usepackage{bbm}
\usepackage{bm} 
\usepackage{lipsum}
\usepackage[dvipsnames]{xcolor}
\usepackage{url}
\usepackage{hyperref}
\usepackage{stmaryrd}
\usepackage{mathrsfs}  
\usepackage{mathabx}  
\usepackage[normalem]{ulem}
\usepackage{cancel}

\allowdisplaybreaks

\usepackage{enumerate}

\usepackage{soul}
\usepackage[margin=1in]{geometry}
\numberwithin{equation}{section}

\newtheorem{thm}{Theorem}[section]
\newtheorem{lem}[thm]{Lemma}
\newtheorem{prop}[thm]{Proposition}

\newtheorem{rem}[thm]{Remark}

\def \be {\begin{equation}}
\def \ee {\end{equation}}
\def\setQ{{\mathbb Y}}

\def \E {\mathbb{E}}
\def \P {\mathbb{P}}
\def \R {\mathbb{R}}
\def \dd {\llbracket d \rrbracket}
\def \NN {\llbracket N \rrbracket}
\def \N {\mathcal{N}}
\def \cF {\mathcal{F}}
\def \PP {\mathbf{P}}
\def \EE {\mathbf{E}}
\def \fD {\mathfrak{D}}
\def \fd {\mathfrak{d}}

\def \blx{\boldsymbol x}
\def \bly{\boldsymbol y}

\newcommand{\sredm}[1]{\ifmmode\text{\xout{\ensuremath{\displaystyle \textcolor{blue}{#1}}}}\else\sout{\textcolor{blue}{#1}}\fi} 

\def\ud{\mathrm{d}}
\def\E{{\mathbb E}}
\def\EN{\llbracket N \rrbracket}
\def\ES{\llbracket d \rrbracket}

\def\R{{\mathbb R}}
\def\RR{{\mathbb R}}
\def\cP{{\mathcal P}}

\def\Y{{\mathcal Y}}

\def \fD {\mathfrak{D}}
\def \fd {\mathfrak{d}}

\def\Y{{\mathbb Y}}

\def\cS{{\mathcal S}}

\def\one{{\mathbbm 1}}

\def\u{z}
\def\v{w}

\def\mumin{\widetilde\mu^{\sqbullet,\min}}

\def\wt{\widetilde}

\begin{document}
\title[Limits of weighted games]{Finite state mean field games with Wright Fisher common noise as limits of $N$-player weighted games}\thanks{-- This is the final version of the paper. To appear in {\it Mathematics of Operations Research}.}

 \author[E. Bayraktar]{Erhan Bayraktar \address{(E. Bayraktar) Department of Mathematics, University of Michigan, Ann Arbor, Michigan 48109, United States}\email{erhan@umich.edu}} 
 
\author[A. Cecchin]{Alekos Cecchin \address{(A. Cecchin) Centre de Math\'ematiques Appliqu\'ees, \'Ecole Polytechnique,
91128 Palaiseau, France}\email{alekos.cecchin@polytechnique.edu}}

 \author[A. Cohen]{Asaf Cohen \address{(A. Cohen) Department of Mathematics, University of Michigan, Ann Arbor, Michigan 48109, United States}\email{asafc@umich.edu}}

\author[F. Delarue]{ Fran\c{c}ois Delarue 
\address{(F. Delarue) Universit\'e C\^ote d'Azur, CNRS, Laboratoire J.A. Dieudonn\'e, 06108 Nice, France
}\email{francois.delarue@unice.fr
}}

\thanks{-- E. Bayraktar is partially supported by the National Science Foundation under grant  DMS-2106556 and by the Susan M. Smith chair. 
A. Cecchin benefited from the support of LABEX Louis Bachelier Finance and Sustainable Growth - project 
ANR-11-LABX-0019, ECOREES ANR Project, FDD Chair and Joint Research Initiative
FiME in partnership with Europlace Institute of Finance.	
A. Cecchin and F. Delarue acknowledge the financial support of French ANR project ANR-16-CE40-0015-01 on ``Mean Field Games''.
F. Delarue is also supported by French ANR project 
ANR-19-P3IA-0002 -- 3IA C\^ote d'Azur -- Nice -- Interdisciplinary Institute for Artificial Intelligence. 
A. Cohen acknowledges the financial support of the National Science Foundation (DMS-2006305).
}

\date{\today}








\begin{abstract}
Forcing finite state mean field games by a relevant form of common noise is a subtle issue, which has been addressed only recently.
Among others, one possible way is to subject the simplex valued dynamics of an equilibrium by a so-called Wright--Fisher noise, very much in the spirit of stochastic models in population genetics. A key feature is that such a random forcing preserves the structure of the simplex, which is nothing but, in this setting, the probability space over the state space of the game. The purpose of this article is hence to elucidate the finite player version and, accordingly, to prove that $N$-player equilibria indeed converge towards the solution of such a kind of Wright--Fisher mean field game. Whilst part of the analysis is made easier by the fact that the corresponding master equation has already been proved to be uniquely solvable under the presence of the common noise, it becomes however more subtle than in the standard setting because the mean field interaction between the players now occurs through a weighted empirical measure. In other words, each player carries its own weight, which hence may differ from $1/N$ and which, most of all, evolves with the common noise.

\end{abstract}

\maketitle

\tableofcontents

\section{Introduction}
\vskip 10pt

In our earlier article \cite{BCCD-arxiv}, we introduced a form of Mean Field Games (MFGs) on a finite state space, with the peculiarity of being driven by a kind of common noise that is reminiscent of the Wright--Fischer model in population genetics. Noticeably, thanks to the diffusive effect of the common noise on the finite dimensional simplex, we succeeded to prove that the master equation associated with this MFG admitted a unique solution and, in turn, that the MFG itself was uniquely solvable without any monotonicity assumption. The purpose of this new work is to address the asymptotic behavior of the finite player counterpart of this mean field game. In a word, we here set up a corresponding finite state many player game with mean field interactions driven by a kind of 
Wright--Fisher common noise and we establish the convergence of its equilibria towards the solution of the MFG 
constructed in \cite{BCCD-arxiv}. Whilst this program looks by now quite classical in the field, the originality of our works lies in the fact that
two (instead of one) attributes 
are here assigned 
 to
each of the $N$ players entering the finite game: Indeed, not only players have their own state (location), but they are also
given a weight (we sometimes say ``a mass''), accounting for their influence onto the rest of the population. In other words, the empirical distribution that enters the definition of the weak interaction between the players is no longer a uniform distribution, but a more general finite probability weighted by the masses (or the weights) assigned to all of them; below, we call it ``the weighted empirical distribution''. Importantly, the weights are subjected to 
some external common noise, and this makes the main feature of our model. In a nutshell, 
the state dynamics of the system should be regarded as a system of controlled interacting Markov chains,  
each of the players controlling the rate of its transition between states. As for the weights, 
which are the second attribute of each player, they indeed jump simultaneously according to the common noise, the form of the latter being inspired by the Wright--Fisher model from population genetics \cite{fis1999,wri1931} consistently with the limiting model addressed in \cite{BCCD-arxiv}. To make it clear, the common noise shuffles the weights of the players at a rate proportional to $N$ according to a multinomial distribution with parameters $N$ and the weighted empirical distribution of the system itself. A key fact, which we make clear in the sequel of the text, is that the mass shared by one player among all the players with the same state is the same after and before the common shuffling. On the top of those dynamics, 
each of the players aims to minimize an expected weighted cost.  
Very importantly, the expected weighted cost to each player includes the own mass of the player, which plays the role of a ``density''. Put it differently, 
each player has its own perception of the uncertainties, depending on its own mass.

Our main result says that the cost per unit weight under the unique Nash equilibrium, or equivalently 
the equilibrium value of each player but renormalized by its own mass, 
 converges to the solution of the master equation. Generally speaking, the latter is a parabolic Partial Differential Equation (PDE) describing the equilibrium cost (or the value) of the MFG analyzed in \cite{BCCD-arxiv}, see Theorem \ref{thm:convergence_value} below.
In our setting, the master equation reads as a PDE driven by a so-called Kimura like operator, which is a second-order operator acting on 
functions defined on the finite dimensional simplex. This Kimura operator should be regarded as (a variant of) the generator of 
a multi-dimensional Wright--Fischer diffusion process. One of the main underlying obstacles is that this operator degenerates near the boundary of the simplex, 
as a consequence of which corresponding harmonic functions may be singular at the boundary. A major reference on the subject is the monograph by Epstein and Mazzeo \cite{EpsteinMazzeo}, 
which is repeatedly cited in 
\cite{BCCD-arxiv}. In order to avoid as much as possible the influence of the boundary, a key step in \cite{BCCD-arxiv} is to force 
the dynamics
by an extra drift that points inward the simplex at the boundary. 
Our finite player model exhibits the same feature and, in the end, the extra drift must be strong enough to ensure that all the results indeed hold true.   
We eventually use the convergence of the renormalized value of the finite game to show that 
the weighted empirical distribution, under the Nash equilibrium, converges to the flow of measures described in the mean field game, see Theorem \ref{thm:convergence} below. 

Before we compare these two results with the existing literature, we feel fair to recall that the theory of MFGs was initiated by the seminal works of Lasry and Lions \cite{Lasry2006,LasryLions2}, and Huang, Malham{\'e}, and Caines \cite{Huang2006,Huang2007}. 
The very first objective of it was precisely to provide an asymptotic formulation for 
 many player games with weak interaction. Mathematically speaking,
 the connection between finite games and mean field games raises many subtle and challenging questions. 
Actually, 
  there are two types of convergence results in the MFG literature for justifying the passage from finite to infinite games: One strand is concerned with constructing asymptotic Nash equilibria in the $N$-player game using solutions of the MFG; The other one is to show that the costs and empirical measure of the $N$-player game under Nash equilibria converge to a mean field equilibrium, which is in fact the more challenging problem between the two. In order to handle the latter problem,    
Cardaliaguet, Delarue, Lasry, and Lions \cite{CardaliaguetDelarueLasryLions} used the master equation. This method was later utilized in a finite state Markov chain setup \cite{bay-coh2019, bayzhang2019, cecdaifispel, cec-pel2019} and our own approach is obviously inspired from it. For further discussion on MFGs and the convergence problem, the reader is referred to Carmona and Delarue \cite{CarmonaDelarue_book_I,CarmonaDelarue_book_II}, Lacker \cite{Lacker2015general, Lacker2017, lac2020closed} and Fischer \cite{Fischer2017}, the latter two authors having successfully developed other strategies based upon tightness arguments. 

In comparison with all the previously cited references, 
one peculiarity of our model is the role of the common noise. Differently from the aforementioned references
where 
 it was already allowed to be present, 
the common noise is fundamental in our work as it guarantees the smoothness of the solution of the master equation and, subsequently, the uniqueness of 
the asymptotic equilibrium. Accordingly, our need to have a smooth solution to the limiting master equation
 dictates the form of the common noise in our finite player game. In this respect, it is fair to say that, independently of the well-posedness of the master equation, adding common noise to finite state mean field games is in fact a subtle issue. 
For instance, 
this difficulty was pointed out by Bertucci, Lasry, and Lions \cite{BertucciLasryLions}, who solved it by 
introducing a common noise in the form of simultaneous jumps governed by a deterministic transformation of the state space. For an alternative model with the simultaneous jumps to the state space see \cite{belak2019continuous}. 
Although it shares some similarities, the idea introduced in \cite{BCCD-arxiv}
is slightly different, since the simultaneous jumps therein obey a stochastic rule given in the form of a  Wright--Fisher noise. 
Accordingly, the finite player version that we construct below must obey some rules that permit to recover asymptotically the same Wright--Fisher noise. 
This is precisely the step where the extra attribute, namely the weight, comes in: Thanks to it, we are able to inject the Wright--Fisher common noise in the finite game, and this results in a random measure in the limit. 
In other words, we shift the simultaneous jumps in locations, as used in \cite{BertucciLasryLions}, to simultaneous jumps in weights\footnote{The interested reader may have a look at \cite{delarue-cetraro} for another construction of a finite state MFG with common noise.}.
To our mind, allowing the empirical distribution to be non-uniform is an interesting idea in its own, independently of the 
precise model that we here address.    
Typically, weakly interacting many player games are indeed formulated such that the weight of each player is fixed and deterministic in time, that is $1/N$. More than often, the population size also remains fixed in time. A few exceptions are the models studied by \cite{cam-fis2018, nut2018, cla-ren-tan2019} and 
\cite{bertucciJMPA},  where the size of the population evolves in time, yet the weights are still homogeneous with respect to the current population. In
contrast, in our model, while the size of the population remains fixed, the common noise yields a stochastic evolution of the weights. 

We now outline the key steps in establishing the convergence of the many player game to the MFG. 
In Section \ref{sec:2},  
we characterize the unique Nash equilibrium of the $N$-player game via a system of non-local differential equations, which requires special care 
due to the possible degeneracies of the weights in the empirical distribution of the system. 
To prove the convergence of the normalized Nash system to the solution of the master equation, which is our first main result, we plug in the solution of the master equation to the normalized Nash system. This results in a remainder; 
the estimates of this remainder are stated in Subsection \ref{sec:5} and proved later on in 
Section \ref{se:7:b}.
In comparison with similar approaches in MFG theory, we face additional difficulties that arise when the weighted empirical distribution approaches the boundary of the simplex. Indeed, the challenge in estimating the remainder comes from the aforementioned fact that the weights are shuffled in a way that 
each player receives a fair share with respect to the players in the same state. Because of this shuffling, the inverse of the weighted empirical distribution shows up in the non-local term in the differential equation describing the renormalized Nash system. This requires us to analyze the boundary behavior of the empirical distribution process. Moreover, 
we also need to estimate the moments of the weight process. These preliminary estimates are stated in Subsection \ref{sec:4}
and proved in Section \ref{sec:6:b} by concentration estimates for the multinomial distribution. {Even though boundary estimates were also needed in 
\cite{BCCD-arxiv}
to analyze the master equation, we use here different methods since the stochastic flow of measures is no longer a diffusion but a jump process.
In order to guarantee all these estimates to be true, we need the aforementioned inward pointing drift to be strong enough; once again, this is consistent 
with the analysis carried out in 
\cite{BCCD-arxiv}.} 
Lastly, Section \ref{sec:3} is another preparatory section, where we provide another interpretation of the normalized Nash system, based on an auxiliary game; 
interestingly, this alternative representation supplies us with some very useful uniform bounds on the equilibrium feedbacks. Our second main result concerns the convergence of the weighted empirical distribution towards the solution of the MFG. In order to prove it, we combine the hence established convergence 
of the normalized Nash system together with diffusion approximation arguments. {We emphasize that the limiting process is a diffusive stochastic Fokker--Planck equation and as such has a different structure from the weighted empirical distribution.}

The rest of the paper is organized as follows. Section \ref{sec:2} starts with a description of the MFG model and the master equation studied in \cite{BCCD-arxiv}. Then it provides the $N$-player game with common noise as well as the analysis of the Nash system. It ends with the statements of the main results: Theorems \ref{thm:convergence_value} and \ref{thm:convergence}. In the next two sections, we prepare for the proofs of the main results. Section \ref{sec:3} is devoted to the analysis of the normalized system as well as presenting an auxiliary process associated with it. In Subsection \ref{sec:4}, we state estimates on the weighted empirical distribution and the moments of the weights and, in 
Subsection \ref{sec:5}, we reformulate the solution of the master equation as an approximate solution of the normalized $N$-Nash system. Finally, the proofs of the main results are provided in Section \ref{sec:6}. Sections \ref{sec:6:b}
and \ref{se:7:b} are devoted to the proofs of the auxiliary results stated in Section \ref{sec:4:b}. 
A future outlook is given in Section \ref{sec:7}. In the Appendix, we prove some of the results about uniqueness of the Nash equilibria of the 
$N$-player, that  are technically demanding but less important for the proof. 
\vskip 5pt

In the rest of this section we will list some frequently used notation.

\subsection*{Notation}
Denote $\ES:=\{1,\ldots,d\}$. For an integer $N \geq 1$, a tuple $\blx =(x^{1},\cdots,x^{N}) \in \ES^N$ and another integer $l \in \EN$, we write 
$\blx^{-l}$ {for} the tuple $(x^{1},\cdots,x^{l-1},x^{l+1},\cdots,x^{N})$. 
For some $j \in \ES$, the notation $(j,\blx^{-l})$ is then understood as $(x^{1},\cdots,x^{l-1},j,x^{l+1},\cdots,x^{N})$.
A tuple $\bly = (y^1,\cdots,y^N) \in ({\mathbb R}_{+})^N$ is said to belong to $\setQ$ if 
the entries belong to the set of nonnegative rational numbers $\mathbb{Q}_+$ and sum to $N$. 
For two  tuples $\blx \in \ES^N$ and $\bly = (y^{1},\cdots,y^{N}) \in ({\mathbb R}_{+})^N$, we also introduce the weighted empirical measure:
\[
\mu^N_{\bm{x},\bm{y}}=\frac1N \sum_{l=1}^N y^l \delta_{x^l}. 
\]
Clearly, 
$\mu^N_{\bm{x},\bm{y}}$
is a measure on $\ES$, which we may regard as a $d$-tuple; for $i \in \ES$, we let
\begin{equation}
\label{eq:muNi}
\mu^N_{\bm{x},\bm{y}}[i]:= \frac1N \sum_{l=1}^N y^l \one_{\{x^l=i\}}.
\end{equation}
Scalar product between vectors $z$ and $w$ in $\R^d$ is denoted by $z\cdot w$.

Below, we denote by ${\mathcal P}(\ES)$ 
the space of probability measures on $\ES$, which we identify with the 
simplex $\cS_{d-1}:=\{({p}^1,\cdots,{p}^d) \in ({\mathbb R}_{+})^d : \sum_{i=1}^d p^i=1\}$. Also, we define 
$\hat\cS_{d-1}:=\{(x^1,\cdots,x^{d-1}) \in ({\mathbb R}_{+})^d : \sum_{i=1}^{d-1} x^i\le 1\}$.
In particular, we sometimes regard the Dirac mass $\delta_{i}$, for $i \in \ES$, as the $i$th vector of the canonical basis of $\R^d$.
 
Whenever $\mu$ is a probability measure on $\ES$ (i.e. $\mu \in {\mathcal P}(\ES)$), we call ${\mathcal M}_{N,\mu}$ the multinomial distribution 
of parameters $N$, $(\mu[1], \cdots, \mu[d])$, namely 
\begin{equation*}
{\mathcal M}_{N,\mu}({\boldsymbol k}) = \frac{N!}{k^{1}!\cdots k^{d}!} \prod_{i=1}^d \mu[i]^{k^{i}}. 
\end{equation*}
for ${\boldsymbol k} \in {\mathbb N}^d$ with $k^{1}+\cdots+k^{d}=N$.
In order to have another representation of the multinomial distribution, we assume that we are given a probability space, say $(\Xi,{\mathcal G},{\mathbf P})$, equipped with a collection of random variables $(S_{N,\mu})_{\mu \in {\mathcal P}(\ES)}$, all of them with values in the set of multi-indices ${\boldsymbol k} \in {\mathbb N}^d$ that sum to $N$, such that 
${\mathbf P} \circ S_{N,\mu}^{-1} = {\mathcal M}_{N,\mu}$ for all $\mu \in {\mathcal P}(\ES)$. Below, we denote the 
$d$ entries of $S_{N,\mu}$ in the form $(S_{N,\mu}[i])_{i \in \ES}$. Expectation under ${\mathbf P}$ is denoted 
${\mathbf E}$.
When the value of $N$ is fixed and there is no ambiguity, 
we just write $(S_\mu)_{\mu}$ for 
$(S_{N,\mu})_{\mu}$.

For a real-valued function $v$ on $\ES^N \times ({\mathbb Q}_{+})^N$, 
we define the first-order variation (or discrete gradient)
at point $(\bm{x},\bm{y})$ 
and
in the direction $l$
as the tuple
$\Delta^l v(\bm{x},\bm{y})[\bullet] \in\R^d$, defined by 
$\Delta^l v(\bm{x},\bm{y}) [j] = v((j,\bm{x}^{-l}),\bm{y})-v(\bm{x},\bm{y})$.
Quite often in the text, we indeed put a bullet symbol $\bullet$
 to emphasize that the related quantity has to be understood as a $d$-tuple.

Lastly, for a real $x$, the positive part of $x$ is denoted by $x_{+}=\max(x,0)$.

\vspace{6pt}
\noindent{\bf Derivatives on the simplex.}
\label{subsub:derivatives:simplex} The formulation of the master equation is given using intrinsic derivatives. 
Since this is quite common material, we feel more convenient to introduce them now, as part of our notation. For a real-valued function $h$ defined on ${\mathcal S}_{d-1}$, define the functions $\hat h^i:\hat\cS_{d-1}\to\R$, $i\in\ES$, as follows:
\begin{equation}
\notag
\begin{split}
\hat{h}^i \bigl(p^{-i}\bigr) &:= h (p)
=  h
\Bigl( p_{1},\cdots,p_{i-1},1- \sum_{k \not = i} p_{k},p_{i+1},\cdots,p_{d} \Bigr),
\\
&\textrm{\rm with}
\ p^{-i} = \Bigl(p_{1},\cdots,p_{i-1},p_{i+1},\cdots,p_{d}\Bigr),
\end{split}
\end{equation}
We then say that $h$ is differentiable on ${\mathcal S}_{d-1}$ if 
$\hat{h}^i$ is differentiable on $\hat{\mathcal S}_{d-1}$ for some (and hence for any) $i \in \ES$.
 The corresponding intrinsic gradient reads in the form ${\mathfrak D} h=(\mathfrak{d}_1 h, \dots, \mathfrak{d}_d h)$, with
\begin{equation*} 
\mathfrak{d}_{i}h(p)= - \frac1d \sum_{j \not =i} \partial_{p_{j}} \hat{h}^i\bigl( p^{-i} \bigr), \quad p \in {\mathcal S}_{d-1},\; i\in\ES.
\end{equation*}
It follows that $\sum_j \mathfrak{d}_{j}h =0$, meaning that the gradient belongs to the tangent space to the simplex. In particular, note that, if $h$ is defined on a neighborhood of ${\mathcal S}_{d-1}$ in $\R^d$, the following holds true
\begin{equation*}
\mathfrak{d}_{i} h(p) - \mathfrak{d}_{j}h(p) =  \partial_{p_{i}} h(p) - \partial_{p_{j}}h(p), 
\end{equation*}
for $i,j \in \ES$ and $p \in {\mathcal S}_{d-1}$. 
We may define in a similar manner the second order derivatives $({\mathfrak d}^2_{i,j} h)_{i,j \in \ES}$ on the simplex. We refer to \cite[\S 3.1.2]{BCCD-arxiv} for more details.


\section{The model and the main results}\label{sec:2}
\subsection{Finite mean field games with Wright--Fisher common noise}
Let us prepare the ground and recall the setup and the main results from the finite state MFG with common noise analyzed in \cite{BCCD-arxiv}.
\subsubsection{Formulation of the mean field game and the master equation}

Consider a filtered probability space $(\Omega,\cF,\{\cF_t\}_{0\le t\le T},\mathbb{P})$, where the filtration $\{\cF_t\}_{0\le t\le T}$ satisfies the usual conditions. The probability space supports a collection of independent standard Brownian motions $((W^{i,j})_{0 \le t \le T})_{i,j \in \ES : i \not = j}$. 
To describe the MFG, we consider a {\it representative player}, whose
statistical state (conditional on the noise 
$((W^{i,j})_{0 \le t \le T})_{i,j \in \ES : i \not = j}$) 
is tracked through a process $(Q_t=(Q^i_t:i\in\ES))_{0\le t\le T}$ with the specific feature that each $Q_{t}$ is a density on the product space $\Omega \times \ES$. 
According to 
\cite{BCCD-arxiv} and consistently with the finite player version given in the next Subsection \ref{subse:2:2}, $Q_{t}$ describes the statistical distribution of the state 
$X_{t}$ of the representative player (which takes values in $\ES$) 
under \textit{its own perception of the 
world}, the latter being formalized through a density $Y_{t}$ on the product space $\Omega \times \ES$.
The formulation of the MFG then relies on:


\begin{itemize}
\item 
The own control of the representative player, which is given in feedback form by a (bounded and measurable) function $\beta:[0,T]\times\ES\times\cS_{d-1}\to\R^d$, such that, for any $(t,i,p)\in[0,T]\times\ES\times\cS_{d-1}$,
\begin{align}\notag
\beta_j(t,i,p)\ge 0,\quad z\in\ES\setminus\{i\},\quad\beta_i(t,i,p)=-\sum_{j\ne i}\beta_j(t,i,p);
\end{align}
for $j\ne i$, the value $\beta_j(t,i,p)$ is the rate of transition for the representative player at time $t$ from its state $i$ to state $j$, while the environment is in state $p$.  
\item The {\it stochastic environment} $(P_t=(P^i_t:i\in\ES))_{0\le t\le T}$. This is a progressively-measurable simplex-valued process whose dynamics is given through a fixed point type of argument in the sequel. This is often referred to as the mean field game solution. Essentially, this is the 
conditional distribution of the system under equilibrium;
\item the common noise, which
is formed by the antisymmetric Brownian motions $((\overline W^{i,j}_{t} := (W_{t}^{i,j}-W_{t}^{j,i})/\sqrt{2})_{0 \leq t \leq T})_{i,j \in \ES : i \not =j}$ 
and a parameter\footnote{Here, the upper bound for $\varepsilon$ is arbitrary and could be replaced by any other finite positive real.} $\varepsilon\in(0,1)$, referred to as the {\it intensity of the common noise}. 
\end{itemize}
The dynamics of the representative player are given explicitly by
\begin{equation}
\label{eq:weak:sde:final:2:b}
\begin{split} 
\ud Q_{t}^i &=   \sum_{j \in \ES} \Bigl( Q_{t}^j \bigl( \varphi(P_{t}^i) + \beta_i(t,j,P_{t}) \bigr) - Q_{t}^i 
\bigl( \varphi(P_{t}^j) 
+  \beta_j(t,i,P_{t}) \bigr)
\Bigr) \ud t
+  {\varepsilon}  
Q_{t}^i \sum_{j \in \ES}
\sqrt{\frac{P_{t}^j}{P_{t}^i}}  
\ud \overline W_{t}^{i,j}, 
\end{split}
\end{equation}
for $t \in [0,T]$, with the same deterministic initial condition as $P$; that is,  $Q_0=P_0$. 
Here, $\varphi$ is 
 a non-increasing Lipschitz function from $[0,\infty)$ into itself such that 
\begin{equation}
\label{eq:varphi}
\varphi(r) := \left\{ \begin{array}{l}
\kappa, \quad r \leq \delta,
\\
0, \quad r >2 \delta,
\end{array}
\right.
\end{equation}
where $\delta$ is a fixed arbitrary positive parameter. 
The role of $\varphi$ 
is clarified in the statement of Proposition 
\ref{prop:previous} below. For a sufficiently large value of $\kappa$,  it forces the coordinates of 
the process
$(Q_{t})_{0 \leq t \leq T}$ (and in the end of 
$(P_{t})_{0 \leq t \leq T}$ itself, at least 
whenever $(P_{t})_{0 \leq t \leq T}$ is indeed chosen as the fixed point)
to stay (strictly) positive, and even more to stay away from $0$ with large probability.
In this respect, it is worth recalling from 
\cite{BCCD-arxiv} that  
$(Q_{t})_{0 \leq t \leq T}$
does \textbf{not} take values in the simplex. To put it clear, $Q_{t}$ only defines a density on the product space $\Omega \times \ES$: Using terminologies from statistical mechanics, 
it is an annealed but not a quenched density.

The representative player aims to minimize the following {\it cost function} 
\begin{equation}
\notag
\begin{split}
{\mathcal J}\bigl({\boldsymbol \beta},{\boldsymbol P}\bigr)
&:= \sum_{i \in \ES}
{\mathbb E}\biggl[ Q_{T}^i g (i,P_{T}) 
+ \int_{0}^TQ_{t}^i
 \Bigl(   f \bigl(i,P_{t} \bigr) + \frac12
 \sum_{j \not = i}
\bigl\vert \beta_j(t,i,P_t)  \bigr\vert^2
\Bigr) \ud t
\biggr],
\end{split}
\end{equation}
where $f,g:\ES\times \cS_{d-1}\to\R$  satisfy suitable assumptions given in Proposition \ref{prop:previous}.

Finally, a \emph{solution of the mean field game (with common noise)} is a pair $({\boldsymbol P}, \alpha)$ 
such that
\begin{itemize}
\item[(i)] ${\boldsymbol P}=(P_{t})_{0 \leq t \leq T}$ is a $\cS_{d-1}$-valued process, progressively measurable with respect to $\mathbb{F}^{\boldsymbol W}$, with some fixed $p_{0}=(p_{0,i})_{i \in \ES} \in {\mathcal S}_{d-1}$ as initial condition, and 
$\alpha : [0,T] \times \ES \times \cS_{d-1} \times \ES \rightarrow \RR$ is a bounded feedback strategy;
\item[(ii)] ${\boldsymbol P}$ and $\alpha$ satisfy  in the strong sense the equation
\begin{equation}
\notag\label{eq:weak:sde:final}
\begin{split}
\ud P_{t}^i &=  \sum_{j \in \ES} \Bigl( P_{t}^j \bigl( \varphi(P_{t}^i) + \alpha(t,j,P_{t})(i) \bigr) - P_{t}^i \bigl( \varphi(P_{t}^j) 
+ \alpha(t,i,P_{t})(j) \bigr)
\Bigr) \ud t
 + {\varepsilon}  
\sum_{j \in \ES}
\sqrt{P_{t}^i P_{t}^j}  
\ud \overline W_{t}^{i,j};
\end{split}
\end{equation}
\item[(iii)] ${\mathcal J}\bigl(\alpha,{\boldsymbol P}\bigr)\leq{\mathcal J} \bigl({\boldsymbol \beta},{\boldsymbol P}\bigr)$ for any admissible strategy ${\boldsymbol \beta}$.
\end{itemize}
We say that the solution $({\boldsymbol P},\alpha)$ is unique if given another solution $(\widetilde{\boldsymbol P}, 
 \widetilde{\alpha})$, we have 
$P_t =\widetilde{P}_t$  for any $t\in [0,T]$, $\mathbb{P}$-a.s., and 
$\alpha(t,i,P_{t})(j)=\widetilde \alpha(t,i,P_{t})(j)$
 $dt\otimes\mathbb{P}$-a.e., for each $i,j\in\ES$.

The existence of a unique MFG solution was established in \cite{BCCD-arxiv} by showing that the associated master equation has a unique smooth solution. In this case, the master equation is a system of $d$ parabolic PDEs set on the $(d-1)$-dimensional simplex: 

\begin{align}
\notag&\partial_t U^i(t,p) 
-\frac12\sum_{j\ne i}(U^i(t,p)-U^j(t,p))_+^2
+f^i(p) 
+
\sum_{j \in \ES} \varphi(p_{j})
\bigl[ U^j(t,p) -   U^i(t,p) \bigr]
\\\label{eq:master:equation:intrinsic}
&\qquad + 
\sum_{j,k \in \ES} p_{k}\bigl[ \varphi(p_{j}) + \bigl(U^k(t,p)-U^j(t,p)\bigr)_+ \bigr] \left( \fd_{j} U^i(t,p) - \fd_{k} U^i(t,p)\right) \\\notag
&\qquad+ \varepsilon^2\sum_{j  \neq i } p_{j} \left( \fd_{i} U^i(t,p) - \fd_{j} U^i(t,p)\right)
+\frac{\varepsilon^2}{2} \sum_{j,k \in \ES}(p_j \delta_{j,k}-p_{j} p_{k}) \fd^2_{j,k} U^i(t,p) =0,\\\notag
&U^i(T,p)= g^i(p),
\end{align}
for $(t,p) \in [0,T] \times \textrm{\rm Int}(\hat{\mathcal S}_{d-1})$. 

The statement of the theorem requires that the cost functions are smooth enough in the sense of so-called Wright--Fisher spaces. The reader is referred to 
\cite{EpsteinMazzeo}
(see also \cite[\S 3.2.2]{BCCD-arxiv}) for a thorough discussion on these spaces.
 We provide below a short reminder along the lines of 
\cite{cec-del2020}: 
\begin{enumerate}
\item 
$\mathcal{C}^{\gamma}_{\rm WF}(\mathcal{S}_{d-1})$ consists of continuous functions  on ${\mathcal S}_{d-1}$ that are $\gamma$-H\"older continuous up to the boundary with respect to the metric associated with the 
Wright--Fisher noise;
\item 
$\mathcal{C}^{2+\gamma}_{\rm WF}(\mathcal{S}_{d-1})$ consists of continuous functions  on ${\mathcal S}_{d-1}$ that are twice continuously differentiable in the $(d-1)$-dimensional interior of ${\mathcal S}_{d-1}$, 
with derivatives satisfying a suitable behavior at the boundary and a suitable H\"older regularity that depends on the order of the derivative; in particular, 
the derivatives
of order 1 are H\"older continuous up to the boundary, but the derivative of order $2$ may blow up at the boundary and be only locally H\"older continuous in the interior;
\item $\mathcal{C}^{1+\gamma/2,2+\gamma}_{\rm WF}([0,T]\times\mathcal{S}_{d-1})$
is the parabolic version of 
$\mathcal{C}^{2+\gamma}_{\rm WF}(\mathcal{S}_{d-1})$; 
it consists of continuous functions  on $[0,T] \times {\mathcal S}_{d-1}$ that are continuously differentiable in 
time $t \in [0,T]$ and that are twice continuously differentiable in space in the $(d-1)$-dimensional interior of ${\mathcal S}_{d-1}$, with derivatives satisfying a suitable behavior at the boundary and a suitable H\"older regularity; in particular, the time derivative and the space derivatives of order $1$ are H\"older continuous up to the boundary but the derivative of order $2$ may blow up at the boundary. 
\end{enumerate}

\begin{prop}[Theorems 3.2 and 3.4 in \cite{BCCD-arxiv}]\label{prop:previous}
Assume that, for some $\gamma >0$,  each $f(i,\cdot)$, for $i \in \ES$, 
belongs to 
${\mathcal C}_{\textrm{\rm WF}}^{\gamma}(  {\mathcal S}_{d-1})$,
and
each $g(i,\cdot)$, for $i \in \ES$, belongs to 
${\mathcal C}_{\textrm{\rm WF}}^{2+\gamma}({\mathcal S}_{d-1})$.
Then,
for any $\varepsilon \in (0,1)$,
there exist a universal exponent $\eta \in (0,1)$ (hence independent of $\varepsilon$)
and a threshold $\kappa_{0}>0$, 
only depending 
on $\varepsilon$, $\|f\|_{\infty}$, $\| g \|_{\infty}$ and $T$, such that, for any $\kappa \geq \kappa_{0}$
and $\delta \in (0,1/(4\sqrt{d}))$, the master equation 
\eqref{eq:master:equation:intrinsic}
has a unique solution in  
$[{\mathcal C}_{\textrm{\rm WF}}^{1+\gamma'/2,2+\gamma'}([0,T] \times {\mathcal S}_{d-1})]^d$, 
for $\gamma' = \min(\gamma,\eta)/2$. 
Furthermore, 
for any (deterministic) initial condition $p_0=(p_{0,i})_{i \in \ES} \in {\mathcal S}_{d-1}$ with positive entries, the mean field game  has a unique solution and it is given by
\begin{align}
&dP_{t}^i \nonumber
=  \sum_{j \in \ES} \biggl[ P_t^j \Bigl( \varphi\bigl(P_t^i\bigr)
+ \bigl( U^j - U^i \bigr)_{+}(t,P_t) \Bigr)
-
 P_t^i \Bigl( \varphi\bigl(P_t^j\bigr)
+ \bigl( U^i - U^j\bigr)_{+}(t,P_t) \Bigr)
\biggr] dt \nonumber
\\
&\hspace{30pt} 
+ \frac{\sqrt{\varepsilon}}{\sqrt{2}}
\sum_{j \in \ES} \sqrt{P_t^iP_t^j}
d \bigl[ W_{t}^{i,j} - W_{t}^{j,i} \bigr],
\label{eq:SDE:p}
\end{align}
with $p_{0}$ as initial condition.
\end{prop}
Observe that when $\varepsilon=0$, we reduce to the finite state MFG without common noise analyzed in \cite{bay-coh2019,cec-pel2019}.
Also, in the sequel, it is implicitly understood that
\textbf{the assumptions of Proposition 
\ref{prop:previous}
are in force}. In particular, $f$ and $g$ have the same regularity as therein, for some $\gamma \in (0,1)$. 
Also, the parameter 
 $\kappa$ in Definition 
\ref{eq:varphi}
of $\varphi$ is indeed assumed to be large enough. Depending on our needs, we may even require the value of $\kappa$ to be larger in some of our 
statements.

Now, we are ready to present the $N$-player game and the main results.

\subsection{Finite player version}
\label{subse:2:2}
We now introduce the 
finite player analogue of the mean field game.
Not only should we assign a time-dependent state with each player in the finite game 
but also a weight. Throughout, we use the following generic notations: $N$ is the number of players in the finite game; 
for each index $l \in \EN:=\{1,\cdots,N\}$ and at each time 
$t \in [0,T]$, $X_{t}^l$ denotes the state of player $l$ at time $t$
and $Y_{t}^l$ denotes its \textit{weight}, which is a non-negative rational number. 
The tuple $(X_{t}^1,\cdots,X_{t}^N)$ is written with a boldface letter ${\boldsymbol X}_{t}$ and similarly 
for the weights.
The weighted empirical measure is denoted by 
$$\mu^N_t := \mu^N_{{\bm X}_{t},{\bm Y}_{t}}= 
\frac1N \sum_{l=1}^N Y^l_t \delta_{X^l_t}.$$

\subsubsection{Controlled dynamics}
\label{subsubse:2:2:1}

 The dynamics of $({\boldsymbol X}_{t},{\boldsymbol Y}_{t})_{0 \le t \le T}$ are constructed on the same probability space 
$(\Omega,{\mathcal F},{\mathbb P})$ as before. Consistently with the set-up used for the formulation of the mean field game, they are subjected to a tuple of controls, formalized in the form of a collection of transition rates chosen by each player. 
We may indeed denote by ${\mathbb A}$ the collection of tuples $(\alpha^i[j])_{i,j \in \ES: i \not =j}$ with coordinates in ${\mathbb R}_{+}$. 
A control is then a function ${\boldsymbol \alpha} : [0,T] \times \ES^N \times {\mathbb Q}_{+}^N \rightarrow 
{\mathbb A}^N$ which maps $(t,{\boldsymbol x},{\boldsymbol y}) \in [0,T] \times \ES^N 
\times {\mathbb Q}_{+}^N$ onto ${\boldsymbol \alpha}(t,{\boldsymbol x},{\boldsymbol y}) 
= ((\alpha^l(t,(i,\blx^{-l}),\bly)[j])_{i,j \in \ES : i \not =j})_{l\in\NN}\in {\mathbb A}^N$. 
In other words, $\alpha^l :  [0,T] \times \ES^{N-1} \times {\mathbb Q}_{+}^N \ni (t,\blx^{-l},\bly) \mapsto
(\alpha^l(t,(i,\blx^{-l}),\bly)[j])_{i,j \in \ES : i \not = j}
\in  
{\mathbb A}$ is the feedback function chosen by player $l$; for sure, 
another (but equivalent) way to do is to consider $\alpha^l$ as the application $\alpha^l : 
[0,T] \times \ES^{N} \times {\mathbb Q}_{+}^N \ni (t,\blx,\bly) \mapsto
(\alpha^l(t,\blx,\bly)[j])_{j \in \ES : j \not = x^l}$
which no longer takes its values in ${\mathbb A}$ but in $({\mathbb R}_{+})^{d-1}$. We mostly adopt the latter point of view in the sequel and hence reserve the notation $\alpha^l$ for the $({\mathbb R}_{+})^{d-1}$-valued mapping. For simplicity, we assume that the feedback functions are bounded and measurable.

For a given control ${\boldsymbol \alpha}$, and for the same function $\varphi$ and the same viscosity parameter $\varepsilon \in (0,1]$ as in the previous paragraph, the population evolves according to a continuous time Markov chain with transition rates  given by
\begin{equation}
\label{eq:markov:transition:1}
\begin{split}
&\P\left(\left.X^l_{t+h}=j,\bm{X}^{-l}_{t+h}=\bm{x}^{-l},\bm{Y}_{t+h}=\bm{y}\right| \bm{X}_t=\bm{x}, \bm{Y}_t=\bm{y}\right)
\\
&= \Bigl({\varphi(\mu^N_{\bm{x},\bm{y}}[j])}+\alpha^l(t,\bm{x},\bm{y})[j]\Bigr) h+o(h),
\end{split}
\end{equation}
whenever $j \not = x^l$, and

\begin{equation}
\label{eq:markov:transition:2}
\begin{split}
&\P\biggl(\bm{X}_{t+h}=\bm{x},Y^l_{t+h}=y^l\frac{k^{x^l}}{N\mu^N_{\bm{x},\bm{y}}[x^l]} \  \forall l\in \NN
\, 
\bigg| \, \bm{X}_t=\bm{x}, \bm{Y}_t=\bm{y} \biggr)
\\
&= \varepsilon
{\mathcal M}_{N,\mu^N_{\blx,\bly}}({\boldsymbol k}) N h+o(h),
\end{split}
\end{equation}
where 
${\boldsymbol k}$ is a $d$-tuple of integers $(k^{1},\cdots,k^{d})$ with $k^{1}+\cdots+k^{d}=N$.
In fact, the above definition requires some care: $(i)$ The ratio $k_{x^l}/(N\mu^N_{\bm{x},\bm{y}}[x^l])$ is treated as 1 if 
$\mu^N_{\bm{x},\bm{y}}[x^l]=0$; 
so, on the left-hand side, we 
should write the value of $Y^l_{t+h}$ in the form 
$Y^l_{t+h}=y^l k_{x^l}/(N\mu^N_{\bm{x},\bm{y}}[x^l]) \one_{\{\mu^N_{\bm{x},\bm{y}}[x^l]\neq 0\}}$; (ii) The multinomial distribution on the right-hand side is well-defined if 
$\mu^N_{{\boldsymbol x},{\boldsymbol y}}$ is a probability measure on $\ES$, meaning that 
$y^1 + \cdots + y^N =N$.

In the sequel, we thus always assume that the Markov chain is initialized from a point $({\boldsymbol x},{\boldsymbol y})$
such that $y^1 + \cdots + y^N=N$, namely $\bly \in \setQ$. In order to guarantee that the Markov chain is well-defined, it is then needed to check that   
the latter condition is preserved by the dynamics, meaning that, at any time $t >0$, 
$Y_{t}^1 + \cdots + Y_{t}^N=N$. To do so, it suffices to check that, for any $(\bm{x},\bm{y}) \in \ES^N \times 
\setQ$ and for any ${\boldsymbol k} \in {\mathbb N}^d$ with 
$k^{1}+\cdots+k^d=N$, it holds that 
\begin{equation}
\label{eq:sum:proba}
\sum_{l \in \NN} y^l \frac{k^{x^l}}{N\mu^N_{\bm{x},\bm{y}}[x^l]} = N.
\end{equation}
The above is well-checked. Indeed, by 
\eqref{eq:muNi},
\begin{equation*}
\begin{split}
\sum_{l \in \NN} y^l \frac{k^{x^l}}{N\mu^N_{\bm{x},\bm{y}}[x^l]}
= \sum_{i \in \ES} \sum_{l \in \EN} \one_{\{x^l=i\}} y^l \frac{k^{i}}{N\mu^N_{\bm{x},\bm{y}}[i]}
= \sum_{i \in \ES} k^{i}=N. 
\end{split}
\end{equation*}
Another important point related to the dynamics of $\bm{Y}$ is that 
$0$ is an absorption point for each coordinate, meaning that $Y^l$ remains in $0$ once it has reached it. 

Last but not least, it is implicitly understood that there are no other possible jumps for the whole system than those described in 
\eqref{eq:markov:transition:1}
and
 \eqref{eq:markov:transition:2}, meaning that 
 \begin{equation}
\begin{split}
&\P\Bigl( \bm{X}_{t+h}=\bm{x},\bm{Y}_{t+h}=\bm{y}\big| \bm{X}_t=\bm{x}, \bm{Y}_t=\bm{y}\Bigr)
\\
&= 1 - \sum_{l=1}^N \sum_{j \not = x^l} \Bigl({\varphi(\mu^N_{\bm{x},\bm{y}}[j])}+\alpha^l(t,\bm{x},\bm{y})[j]\Bigr) h - \varepsilon Nh +o(h). 
\end{split}
\end{equation}
Below, we make an intense use of the generator of the process $(\bm{X},\bm{Y})$. It is given by
\[ 
\begin{split}
\mathcal{L}^N_t v(\bm{x},\bm{y})&= 
\sum_{l \in \NN} \sum_{j \in \ES} \left({\varphi(\mu^N_{\bm{x},\bm{y}}[j])}+\alpha^l(t,\bm{x},\bm{y})[j]\right)\bigl[v\bigl((j,\bm{x}^{-l}),\bm{y}\bigr)-v(\bm{x},\bm{y})\bigr]
\\
&+ \varepsilon N \sum_{k\in \NN^d}
{\mathcal M}_{N,\mu^N_{\blx,\bly}}({\boldsymbol k})
\biggl[ v\biggl(\bm{x}, y^1\frac{k_{x^1}}{N\mu^N_{\bm{x},\bm{y}}[x^1]},\dots, y^N\frac{k_{x^N}}{N\mu^N_{\bm{x},\bm{y}}[x^N]}\biggr)
 -v(\bm{x},\bm{y})
\biggr],
\end{split}
\]
with the same convention as before for the ratios in the second line whenever one of the denominators cancels. 
Using our notations, this can be written as 
\be 
\label{gen}
\begin{split}
\mathcal{L}^N_t v(\bm{x},\bm{y})&= 
\sum_{l=1}^N  \left( {\varphi(\mu^N_{\bm{x},\bm{y}}[\bullet])}+\alpha^l(t,\bm{x},\bm{y})[\bullet]\right)\cdot \Delta^l v(\bm{x},\bm{y})[\bullet] \\
&+ \varepsilon N \EE 
\biggl[ v\biggl(\bm{x}, y^1\frac{S_{\mu^N_{\bm{x},\bm{y}}}[x^1]}{N\mu^N_{\bm{x},\bm{y}}[x^1]},\dots, y^N\frac{S_{\mu^N_{\bm{x},\bm{y}}}[x^N]}{N\mu^N_{\bm{x},\bm{y}}[x^N]}\biggr)
 -v(\bm{x},\bm{y})
\biggr],
\end{split}
\ee
where 
the inner product in the first line is in ${\mathbb R}^d$ 
and where we recall that 
$\Delta^l v(\bm{x},\bm{y})[\bullet] \in\R^d$ is defined by 
$\Delta^l v(\bm{x},\bm{y}) [j] = v((j,\bm{x}^{-l}),\bm{y})-v(\bm{x},\bm{y})$.

\subsubsection{Cost functional} 
With the same coefficients $f,g : \ES \times {\mathcal S}_{d-1} \rightarrow {\mathbb R}$ as before, we can now assign a cost with each player. 
For each $l \in \EN$, player $l$ aims at minimizing the cost\footnote{Notice that, in the formula below, we should write 
$\alpha^l(t,\bm{X}_t,\bm{Y}_t)[\bullet]$; for convenience, we remove the bullet $\bullet$.}
\be 
\label{cost}
J^l(\bm{\alpha})= \E \biggl[
\int_0^T Y^l_t\left({L}(X^l_t,\alpha^l(t,\bm{X}_t,\bm{Y}_t))+f(X^l_t,\mu^N_t)\right)dt +Y^l_T g(X^l_T,\mu^N_T)
\biggr],
\ee
where $\bm{\alpha} =(\alpha^1,\dots,\alpha^N)$ is the tuple of controls chosen by the $N$ players. As usual in game theory, 
the cost to player $l$ hence implicitly depends on the controls chosen by the others. Above, the function $L$ denotes the same Lagrangian as in the previous subsection, namely $L(i,\alpha)=\frac12\sum_{j\neq i} |\alpha(j)|^2$,
for $i \in \ES$ and $\alpha =(\alpha(j))_{j \in \ES : j \not =i} \in ({\mathbb R}_{+})^{d-1}$. Hence, we may rewrite
the cost as
\begin{equation*}
\begin{split}
J^l(\bm{\alpha})&= \sum_{i \in \ES} \E \biggl[
\int_0^T Y^l_t \one_{\{X^l_t=i\}}\biggl(\frac12\sum_{j\neq i} \bigl\vert \alpha^l\bigl(t,(i,\bm{X}^{-l}_t),\bm{Y}_t\bigr)[j] \bigr\vert^2 +f(i,\mu^N_t)\biggr)dt
 +Y^l_T \one_{\{X^l_T=i\}}g(i,\mu^N_T)
\biggr].
\end{split}
\end{equation*}
Of course, the occurrence of the weights in the definition of the cost functional (which is one of the unusual feature of our model) is reminiscent of the formula used in the mean field limit. Obviously, this is our objective to make the connection rigorous. 
Accordingly, we can introduce the Hamiltonian
\be
\label{eq:def:Hamiltonian}
H(i,u)= -\frac12\sum_{j\neq i} (u^i-u^j)^2_+ =\inf_{\alpha\in (\R_+)^d} \bigg\{\sum_j\alpha[j] (u^j-u^i) +  \frac12\sum_{j\neq i} |\alpha[j]|^2\bigg\},
\ee
whose argmin is $a^*(i,u)[j]=(u^i-u^j)_+$. (Notice that, in the above infimum, $\alpha$ is $d$-dimensional, whilst controls have been regarded as being $(d-1)$-dimensional so far. Obviously, this is for notational convenience only.)


\subsubsection{More about our choice of common noise}
The reader might object that, despite the terminology used throughout the paper, 
our model does not coincide with the standard Wright--Fisher 
one. Although our model is indeed different, the common noise shares some similarities, hence 
our choice to call it ``Wright--Fischer''. 
To wit, it is easy to see from 
\eqref{eq:markov:transition:2}
that, at any random time $t$ when the common noise rings (or jumps), 
\begin{equation*}
\mu^N_{t}[i] = \frac1N \sum_{l \in \NN} Y_{t}^{l} \one_{\{X_{t-}=i\}}
=  \frac1N \sum_{l \in \NN} Y_{t-}^{l} \frac{S_{\mu_{t-}^N}[i]}{N \mu_{t-}^N[i]} \one_{\{X_{t-}=i\}}
=  \frac{S_{\mu_{t-}^N}[i]}{N}, \quad i \in \ES,
\end{equation*}
where, conditional on the observations up to $t-$,  
$S_{\mu_{t-}^N}$ is a multinomial distribution of parameters $N$ and $(\mu_{t-}^N[e])_{e \in \ES}$. 
The way the weighted empirical distribution is hence updated is thus similar to the way the empirical distribution is updated in the standard 
Wright--Fisher model. 

To make the comparison even stronger, it is interesting to observe that there is in fact a more direct way to introduce a Wright--Fisher common noise in the game. Indeed, consistently with the very definition of the Wright--Fisher model, we could require that, at any jump time of the common noise, all the players in the game resample their own state (or location in $\ES$), independently of the others and of the past, according to the uniform empirical distribution of the system. Here, what we mean by ``uniform empirical distribution'' of the system is 
the standard empirical distribution, obtained by assigning the weight $1/N$ to each player. As a result, this model would not feature any additional
weight process $({\boldsymbol Y}^l)_{l \in \NN}$. Even though it is very appealing, this approach has however a very strong drawback, which makes it useless for our own purpose: Because of the resampling
(which would occur very often under the same intensity as in \eqref{eq:markov:transition:2}), the empirical measure
would be strongly attractive; in turn, the latter would preclude any interesting deviating phenomenon.
This is in contrast to our model: Since the common noise acts on the weights, players may really deviate 
from the empirical measure of the whole population. This is the rationale for assigning two attributes to each player. This is also the main 
conceptual innovation of our model.

\subsubsection{Application to models with heterogeneous influences}
\label{subse:apps:influences}
{ As we already pointed out, one of the main feature of the finite game that we have just introduced is that the weights 
$(Y^1_t,\cdots,Y^N_t)_{t \geq 0}$ of the 
players in the flow of empirical measures $(\mu_t^N)_{t \geq 0}$ may be different. 
From the practical point of view, the weight $Y^l_t$ that is carried out by player $l$ at time $t$ may be regarded as the instantaneous influence of 
player $l$ onto the whole collectivity. In this sense, this type of game may be used in order to describe models with heterogeneous influences. 
Notice that this interpretation is consistent with the one that we gave 
earlier in the text. 
Following 
\cite{BCCD-arxiv}, we indeed explained that the weight to player $l$ could be regarded as
describing the own perception 
of the world  
by player $l$. 
What we are saying here is that the way player $l$ perceives the world can be determined by the way player $l$ is in fact perceived by the others: in other words, the collectivity is acting as a mirror.

Generally speaking, influences here evolve with time 
according to the transitions 
\eqref{eq:markov:transition:2}, the very interpretation 
of which is as follows. 
Intuitively, influences within the population are updated from time to time, when the bell of the common noise rings. 
Equivalently, common noise defines random times at which polls (or elections) are made  in order to determine the current influences of 
the players. 
In this regard,
\eqref{eq:markov:transition:2}
is reminiscent of models in population genetics:
When the common noise rings, every player in the population chooses at random 
a feature in $\ES$ 
according to the current weighted empirical measure. Obviously, features are sampled independently. 
The new influence of player $l$ is then recomputed 
by multiplying the earlier one by the ratio equal to `the number of outcomes for the feature carried by $l$ divided by the expected number of outcomes for this feature', 
with the later denominator in the ratio being also understood as the total influence of the given feature.
In the mean, the influence keeps constant, but, obviously, some fluctuations may occur.
We then recover the aforementioned notion of `perception' by returning to 
the formula 
\eqref{cost}
for the cost. Intuitively, an action may not have the 
same impact 
whether a player is very or poorly influent. For instance, 
a strong influencer may gain or loose a lot of followers after a good or bad decision, whence the presence of the weight 
$Y^l$ as a density in \eqref{cost}.

In order to clarify the exposition, we provide two economic and social phenomena which fit this concept. Typically, we may consider a network with $N$ agents, each of them aiming at selling a product, which can be of one out of $d$ types. 
Consistently with the previous description,
 any player has her own influence/importance in the market: The position $X^l_t\in\ES$ of player $l$ is thus understood as the type she sells, while the weight $Y^l_t$ is seen as her influence. Moreover, there are external random shocks which, at random times, resample the influences of players according to 
\eqref{eq:markov:transition:2}. Any player then chooses her control, which is the rate at which she changes her type, in order to maximize a reward given by (the opposite of) \eqref{cost}. Therein, a quadratic cost has to be paid in order to trade another type of product; moreover, the rewards $-f$ and $-g$ include the weighted empirical measure $\mu^N_t = \frac1N \sum_{l=1}^N Y^l_t \delta_{X^l_t}$. The component 
$\mu^N_t [i] = \frac1N \sum_{l=1}^N Y^l_t \mathbbm{1}_{\{X^l_t =i \}}$ is then understood as the total influence of the type $i$.
 Hence, $f$ and $g$ can be taken e.g.  monotone if it is more advantageous to sell products with a small influence, while they can be antimonotone if it is
 better to sell products with a large influence.
 The two rewards $-f$ and $-g$ in 
 \eqref{cost}
are multiplied by 
 $Y^l$
in order to account for the fact that the reputation of 
 a trade mark 
 may have a direct impact on the selling price. 
 Similarly, a highly reputed trade mark may pay a higher price for 
 switching from one type of product to another one, which may justify why the 
 quadratic cost in
  \eqref{cost}
is also multiplied by 
 $Y^l$. 

As our second social phenomenon, let us focus on a social voter model where $N$ people can choose one out of two opposite opinions. Similarly to \cite{cecdaifispel}, we assume that $X^l_t = \pm 1$. In addition, any player has her own influence $Y^l_t$ which can be thought, for instance,  as the number of followers in a social network. Again, external random shocks might be thought as elections that shuffle people's influence.  People are rational and want to minimize the cost 
\[
J^l(\bm{\alpha})= \E \biggl[
\int_0^T Y^l_t  \frac{|\alpha^l(t,\bm{X}_t,\bm{Y}_t)[-X^l_t]|^2}{2} dt -Y^l_T X^l_T \mathcal{M}(\mu^N_T) \bigg], 
\]
where $\alpha^l(t,\bm{X}_t,\bm{Y}_t)[-X^l_t]$ is the switching rate of player $l$, from $X^l_t$ to $-X^l_t$, and $\mathcal{M}(\mu^N_T)$ denotes the mean of the measure, i.e. $\mathcal{M}(\mu^N_T) = \mu^N_T[1]-\mu^N_T[-1] = 
\frac1N \sum_{l=1}^N Y^l_t \mathbbm{1}_{\{X^l_t =1 \}} 
- \frac1N \sum_{l=1}^N Y^l_t \mathbbm{1}_{\{X^l_t =-1 \}}$. Therefore the terminal cost is antimonotone and, morally, if the total influence of opinion $1$ is larger than the total influence of $-1$, then the mean is positive. In turn, any player wants to be in 1 as she has to minimize the cost. Still, there is a cost to pay to change one's opinion, which is larger when the influence is large. Thus any player wishes to imitate  the majority, but, differently from \cite{cecdaifispel}, this is here a weighted majority which takes into account the influence of people. This is a typical model with \textit{heterogeneous influences}. 

\subsection{Existence of Nash equilibria}

We are now able to state the \emph{Nash system}, which 
is an equation for the equilibrium values of the game; equivalently, 
the Nash system 
gives the cost to each player when all of them play \textit{the} Nash equilibrium (we prove below that the latter is indeed unique in a relevant sense).
It is a system of $N$ functions indexed by $l\in\NN$, $\bm{x} \in \ES^N$ and $\bm{y} \in \setQ$, which formally writes
(with $(v^{N,l}(t,\bm{x},\bm{y}))_{l \in \EN}$ as unknown)
\be 
\label{nash}
\begin{split}
&\frac{d}{dt} v^{N,l}
+
\sum_{m \in \EN } 
\varphi(\mu^N_{\bm{x},\bm{y}}[\bullet])
\cdot \Delta^m v^{N,l}[\bullet] 
+\sum_{m\neq l} a^*\Big(x^m, \frac{1}{y^m} v^{N,m}_{\bullet}\Big)\cdot \Delta^m v^{N,l}[\bullet] 
\\
& \qquad 
+y^l H\Big(x^l, \frac{1}{y^l}v^{N,l}_\bullet \Big) + y^lf(x^l, \mu^N_{\bm{x},\bm{y}})
\\
& \qquad +\varepsilon N \EE 
\biggl[ v^{N,l}\biggl(t,\bm{x}, y^1\frac{S_{\mu^N_{\bm{x},\bm{y}}}[x^1]}{N\mu^N_{\bm{x},\bm{y}}[x^1]},\dots, y^N\frac{S_{\mu^N_{\bm{x},\bm{y}}}[x^N]}{N\mu^N_{\bm{x},\bm{y}}[x^N]}\biggr)
 -v^{N,l}(t,\bm{x},\bm{y})
\biggr]=0,
\\
& v^{N,l}(T,\bm{x},\bm{y})= y^l g(x^l, \mu^N_{\bm{x},\bm{y}}),
\end{split}
\ee
where $v^{N,l}_\bullet$ denotes the vector in $\mathbb{R}^d$ given by
$(v^{N,l} (t, \bm{x},\bm{y})[j]= v^{N,l} (t, (j, \bm{x}^{-l}),\bm{y}))_{j \in \ES}$.

In a word, the Nash system here reads as a countable system of ordinary differential equations, which makes it a bit more subtle than the Nash system that arises in the analysis of $N$-player games over a finite state space (as it is the case in the study of the convergence problem for finite state mean field games without common noise, see for instance 
\cite{bay-coh2019,cec-pel2019}). 
Also (and this is another difficulty, very specific to our setting), none of the above equations makes sense whenever any of the coordinates of 
${\boldsymbol y}$ vanishes because of the ratio $v^{N,m}_{\bullet}/y^m$ in 
$a^*$ and of the ratio $v^{N,l}_{\bullet}/y^l$ in the Hamiltonian.

Our first main result in this regard comes as a verification argument, the proof of which is postponed to Appendix \ref{subse:proof:verif}:

\begin{prop}[Verification Argument]
\label{prop:verification:Nash}
Consider 
 a collection
  of measurable functions
  $(a^l : [0,T] \times \ES^N \times \setQ
\ni (t,\blx,\bly)
 \mapsto a^l(t,(i,\blx,\bly)[j])_{j \in \ES : j \not =x^l} \in ({\mathbb R}_{+})^{d-1})_{l \in \EN}$ taking values in a compact domain.
Assume that $(v^{N,l})_{l \in \EN}$ solves the Nash system \eqref{nash} with the special features that 
\begin{enumerate}
\item 
$a^*(x^l,v^{N,l}_{\bullet}(t,\bm{x},\bm{y})/y^{l})$ is understood as $a^l(t,\bm{x},\bm{y})$
whenever $y^l=0$;
\item the Hamiltonian $y^l H(x^l,v^{N,l}_{\bullet}/y^l)$ is understood as $0$
whenever $y^l=0$.
\end{enumerate}
 Then, the feedback strategy vector $\bm{\alpha}^*=(\alpha^{*,1},\dots,\alpha^{*,N})$ given by 
\be 
\label{optcon}
\begin{split}
\alpha^{*,l}(t,\bm{x},\bm{y})[j]=a^*\bigg(x^l, \frac{1}{y^l} v^{N,l}_\bullet(t,\bm{x},\bm{y})\bigg)[j] :&= \frac{1}{y^l}\Bigl(v^{N,l}\big(t,\bm{x},\bm{y})- v^{N,l}\bigl(t,(j,\bm{x}^{-l}),\bm{y}\bigr)\Big)_+
\\
&= \frac{1}{y^l}\Bigl(- \Delta^l v^{N,l}\big(t,\bm{x},\bm{y})[j]\Big)_+,
\end{split}
\ee 
for $l \in \EN$ such that $j \not = x^l$ and $y^l >0$,
and by 
\be 
\label{optcon:2}
\alpha^{*,l}(t,\bm{x},\bm{y})[j]=a^l(t,\bm{x},\bm{y})[j],
\ee
for $l \in \EN$
such that $j \not = x^l$ and $y^l =0$, 
defines a Nash equilibrium in Markov feedback form for the $N$-player game. Moreover, the functions $(v^{N,l})_{l \in \EN}$ are the values of the equilibrium, i.e. 
\be 
v^{N,l}(t,\bm{x},\bm{y}) = J^l\bigl(t,\bm{x},\bm{y}, \bm{\alpha^{*}} \bigr)
=\inf_{\beta} J^l\bigl(t,\bm{x},\bm{y}, \beta,\bm{\alpha}^{*,-l}\bigr),
\ee
where $J^l(t,\bm{x},\bm{y},\bm{\alpha})$ denotes the cost when the process $(\bm{X},\bm{Y})$ starts at time $t$ with $(\bm{X}_t,\bm{Y}_t)= (\bm{x},\bm{y})$.
\end{prop}

The indetermination when $y^l=0$ is well-understood: In that case, the
coordinate $Y^l$ remains stuck in $0$ and the running and terminal costs become zero whatever the choice of the 
strategy.
In the sequel, we circumvent part of the indetermination by 
restricting uniqueness 
of $a^l$ in 
\eqref{optcon:2}
to triples $(t,\blx,\bly)$ such that $y^l>0$: 

\begin{prop}[Existence and uniqueness of equilibria]
\label{prop:existence:!:continuous:eq}

The Nash system \eqref{nash} has a solution $(v^{N,l})_{l \in \EN}$ such that, for each $l \in \NN$, 
the function
$(t,\blx,\bly) \mapsto 
v^{N,l}(t,\blx,\bly)/y^l$, which is a priori defined 
on the set 
$\{ (t,\blx,\bly) \in [0,T] \times \ES^N \times \setQ : y^l >0\}$, extends to 
$[0,T] \times \ES^N \times \setQ$ into a function 
$[0,T] \times \ES^N \times \setQ \ni (t,\blx,\bly) \mapsto w^{N,l}(t,\blx,\bly)$
that is bounded by $T \| f\|_{\infty} + \| g \|_{\infty}$. It satisfies, 
for any 
$l,m \in \NN$, any $\bly \in \Y$ such that $y^l >0$ and 
$y^n=0$ and any $(t,\blx) \in [0,T] \times \ES^N$,
\begin{equation}
\label{eq:insensitivity}
\Delta^n w^l(t,\bm{x},\bm{y}) [j] = 0, \quad j \in \ES. 
\end{equation}

Accordingly, 
${\boldsymbol \alpha}^*$, as given by \eqref{optcon}, extends to 
tuples $\bly$ satisfying $y^l=0$ for some $l \in \EN$, replacing therein 
$v^{N,l}(t,\blx,\bly)/y^l$ by $w^{N,l}(t,\blx,\bly)$, 
and hence defines an equilibrium in Markov feedback form. 
Any other equilibrium
in Markov feedback form 
 $(a^l : [0,T] \times \ES^N \times \setQ
\ni (t,\blx,\bly)
 \mapsto a^l(t,(i,\blx,\bly)[j])_{j \in \ES : j \not =x^l} \in ({\mathbb R}_{+})^{d-1})_{l \in \EN}$
 satisfies
\begin{equation}
\label{eq:uniqueness:nash}
\begin{split}
a^l(t,(i,\blx,\bly)[j] = \frac{1}{y^l}\Bigl(- \Delta^l v^{N,l}\big(t,\bm{x},\bm{y})[j]\Big)_+,
\end{split}
\end{equation} 
for $l \in \EN$ such that $j \not = x^l$ and $y^l >0$. 
\end{prop}

Interestingly, 
property 
\eqref{eq:insensitivity}
should read as an insensitivity property: It says that, whenever 
player $l$ has a non-zero mass and implements 
the equilibrium strategy given by 
\eqref{optcon}, it is insensitive to the peculiar state of 
any player $n$ with a zero mass $y^n$ at the same time. 
As explained right below, this feature 
implies a form of uniqueness of the hence constructed Nash equilibrium.

Indeed, although 
Proposition 
\ref{prop:existence:!:continuous:eq}
does not give uniqueness of an equilibrium in Markov feedback form 
(because of the indetermination 
raised by Proposition \ref{prop:verification:Nash} at points 
at which the mass of one player vanishes), 
\eqref{eq:uniqueness:nash}
shows that the equilibrium feedback function chosen by player 
$l$ at any point $a^l(t,\blx,\bly)$ 
with $y^l>0$ is in fact uniquely determined. Accordingly, the equilibrium state process $(X^l,Y^l)$ of player
$l$ is uniquely defined up to the first time when 
its mass $Y^l$ hits $0$. Once the mass process $Y^l$
has touched $0$, it remains in $0$; the position $X^l$ may still vary according to the feedback 
function 
$a^l$ but this has no influence on the remaining expected cost
of player $l$ itself (because the running and terminal costs in 
\eqref{cost} are multiplied by the mass)
nor on the  
remaining expected cost
of the other players. 
The latter feature is a bit subtle and is a consequence of the following two facts: 
\begin{enumerate}
\item
The first one is that, at any time $t$, 
the player $l$ has no influence
on the running and terminal costs to any player $m \not =l$
  if its mass at time $t$ is zero; indeed, 
its contribution to 
the empirical measure $\mu^N_{t}$ is then null;
\item 
The second fact is that, by the insensitivity property
\eqref{eq:insensitivity}, 
the player $l$ has no influence
on the feedback of any player $m \not =l$
 at any time when its mass is zero.
\end{enumerate}

The proofs of 
Propositions 
\ref{prop:verification:Nash}
and
\ref{prop:existence:!:continuous:eq}
are partially postponed to 
Appendix 
\ref{subse:heuristic:master equation}, since 
only a part of those former two results 
is needed in the core of our analysis. 
In fact, what we really need is the existence of 
the functions $(w^{N,l})_{l \in \NN}$ quoted in the statement 
 of Proposition 
\ref{prop:existence:!:continuous:eq}
and, most of all, the fact that those functions can be bounded independently 
of $N$. 
The latter facts are addressed in 
Section \ref{sec:3}, while the remaining claims are just checked in the 
Appendix 
\ref{subse:heuristic:master equation}.

\subsection{Convergence results} 
\label{sec:convres}
Consider the 
family of solutions $((v^{N,l})_{l \in \NN})_{N \geq 1}$ 
given by  
Proposition 
\ref{prop:existence:!:continuous:eq}.  Accordingly, 
set, 
for any $N\in\mathbb{N} \setminus \{0\}$ and $l \in\NN$, the real-valued functions $\v^{N,l},\u^{N,l}$ on $[0,T]\times\ES^N \times\Y$ by
\begin{align}\label{wz}
\v^{N,l}(t,\blx,\bly)&:=\frac{1}{y^l}v^{N,l}(t,\blx,\bly),\qquad
\u^{N,l}(t,\blx,\bly):=
U^{x^l}(t,\mu^N_{\blx,\bly}),
\end{align}
where $U\in[\mathcal{C}^{1+\gamma'/2,2+\gamma'}_{\rm WF}([0,T] \times \mathcal{S}_{d-1})]^d$, for some 
$\gamma' \in (0,1)$ (depending on $\gamma$, as mentioned in the statement of Proposition 
\ref{prop:previous}), is the classical solution to the master equation \eqref{eq:master:equation:intrinsic}
(recalling that $\kappa$ is implicitly required to be large enough, which is in any case 
explicitly recalled in our main statements below). 

Our first main result provides a bound for the difference between $w^{N,l}$ and $z^{N,l}$:
\begin{thm}[Distance between the $N$-player and mean field value functions]\label{thm:convergence_value}
Under the assumption
of Proposition 
\ref{prop:previous},  
we can find a constant $\overline \kappa_{0}$, 
only depending on $\varepsilon$, $T$,
$\| f \|_{\infty}$
and $\|g \|_{\infty}$,  such that, for
any 
$\kappa \geq \overline \kappa_{0}$, 
there exist an exponent $\chi>0$, only depending on $\gamma$, and a constant $C$, 
depending on $\delta$, $\kappa$, $T$, $\|f\|_{\infty}$, $\|g\|_{\infty}$ and $d$, 
 such that, for any $N\in\mathbb{N} \setminus \{0\}$, $l \in\NN$, ${\bm x} \in\ES^N$, ${\bm y} \in{\mathbb Y}$,  
one has
\begin{align}\label{bound1}
 \big\vert (z^{N,l}-w^{N,l}) \bigr\vert(t,\bm{x},\bm{y}) 
 \leq C  N^{-\chi}
 \biggl( 
 \prod_{i \in \ES}
\frac1{N^{-\epsilon}+\mu_{\bm{x},\bm{y}}^N[i]}
\biggr)^{1/(2d)} 
\biggl( \frac1N \sum_{m \in \EN} \vert y^m  \vert^{\ell} \biggr)^{1/2},
\end{align}
with $\epsilon=1/8$ and $\ell=3$. \end{thm}

\begin{rem}
	\label{rem2}
\begin{enumerate}

\item We will be using the the bound in \eqref{bound1}  along a sequence $(\bm{x}^N,\bm{y}^N)_{N \geq 1}$ 
(with each $(\blx^N,\bly^N)$ in $\ES^N \times {\mathbb Y}_{N}$,
where we put an additional index $N$ in the notation ${\mathbb Y}$ since the latter obviously depends on $N$) 
that satisfies the following 
\begin{equation}\label{xNyN}
\sup_{N \geq 1} \biggl\{
 \max_{i \in \ES} \frac1{N^{-\epsilon}+\mu^N_{\bm{x}^N,\bm{y}^N}[i]}
 +
 \frac1N 
 \sum_{m \in \NN} \vert y^{N,m} \vert^{\ell}
 \biggr\} < \infty.
\end{equation}
Importantly, the bound in \eqref{bound1} is not uniform in $(\bm{x}, \bm{y})$.  The reason for this is that
the proof 
requires 
  a moment estimate on the weights, but this estimate  explodes as $\mu^N_{\bm{x},\bm{y}}$ approaches the boundary of the simplex; see \eqref{boundY_magenta} below. 
 The fact that singularities may emerge near the boundary can be anticipated from the dynamics \eqref{eq:markov:transition:2} where $1/\mu^N_{\bm{x}, \bm{y}}$ appears. Actually, this observation is also consistent with our previous result Theorem \ref{prop:previous}, which establishes uniqueness of mean field game solutions only for initial distributions in the interior of the simplex, and requires the optimal process to be sufficiently far away from the boundary with high probability.
Instead, without common noise (i.e. $\varepsilon=0$), the convergence rate is uniform -- if the master equation possesses a smooth solution-- because the weights are then constantly 1.

\item The value of $\chi$ could be made explicit in terms of $\gamma$. 
For sure, this would require an additional effort in the proof to track the dependence of 
$\chi$ upon $\gamma$. In fact, we have felt easier not to 
address this question since the value of $\gamma$, as given by \cite[Theorem 2.10]{BCCD-arxiv}, is itself implicit; 
so, the resulting interest for having an explicit formula for $\chi$ in terms of $\gamma$ seems rather limited. 

In the same spirit, it is worth noticing that $\overline{\kappa}_{0}$ in
the statement of 
Theorem \ref{thm:convergence_value}
may differ from 
$\kappa_{0}$ in 
the statement 
of Proposition 
\ref{prop:previous}; in other words, we are not able to work with the same\footnote{In fact, there are subtle questions here since the parameter $\varepsilon$ in $\overline{\kappa}_{0}$ only comes through $\kappa_{0}$ itself and does not show up explicitly in our own computations. In other words, the construction of 
$\overline{\kappa}_{0}$ follows
from constraints that are not the same as those used to define $\kappa_{0}$.} $\kappa_{0}$ as in the statement of 
Proposition  
\ref{prop:previous}.
\item The reader will notice that, in the proof below, we pay a heavy price for the fact that the model is not elliptic. 
If the transition rates were assumed to be bounded from below by a positive constant on the whole domain (in our setting, the rates are just lower bounded in the neighborhood of the boundary thanks to the function $\varphi$), then the arguments 
would simplify.  
\end{enumerate}

\end{rem}

Convergence of the empirical measure 
is directly proved by means of a diffusion approximation theorem.
In particular, it is worth noting that, differently from 
\cite{CardaliaguetDelarueLasryLions}, we do not 
address the distance between the equilibrium particle system 
$(\bm{X},\bm{Y})$ and the auxiliary particle system 
obtained by replacing the feedback function $\Delta w^{N,l}$
by $\Delta z^{N,l}$. Indeed, 
if ever we were willing to do so, then we would have 
to invoke a similar diffusion approximation result
but for the convergence of the auxiliary 
particle system; the proof would be exactly the same. We hence feel more straightforward 
to apply such a diffusion approximation argument to the equilibrium particle system. 
Also, it must be stressed that 
a peculiar interest of the auxiliary particle system is to allow for refined 
convergence results for the fluctuations and the deviations of the 
finite Nash equilibrium. For instance, this idea has been developed in 
\cite{bay-coh2019,cec-pel2019,del-lac-ram2019,DelarueLackerRamananLDP}, but it 
requires a sufficiently strong rate of convergence 
for the difference 
$\Delta w^{N,l}-\Delta z^{N,l}$. Here, 
the bound obtained in 
the statement of Theorem 
\ref{thm:convergence_value}
is rather weak and is below the threshold that would be needed for this approach
(even though the value for $\chi$ is not explicit, it is clear from the proof that it is small).

In the statement below, 
we use the notation 
$(\mu_{t}:=\mu_{\bm{X}_{t}^N,\bm{Y}_{t}^N}^N)_{0 \leq t \leq T}$
for the empirical measure of the equilibrium particle system with $N$ players.
\begin{thm}[Convergence of the empirical measure]
\label{thm:convergence}
For the same regime of parameters as in the assumption of 
Theorem \ref{thm:convergence_value}, 
consider, for each $N \geq 1$, an initial condition 
$(\bm{x}^N,\bm{y}^N)=((x^{N,l})_{l \in \NN},(1,\dots,1)_{l \in \NN}) \in \ES^N \times 
\setQ$ that satisfies
\begin{equation}
\label{conv:init}
\lim_{N \rightarrow \infty} \mu^N_{\bm{x}^N,\bm{y}^N}[i]=p_{0}^i>0,
\end{equation}
for all $i \in \ES$ and
for some $p_{0}\in{\mathcal P}(\ES)$. 
Then, the sequence 
$(\mu^N_{t})_{0 \le t \le T}$ (seen as random elements taking values in the Skorokhod
space $\mathcal{D}([0,T],\mathbb{R}^d)$) converges in the weak sense on ${\mathcal D}([0,T];\R^d)$, equipped with 
the $J1$ Skorokhod topology, to the solution $(P_{t})_{0 \le t \le T}$ of 
the SDE \eqref{eq:SDE:p}.
\end{thm}

\section{Equilibria with uniformly bounded feedback functions}
\label{sec:3}
One of the purposes of this section is to 
prove 
the first part of
Proposition 
\ref{prop:existence:!:continuous:eq}, 
namely the fact that we can construct solutions
$(v^{N,l})_{l \in \EN}$
to the Nash system \eqref{nash}
such that 
$(v^{N,l}(t,\blx,\bly)/y^l)_{l \in \EN}$ is bounded, independently of $N$. Actually, we kill here two birds with one stone: Not only we prove the former boundedness property, but we also manage to establish an interpretation of the normalized value functions $(v^{N,l}(t,\blx,\bly)/y^l)_{l \in \EN}$
as the value functions of another game. This is precisely this new interpretation that permits to prove the former uniform bound in $N$. This uniform bound plays a key role in our subsequent analysis of the convergence problem; it is the main ingredient from 
Proposition 
\ref{prop:existence:!:continuous:eq}
that is used in the sequel.

\subsection{Normalized Nash system}
Following
\eqref{wz}, we are willing to study 
\be\label{def_w}
w^{N,l}(t,\bm{x},\bm{y}):= \frac{1}{y^l} v^{N,l}(t,\bm{x},\bm{y}).
\ee
Indeed, dividing \eqref{nash} by $y^l$, we obtain
(pay attention to the fact that the term on the third line
is heavily impacted by the change of variable)
\[ 
\begin{split}
&\frac{d}{dt}w^{N,l}
+
\sum_{m \in \EN } 
\varphi(\mu^N_{\bm{x},\bm{y}}[\bullet])
\cdot \Delta^m w^{N,l}[\bullet] 
+\sum_{m\neq l} a^*\Big(x^m, w^{N,m}_{\bullet}\Big)\cdot \Delta^m w^{N,l} 
\\
&\ + H\Big(x^l, w^{N,l}_\bullet \Big) + f(x^l, \mu^N_{\bm{x},\bm{y}})
\\
&\  + \varepsilon N \EE 
\biggl[ \frac{S_{\mu^N_{\bm{x},\bm{y}}}[x^l]}{N\mu^N_{\bm{x},\bm{y}}[x^l]}
w^{N,l}\biggl(t,\bm{x}, y^1\frac{S_{\mu^N_{\bm{x},\bm{y}}}[x^1]}{N\mu^N_{\bm{x},\bm{y}}[x^1]},\dots, y^N\frac{S_{\mu^N_{\bm{x},\bm{y}}}[x^N]}{N\mu^N_{\bm{x},\bm{y}}[x^N]}\biggr)
 -w^{N,l}(t,\bm{x},\bm{y})
\biggr]=0,
\end{split}
\]
which, since $\EE[{S_{\mu^N_{\bm{x},\bm{y}}}[x^l]}/{(N\mu^N_{\bm{x},\bm{y}}[x^l]})]=1$, can be rewritten as
\begin{align}
&\frac{d}{dt}w^{N,l}
+\sum_{m \in \EN } 
\varphi(\mu^N_{\bm{x},\bm{y}}[\bullet])
\cdot \Delta^m w^{N,l}[\bullet] 
+\sum_{m\neq l} a^*\Big(x^m, w^{N,m}_{\bullet}\Big)\cdot \Delta^m w^{N,l} 
\nonumber
\\
&\ + H\Big(x^l, w^{N,l}_\bullet \Big) + f(x^l, \mu^N_{\bm{x},\bm{y}})
\label{newnash}
\\
&\ +\varepsilon N \EE 
\biggl[ \frac{S_{\mu^N_{\bm{x},\bm{y}}}[x^l]}{N\mu^N_{\bm{x},\bm{y}}[x^l]}
\biggl(
w^{N,l}\biggl(t,\bm{x}, y^1\frac{S_{\mu^N_{\bm{x},\bm{y}}}[x^1]}{N\mu^N_{\bm{x},\bm{y}}[x^1]},\dots, y^N\frac{S_{\mu^N_{\bm{x},\bm{y}}}[x^N]}{N\mu^N_{\bm{x},\bm{y}}[x^N]}\biggr)
 -w^{N,l}(t,\bm{x},\bm{y})
\biggr)\biggr]=0, \nonumber
\end{align}
with the terminal boundary condition
\be
w^{N,l}(T,\bm{x},\bm{y})= g(x^l, \mu^N_{\bm{x},\bm{y}}).
\ee

Our first result guarantees that 
\eqref{newnash} is well-posed. 
\begin{prop}
\label{prop:nash:system:posedness}
The system of equations \eqref{newnash} has a unique solution among all the 
$N$-tuples of 
bounded functions $(w^{N,l})_{l \in \EN}$ of the three variables $t \in [0,T]$, $\bm{x} \in \ES^N$ and $\bm{y} \in \setQ$, namely
\begin{equation*} 
\max_{l=1,\cdots,N} \sup_{t,\bm{x},\bm{y}} \bigl\vert w^{N,l}(t,\bm{x},\bm{y}) \bigr\vert < \infty,
\end{equation*}
the supremum being taken over $t \in [0,T]$, $\bm{x} \in \ES^N$ and $\bm{y} \in \setQ$. 
\end{prop}

\begin{proof}
\textit{Well-posedness.}
We argue by using a standard fixed point argument. Given an input $w$ in the form of 
an $N$-tuple of 
bounded functions $(w^{l})_{l \in \EN}$ of the three variables $t$, $\bm{x}$ and $\bm{y}$ 
as in the statement, we consider the system 
\begin{align}
&\frac{d}{dt} \widetilde w^{l}
+
\sum_{m \in \EN } 
\varphi(\mu^N_{\bm{x},\bm{y}}[\bullet])
\cdot \Delta^m \widetilde w^{N,l}[\bullet]
+\sum_{m\neq l} a^*\Big(x^m, \widetilde w^{m}_{\bullet}\Big)\cdot \Delta^m \widetilde w^{l} + H\Bigl(x^l, \widetilde w^{l}_\bullet \Bigr) + f\bigl(x^l, \mu^N_{\bm{x},\bm{y}}\bigr) \nonumber
\\
&+\varepsilon N \EE 
\biggl[ \frac{S_{\mu^N_{\bm{x},\bm{y}}}[x^l]}{N\mu^N_{\bm{x},\bm{y}}[x^l]}
\biggl(
w^{l}\biggl(t,\bm{x}, y^1\frac{S_{\mu^N_{\bm{x},\bm{y}}}[x^1]}{N\mu^N_{\bm{x},\bm{y}}[x^1]},\dots, y^N\frac{S_{\mu^N_{\bm{x},\bm{y}}}[x^N]}{N\mu^N_{\bm{x},\bm{y}}[x^N]}\biggr)
 -w^{l}(t,\bm{x},\bm{y})
\biggr)\biggr]=0, \label{newnashtilde}
\\
& \widetilde w^{l}(T,\bm{x},\bm{y})= g\bigl(x^l, \mu^N_{\bm{x},\bm{y}}\bigr). \nonumber
\end{align}
Following 
\cite{cec-pel2019}, we know that \eqref{newnashtilde} has a unique bounded solution $(\widetilde{w}^l)_{l\in\NN}$(notice that the presence of ${\boldsymbol y}$ counts for nothing whenever the input $w$ is frozen: We may then solve the above equation ${\boldsymbol y}$ per ${\boldsymbol y}$). 
This creates a mapping $\Phi$ that sends $w=(w^l)_{l \in \EN}$ onto 
$\widetilde w=(\widetilde w^l)_{l \in \EN}$. 
System \eqref{newnashtilde}
can be regarded as the Nash system
of a (quite standard) stochastic game on $\ES$
with $g$ on the last line as terminal cost and with $f$ on the first line plus the 
whole $w^l$ term on the second line as running cost, see for instance \cite[Proposition 1]{cec-pel2019}.
This allows to represent $\widetilde w^l$, for each $l$, as the equilibrium cost to player $l$ in this auxiliary game.
Bounding each of the cost coefficients therein, we easily deduce that there exists a constant $C$ such
that
\begin{equation*}
\max_{l \in \EN}
\sup_{\bm{x},\bm{y}} \vert \tilde w^l(t,\bm{x},\bm{y})
\vert \leq C 
\biggl( 1 + \int_{t}^T 
\max_{l \in \EN}
\sup_{\bm{x},\bm{y}} \vert w^l(s,\bm{x},\bm{y})
\vert ds \biggr). 
\end{equation*}
In turn, we get that, whenever the input $w$ satisfies 
$\max_{l \in \EN}
\sup_{\bm{x},\bm{y}} \vert w^l(s,\bm{x},\bm{y}) \vert \leq C\exp( C(T-t))$ for any $t \in [0,T]$, the same holds 
for the output $\Phi(w)$. We call ${\mathcal E}$ the class of such inputs and we then prove that, for any two inputs 
$w^{(1)}=(w^{(1),l})_{l \in \EN}$ and $w^{(2)}=(w^{(2),l})_{l \in \EN}$ in ${\mathcal E}$, 
\begin{equation*}
\max_{l \in \EN} 
\sup_{\bm{x},\bm{y}} \bigl\vert 
\widetilde w^{(1),l}(t,\bm{x},\bm{y})
-\widetilde w^{(2),l}(t,\bm{x},\bm{y})
\bigr\vert
\leq C \int_{t}^T 
\max_{l \in \EN} 
\sup_{\bm{x},\bm{y}} \bigl\vert 
 w^{(1),l}(t,\bm{x},\bm{y})
- w^{(2),l}(t,\bm{x},\bm{y})
\bigr\vert
ds,
\end{equation*}
for a possibly new value of the constant $C$. 
The end of the proof is standard: $\Phi^{\circ \ell}$ creates a contraction for a large enough integer $\ell$, which shows the existence of a solution within the class ${\mathcal E}$. Uniqueness over bounded solutions (that are not \textit{a priori} assumed to be in ${\mathcal E}$) is proven in the same way. 
\end{proof}

\subsection{Interpretation of the renormalized Nash system}
\label{subse:interpretation:nash}
The idea is to show that the functions $(w^{N,l})_{l \in \EN}$, as given by 
Proposition 
\ref{prop:nash:system:posedness}, 
are the value functions of a new differential game, with new dynamics, but 
with an equilibrium whose feedback functions are the same as those given by 
Proposition 
\ref{prop:existence:!:continuous:eq}. 
Although this auxiliary game has no real purpose from the modeling point of view, 
the resulting formulation of the functions 
$(w^{N,l})_{l \in \EN}$ as the values of this new game 
gives almost for free some very useful bounds on the 
$(w^{N,l})_{l \in \EN}$'s, see
Proposition 
\ref{prop:tilde:Nash}
below, and 
in turn on the feedback 
strategies 
of the original game, 
as identified in 
\eqref{optcon}.
From a mathematical point of view, the intuition behind the construction of such a new game is nothing but recognizing in Eq. \eqref{newnash} the generator of another process.

However, the definition of this new game requires some care since the state variable of a given player is no longer
an element of  
$\ES \times {\mathbb Q}_{+}$ but of $\ES^N \times ({\mathbb Q}_{+})^N$. In other words,
the state process writes  
as a tuple of processes $(\widetilde{\boldsymbol X}^l,\widetilde{\boldsymbol Y}^l)_{l \in \EN}$, each 
$(\widetilde{\boldsymbol X}^l,\widetilde{\boldsymbol Y}^l)$ writing itself 
as a tuple $(\widetilde{X}^{l,n},\widetilde{Y}^{l,n})_{n \in \EN}=((\widetilde{X}^{l,n}_{t},\widetilde{Y}^{l,n}_{t})_{t \in [0,T]})_{n \in \EN}$ with values in $\ES^N$. 
For each $l \in \EN$, 
$(\widetilde{\boldsymbol X}^l,\widetilde{\boldsymbol Y}^l)$ is the state process associated with player $l$. Below, we must  
distinguish $(\widetilde{\boldsymbol X}^{l},\widetilde{\boldsymbol Y}^l)=(\widetilde{X}^{l,n},\widetilde{Y}^{l,n})_{n \in \EN}$ from 
$(\widetilde{\boldsymbol X}^{{\sqbullet},n},\widetilde{\boldsymbol Y}^{{\sqbullet},n}):=(\widetilde{X}^{l,n},\widetilde{Y}^{l,n})_{l \in \EN}$. 
The interpretation of the latter 
is made clear in the next lines, but say right now that each 
 $\widetilde{\boldsymbol Y}^{{\sqbullet},n}$ is required to take values in 
 $\setQ$ (which is similar to what we required from ${\boldsymbol Y}$ in the original game). 

As before, each player chooses a feedback function.
The subtlety is that, even though the state space has been enlarged to 
$\ES^N$, the feedback function to player $l \in \EN$ is still regarded as a function 
$\alpha^l : [0,T] \times \ES^N \times \setQ \ni (t,\blx,\bly) \mapsto 
(\alpha^l(t,\blx,\bly)[j])_{j \in \ES : j \not =x^l} \in ({\mathbb R}_{+})^{d-1}$.
In particular, the feedback function to player $l \in \EN$ does not see the additional index $n$ we used 
for enlarging the state variable. 
Now, we postulate that, for each $n \in \EN$, the state process 
$(\widetilde{\boldsymbol X}^{{\sqbullet},n},\widetilde{\boldsymbol Y}^{{\sqbullet},n})$
is a Markov process with values in 
$\ES^N \times {\mathbb Y}$
with generator 
\be 
\label{gen2}
\begin{split}
\widetilde{\mathcal{L}}^{n,N}_t v(\bm{x},\bm{y})&= 
\sum_{l=1}^N  \left( {\varphi(\mu^N_{\bm{x},\bm{y}}[\bullet])}+\alpha^l(t,\bm{x},\bm{y}) {[\bullet]}\right)\cdot \Delta^l v(\bm{x},\bm{y}) {[\bullet]} 
\\
&\hspace{-45pt} + \varepsilon N \EE 
\biggl[ \frac{S_{\mu^N_{\bm{x},\bm{y}}}[x^n]}{N\mu^N_{\bm{x},\bm{y}}[x^n]}
\biggl(v\biggl(\bm{x}, y^1\frac{S_{\mu^N_{\bm{x},\bm{y}}}[x^1]}{N\mu^N_{\bm{x},\bm{y}}[x^1]},\dots, y^N\frac{S_{\mu^N_{\bm{x},\bm{y}}}[x^N]}{N\mu^N_{\bm{x},\bm{y}}[x^N]}\biggr)
 -v(\bm{x},\bm{y})\biggr)
\biggr],
\end{split}
\ee
the second line of which differs substantially from the second line of \eqref{gen} (and explicitly depends on the index $n$ appearing on the left-hand side). Equivalently, similar to 
\eqref{eq:markov:transition:2}, 
the corresponding 
jumps of the weight process obey the transitions:
\be
\label{eq:markov:transition:3}
\begin{split}
\P& \biggl( \widetilde{\bm{X}}^{\sqbullet,n}_{t+h}=\bm{x},\widetilde{Y}^{l,n}_{t+h}=y^l\frac{k^{x^l}}{N\mu^N_{\bm{x},\bm{y}}[x^l]} \ \forall l\in \NN\bigg| \widetilde{\bm{X}}_t^{\sqbullet,n}=\bm{x}, \widetilde{\bm{Y}}_t^{\sqbullet,n}=\bm{y}\biggr) \\
&\qquad= \frac{k^{x^n}}{N\mu^N_{\bm{x},\bm{y}}[x^n]}
{\mathcal M}_{N,\mu^N_{\blx,\bly}}({\boldsymbol k}) \varepsilon N h+o(h).
\end{split}
\ee
As already explained, the ratio $k^{x^n}/N\mu^N_{\bm{x},\bm{y}}[x^n]$ is defined as 1 if $\mu^N_{\bm{x},\bm{y}}[x^n]=0$.
Also, 
observe that, by definition of the multinomial distribution $\mathcal{M}_{N,\mu}$, the transition rate $
\varepsilon N( k^{x^n}/ N\mu^N_{\bm{x},\bm{y}}[x^n])
{\mathcal M}_{N,\mu^N_{\blx,\bly}}({\boldsymbol k}) $ is equal to $\varepsilon N \mathcal{M}_{N-1,\mu} (k^{x^n}-1, \bm{k}^{-n})$, which in particular implies that this transition rate is still bounded by $\varepsilon N$.
Of course, the transitions of the $\bm{x}$-variable on $\ES$ are the same as in \eqref{eq:markov:transition:1}.

For sure, the reader may worry about the correlations between the various Markov processes 
$((\widetilde{\boldsymbol X}^{{\sqbullet},n},\widetilde{\boldsymbol Y}^{{\sqbullet},n}))_{n \in \EN}$, but in fact they do not matter. The reason is that the cost functional to player $l \in \EN$ is defined as  
\be 
\label{cost2}
\widetilde{J}^l(t,\bm{x},\bm{y}, \bm{\alpha}) 
:= \E \biggl[
\int_t^T \biggl(L\bigl(\widetilde{X}^{l,l}_s,\alpha^l(s,\widetilde{\bm{X}}^{\sqbullet,l}_s,\widetilde{\bm{Y}}^{\sqbullet,l}_s)\bigr)+f\bigl(\widetilde{X}^{l,l}_s,\widetilde{\mu}^{\sqbullet,l}_s\bigr)\biggr)ds + g\bigl(\widetilde{X}^l_T,\widetilde{\mu}^{\sqbullet,l}_T\bigr)
\biggr],
\ee
where $(t,\bm{x},\bm{y})$ belongs to $[0,T] \times \ES^N \times \setQ$ and is understood as the initial condition of all 
the processes 
$(\widetilde{\bm{X}}^{\sqbullet,n},\widetilde{\bm{Y}}^{\sqbullet,n})_{n \in \EN}$
(all of them being thus required to start from the same initial condition). In the right hand side, $\widetilde{\mu}^{\sqbullet,n}$
is defined as
$\widetilde{\mu}_{s}^{\sqbullet,n}=\mu^N_{\widetilde{X}_{s}^{\sqbullet,n},\widetilde{Y}_{s}^{\sqbullet,n}}$. 
We must insist once again on the difference between \eqref{cost2}
and \eqref{cost}: In \eqref{cost2}, the dynamics of 
the weights of the (other) players are computed with respect to transitions that truly depend on the index $l$ (through the second line in \eqref{gen2}), which explains the rather unusual 
formulation of the game.

Following the proof of the verification argument (see the proof of Proposition 
\ref{prop:verification:Nash} in the appendix), we can prove that the $N$-tuple of 
feedback functions $({\alpha}^{*,l} : [0,T] \times \ES^N \times \setQ \rightarrow ({\mathbb R}_{+})^{d-1})_{l \in \EN}$, 
defined by 
${\alpha}^{*,l}(t,\bm{x},\bm{y})=a^*(x^l,w^{N,l}_{\bullet}(t,\bm{x},\bm{y}))$, 
is a Nash equilibrium of the new game and that 
the $N$ functions 
$(w^{N,l})_{l \in \EN}$ are the value functions of this equilibrium\footnote{Certainly, we could prove a form of uniqueness of this equilibrium, but it would be rather useless for us. For this reason, we feel better not to address it.}. 
There is however a difference with Proposition \ref{prop:verification:Nash},
since, for any $l \in \EN$, ${\alpha}^{*,l}(t,\bm{x},\bm{y})$ is well-defined even if $y^l=0$.
In fact, $(\alpha^{*,l})_{l \in \EN}$ coincides with the equilibrium given by  
Proposition 
\ref{prop:existence:!:continuous:eq} constructed in the proof of Section 
\ref{sec:proof:prop2} below. 

Importantly, the interpretation of $(w^{N,l})_{l \in \EN}$ as the values of the Nash equilibrium
permits to get a bound, independently of $N$:
\begin{prop}
\label{prop:tilde:Nash}
The value functions $(w^{N,l})_{l \in \EN}$ are uniformly bounded by the constant $T\|f\|_\infty  + \|g\|_\infty$.
Accordingly the feedback functions 
given by 
$${\alpha}^{*,l}(t,\bm{x},\bm{y})[j]=
a^*\bigl(x^l,w^{N,l}_{\bullet}(t,\bm{x},\bm{y})\bigr)[j]
= 
\Bigl(w^{N,l}\big(t,\bm{x},\bm{y})- w^{N,l}\bigl(t,(j,\bm{x}^{-l}),\bm{y}\bigr)\Bigr)_+ ,
$$
for $l \in \EN$, 
are bounded by $2T\|f\|_\infty  + 2\|g\|_\infty$
 
\end{prop}

\begin{proof}
The upper bound 
for $w^{N,l}$
is found by playing the strategy $\alpha^l\equiv 0$, while the lower bound follows by the sign of $\ell$.  
The bound for the feedback functions is obvious.
\end{proof}

\subsection{Proof of Proposition \ref{prop:existence:!:continuous:eq}}
\label{sec:proof:prop2}
Here comes now the proof of  
Proposition \ref{prop:existence:!:continuous:eq}. 
Existence of a solution 
follows from 
Proposition 
\ref{prop:nash:system:posedness}. It suffices to let $v^{N,l}(t,\blx,\bly)=y^l w^{N,l}(t,\blx,\bly)$. Then, identity 
\eqref{optcon}
becomes
\begin{equation*}
\alpha^{*,l}(t,\bm{x},\bm{y})[j]= \Bigl(w^{N,l}\big(t,\bm{x},\bm{y})- w^{N,l}\bigl(t,(j,\bm{x}^{-l}),\bm{y}\bigr)\Big)_+,
\end{equation*} 
for $l \in \EN$ such that $j \not = x^l$ and $y^l >0$. 
The bound for $(w^{N,l})_{l \in \NN}$
directly follows from Proposition 
\ref{prop:tilde:Nash}. 

The uniqueness result (which is not really needed in the rest of the paper) is proven in appendix.

\begin{rem}
Let us comment on the information that each player uses in case there is or there is not common noise. Starting with the latter case, let us compare between \cite{cec-pel2019} and \cite{bay-coh2019}. In \cite{cec-pel2019} admissible strategies for the $N$-player games are Markov feedback controls with full state information, which is represented by $\boldsymbol{x}$. The characterization of the equilibrium is given by a system of $N$ equations, where its solution $(\bar v^{N,l}(t,\boldsymbol{x}))_{l=1}^N$ stands for the value functions for the players at time $t$, given that at this time the states of the players are described by the vector $\boldsymbol{x}\in\mathbb{R}^N$. On the other hand, in \cite{bay-coh2019}, each player knows the current state and the empirical distribution of the other players. The equilibrium is now characterized by a single ODE, the solution of which is denoted by  $\bar V^N(t,x,\eta)$. This is the value at time $t$ of a representative player, whose state at this time is $x$ while the empirical distribution of the other players is $\eta\in\mathcal{P}^{n-1}(\ES)$. One can show that for any $l\in\EN$, $\bar v^{N,l}(t,\boldsymbol{x})=\bar V^N(t,x^l,(1/(N-1))\sum_{n,n\ne l}\delta_{x_n})$. This means even if the players have access to the entire configuration $\boldsymbol{x}$, each player may use, instead of the full information the more concise information: private state and empirical distribution. This should not come as a surprise because the game is symmetric and the players are anonymous in the sense that the cost for each player depends on its private state and the empirical distribution.  Hence, it is reasonable that this information is sufficient to describe the equilibrium.

In contrast, in the model studied here (with common noise), the problem would change if strategies were restricted to time, players' own states, and the weighted empirical measure of the system. While players are anonymous, in the sense that the identities of the other players do not matter, players still need to keep track of the weights of the others. A sufficient piece of information for each of the players is the private state and weight, and the {\it distribution} of the weights of the players {\it within} each state. As above, this observation can be supported by showing that the value function of a representative player in the equilibrium with more concise information coincides with the value functions of the players with the full information $(\boldsymbol{x},\boldsymbol{y})$. Now, recall that whenever a player moves between states it carries its weight with it to the new state. Hence, knowing the distribution of the weights of the players within each state is not equivalent to merely knowing the {\it total} weight of the players in each state, i.e., $\mu^N_{\boldsymbol{x},\boldsymbol{y}}$. In other words,
$\mu^N_{\boldsymbol{x},\boldsymbol{y}}$ does not provide enough information on the weights of the others. As an example, there is a difference between the case when there is a state that is occupied with a single player or two players, in each case, having the same collective mass. 
\end{rem}

\section{Auxiliary results for the proofs of Theorems \ref{thm:convergence_value} and \ref{thm:convergence}}
\label{sec:4:b}
The proofs of Theorems 
\ref{thm:convergence_value} and \ref{thm:convergence}
rely on several intermediary results, which are stated in this section: 
In Subsection \ref{sec:4}, we collect several estimates on the weight process
${\boldsymbol Y}^N$; Subsection \ref{sec:5} provides a first connection between the solution 
$U$ of the master equation, as defined in 
\eqref{eq:master:equation:intrinsic}, and the Nash system
\eqref{nash}, very much in the spirit of 
\cite{CardaliaguetDelarueLasryLions}. Since the proofs of those results are rather lengthy and technical, we feel better to postpone them to Sections 
\ref{sec:6:b}
and \ref{se:7:b}
respectively, as otherwise they could distract the reader from the main line of the text. 

\subsection{Analysis of the weight process}\label{sec:4}
As we already alluded to several times, the main new point in our model is the weight process ${\boldsymbol Y}^N$. 
In this regard, we need to establish first some preliminary estimates of ${\boldsymbol Y}^N$ before we address the convergence problem itself. 
A simple look at 
\eqref{eq:markov:transition:2}
shows
that this might be rather involved: The transitions of the weight process are determined by the empirical measure $\mu^N_{t}$, but, in turn, the latter is itself 
defined in terms of the weights; and most of all, the weights show up in 
\eqref{eq:markov:transition:2} 
through the inverse quantities $(1/\mu^N_{t}[i])_{i \in \ES}$. For sure, it is worth recalling that, whenever 
$\mu^N_{t}[i]$ is zero, the ratio
$k^{i}/(N\mu^N_{\bm{x},\bm{y}}[i])$ in the definition of the transition probability 
is understood as 1 and is thus well defined; but, this conventional rule cannot prevent us from a careful analysis of the boundary behavior of the empirical measure and in particular of the reachability of the boundary of the simplex ${\mathcal P}(\ES)$. Note in this regard that, even though $\mu^N_t[i]$ is positive for some time $t \in [0,T)$, it may jump to $0$ after an infinitesimal time. The good point is that, whenever $\mu^N_{t}[i]$ is sufficiently far away from zero, this may happen with a small probability only. As a result, we manage to prove below that
the coordinates of 
the empirical measure can hardly touch $0$,
 provided  
the latter start sufficiently far away from it and the constant $\kappa$ in \eqref{eq:varphi} is large enough. In fact, this result is fully consistent with the analysis performed in 
\cite[\S 2.2.1]{BCCD-arxiv}, where we proved that the equilibria of the limiting mean field game cannot touch the boundary of the simplex.

Throughout the subsection, the number of players $N$ is fixed. Moreover, all the results stated in the subsection are proved in 
Section \ref{sec:6:b}.

\subsubsection{General setting}
\label{subse:4:1}
Actually, not only we need to prove that the empirical measure associated with 
$({\boldsymbol X},{\boldsymbol Y})$ remains away from the boundary of the simplex with 
high probability, but we also need to prove it for the process $(\widetilde{\boldsymbol X}^{\sqbullet,n},\widetilde{\boldsymbol Y}^{\sqbullet,n})_{n \in \EN}$ introduced in the previous section, see 
Subsection \ref{subse:interpretation:nash}. 

In order to make the statement as general as possible, we thus assume that we are given an 
$\ES^N \times \setQ$-valued  process
$(\overline{\bm{X}},\overline{\bm{Y}})=(\overline{X}^l,\overline{Y}^l)$
 (which must be thought 
of as $(\bm{X},\bm{Y})$ itself or as one of 
the $(\widetilde{\bm{X}}^{\sqbullet,n},\widetilde{\bm{Y}}^{\sqbullet,n})$'s, for some 
$n \in \EN$) satisfying the analogue of 
\eqref{eq:markov:transition:1} (with $(\bm{X},\bm{Y})$ replaced by
$(\overline{\bm X},\overline{\bm Y})$) for some feedback function 
${\boldsymbol \alpha}=(\alpha^l)_{l \in \EN}$, bounded by $2 (T \|f \|_{\infty} + \| g\|_{\infty})$, 
together with 
the analogue of either 
\eqref{eq:markov:transition:2}
or
\eqref{eq:markov:transition:3}.

While the meaning for the analogue of 
\eqref{eq:markov:transition:1}
should be clear, 
we feel useful to write down explicitly the analogue of 
\eqref{eq:markov:transition:2}
or
\eqref{eq:markov:transition:3} in the following form:
\be
\label{eq:markov:transition:4}
\begin{split}
\P& \biggl( \overline{\bm{X}}_{t+h}=\bm{x},\overline{Y}^{l}_{t+h}=y^l\frac{k^{x^l}}{N\mu^N_{\bm{x},\bm{y}}[x^l]} \ \forall l\in \NN\bigg| \overline{\bm{X}}_t=\bm{x}, \overline{\bm{Y}}_t=\bm{y}\biggr) \\
&\qquad= \varepsilon N \Bigl( \frac{k^{x^n}}{N\mu^N_{\bm{x},\bm{y}}[x^n]}
\Bigr)^{\iota}
{\mathcal M}_{N,\mu^N_{\blx,\bly}}({\boldsymbol k})  h+o(h),
\end{split}
\ee
where $\iota \in \{0,1\}$ and $n\in \NN$ are fixed once for all in the dynamics. 

In fact, following 
\cite{Cecchin2017}, it is convenient to represent the dynamics of $(\overline{ \bm{X}},\overline{ \bm{Y}})$ as
the solutions of SDEs driven by Poisson random measures. To do so, we let $\N^0,\N^1,\dots, \N^N$ be independent Poisson random measures with respective intensity measures $\nu^0$ on $[0,\varepsilon N]^{\NN^d}$ and $\nu^l$ on $[0,M]^d$, for $l\in\NN$, with 
\begin{align}\label{eq:M}
M:= \kappa +2 (T \|f \|_{\infty} + \| g\|_{\infty}).
\end{align} Intuitively, the $(\N^l)_{l\in\NN}$'s can be thought of as the idiosyncratic noises and $\N^0$ as the common noise. 
Also, 
\begin{equation}
\label{eq:intensity:measures}
\begin{split}
\forall l \in \EN, \quad &d\nu^l(\theta) = \sum_{j \in \ES}
\one_{[0,M]^d}(\theta)
 d \theta^j d \delta_{{0}_{{\mathbb R}^{d-1}}}(\theta^{-j}), 
 \\
&d\nu^0(\theta) = \sum_{{\boldsymbol k} \in \NN^d} 
\one_{[0,\varepsilon N]^{\NN^d}}(\theta)
d \theta^{\bm{k}}d \delta_{{0}_{{\mathbb R}^{N^d-1}}}(\theta^{-\bm{k}}),
\end{split}
\end{equation}
where $\theta^{-j}$ in the first line is the $(d-1)$-tuple $(\theta^{1},\cdots,\theta^{j-1},\theta^{j+1},\cdots)$
and similarly for $\theta^{-\bm{k}}$ on the second line. 
In particular, all the $N$ measures $\nu^1,\cdots,\nu^N$ are in fact the same and can be just denoted by 
$\nu$. Notice also that $\mathcal{N}^0$ depends on $N$.
Then, the dynamics of $(\overline{X}^l,\overline{Y}^l)$ can be written, for any $l\in\NN$, as
\be \label{tilde_processes}
\begin{split}
&d\overline X_t^l = \int_{[0,M]^d} \sum_{j\in \dd} (j-\overline X_{t-}^l)\one_{(0,\beta_{t-}^l(j)]}(\theta^j)\N^l(d\theta, dt) ,\\
&d \overline Y_t^l 
\\
&= \int_{[0,\varepsilon N]^{\NN^d}} \sum_{\bm{k}\in \NN^d} \biggl( \overline Y^{l}_{t-} \frac{k^{\overline{X}^{l}_{t-}}}{N \overline\mu_{t-} [\overline X^{l}_{t-}]}
 - \overline Y^{l}_{t-}\biggr)
 \one_{\{\overline \mu_{t-}[\overline X^{l}_{t-}] \neq 0\}}
 \one_{(0,\beta_{t-}^0({\boldsymbol k})]} (\theta^{\bm{k}}) \N^0(d\theta, dt),
\end{split}
\ee
where 
$$\overline \mu_{t}
:= \mu^N_{\overline {\bm{X}}_{t},\overline {\bm{Y}}_{t}}$$ and 
\begin{equation}
\label{eq:beta:overline:process}
\begin{split}
&\beta_{t}^l(j) :=
\beta^l\bigl(t,\overline{\bm{X}}_{t},\overline{\bm{Y}}_{t}\bigr)[j] \ ; \quad
\beta^l (t,\bm{x},\bm{y})[j] :=
\varphi(\mu^N_{\bm{x},\bm{y}}[j])+\alpha^l(t,\bm{x},\bm{y})[j], 
\quad l \in \EN,
\\
&\beta_t^0({\boldsymbol k}):=
\beta^0\bigl(t,\overline{\bm{X}}_{t},\overline{\bm{Y}}_{t}\bigr)[{\boldsymbol k}] \ ; \quad
\beta^0(t,\bm{x},\bm{y})[{\boldsymbol k}] :=
\varepsilon N \Bigl( \frac{k^{x^{n}}}{N  \mu^N_{\bm{x},\bm{y}}[x^n]} \Bigr)^{\iota}
{\mathcal M}_{N,\mu^N_{\bm{x},\bm{y}}}({\boldsymbol k}).
\end{split}
\end{equation}
The formulation 
\eqref{tilde_processes}
prompts us to let
\begin{equation}
\label{eq:f,g:overline:process}
\begin{split}
\overline f^{l}(t,\bm{x},\bm{y},\theta) &=
\sum_{j\in \dd} (j-x^{l})\one_{(0,\beta^l(t,\bm{x},\bm{y})[j]]}(\theta^j),
\\
\overline g^l(t,\bm{x},\bm{y},\theta) &=
\sum_{\bm{k}\in \NN^d} \Bigl( y^{l} \frac{k^{x^{l}}}{N \mu^N_{\bm{x},\bm{y}} [x^{l}]}
 - y^{l} \Bigr)
 \one_{\{\mu^N_{\bm{x},\bm{y}}[x^l] \neq 0\}}
 \one_{(0,\beta^0(t,\bm{x},\bm{y})[{\boldsymbol k}]]} (\theta^{\bm{k}}).
\end{split}
\end{equation}
We then have, for any test function $v$ of the two variables $\bm{x}$ and $\bm{y}$, 
\be 
\label{integral}
\begin{split}
&\int_{[0,M]^d} \Bigl[ v\Bigl(\bigl( x^l+\overline f^{l}(t,\bm{x},\bm{y},\theta),\bm{x}^{-l}\bigr),\bm{y}\Bigr) - v(\bm{x},\bm{y})\Bigr] \nu(d\theta)
\\
&\quad =  \sum_{j\in \dd} \left(\varphi(\mu^N_{\bm{x},\bm{y}}[j])+\alpha^l(t,\bm{x},\bm{y})[j]\right)\bigl[v\bigl((j,\bm{x}^{-l}),\bm{y}\bigr)-v(\bm{x},\bm{y})\bigr],
\\
& \int_{[0,\varepsilon N]^{\NN^d}}\Bigl[ v\Bigl(\bm{x},y^1+\overline g^{1}(t,\bm{x},\bm{y},\theta), \dots, y^N+\overline g^{N}(t,\bm{x},\bm{y},\theta)\Bigr)- v(\bm{x},\bm{y}) \Bigr] \nu^0(d\theta) 
\\
&\quad= \varepsilon N \EE 
\biggl[ \biggl( \frac{S_{\mu^N_{\blx,\bly}[x^n]}}{N\mu^N_{\blx,\bly}[x^n]} \biggr)^{\iota} \biggl( v\biggl(\bm{x}, y^1\frac{S_{\mu^N_{\bm{x},\bm{y}}}[x^1]}{N\mu^N_{\bm{x},\bm{y}}[x^1]},\dots, y^N\frac{S_{\mu^N_{\bm{x},\bm{y}}}[x^N]}{N\mu^N_{\bm{x},\bm{y}}[x^N]}\biggr)
 -v(\bm{x},\bm{y})\biggr)
\biggr].
\end{split}
\ee
In the rest of the paper, we then make an intense use of It\^o's formula 
for the process $(\overline{\bm{X}},\overline{\bm{Y}})$. It writes, for a general test function $v$, 
as
\begin{equation} 
\label{eq:ito:generic}
\begin{split}
&dv\bigl(\overline{\bm{X}}_t,\overline{\bm{Y}}_t\bigr)
=\sum_{l\in\NN} \int_{[0,M]^d} \biggl( v\Bigl(\bigl(\overline X^{l}_{t-}+\overline f^{l}_{t-}(\theta),\overline{\boldsymbol  X}^{-l}_{t-}\bigr),\overline {\bm{Y}}_{t-}\Bigr) - v\bigl(\overline{\bm{X}}_{t-},\overline{\bm{Y}}_{t-}\bigr) \biggr) \N^l(d\theta, dt)
 \\
&\hspace{15pt} + \int_{[0,\varepsilon N]^{\NN^d}} \biggl( v\Bigl(\overline {\bm{X}}_{t-},\overline Y^{1}_{t-}+\overline g^{1}_{t-}(\theta), \dots, \overline{Y}^{N}_{t-}+\overline{ g}^{N}_{t-}(\theta)\Bigr)- v\bigl(\overline{\bm{X}}_{t-},\overline{\bm{Y}}_{t-}\bigr)  \biggr) \N^0(d\theta,dt),
\end{split}
\end{equation}
with the notations
\begin{equation}
\label{eq:overline:f:overline:g}
\overline{f}_{t}^l(\theta)= \overline f^l\bigl(t,\overline{\bm{X}}_{t},\overline{\bm{Y}}_{t},\theta\bigr),
\quad
\overline{g}_{t}^l(\theta)= \overline g^l\bigl(t,\overline{\bm{X}}_{t},\overline{\bm{Y}}_{t},\theta\bigr).
\end{equation}


\subsubsection{Integrability of the inverse of the empirical measure}
Recall now the parameter $\kappa$ provided in \eqref{eq:varphi}. Together with the above notation, we have the following theorem, which plays a crucial role in our subsequent analysis. Its proof is given in Subsection \ref{subse:6:1}.
\begin{thm}
\label{thmexpbound}
Under the above setting 
(which comprises in particular 
 some feedback function 
${\boldsymbol \alpha}=(\alpha^l)_{l \in \EN}$
that is bounded by  {$2(T \|f \|_{\infty}+ \|g \|_{\infty})$}), assuming in addition that the initial condition 
$(\overline{\bm X}_{0},\overline{\bm Y}_{0})$ is deterministic, 
let, for any $\epsilon \in (0,1/4)$, 
\be
\label{deftau}
\overline \tau_N := \inf \Big\{ t\geq 0 : \min_{i\in\dd} \overline \mu_t[i] < N^{-\epsilon} \mbox{ or } 
\max_{l\in\NN} \overline Y^{l}_t > \frac12 N^{1-\epsilon}
\Big\} \wedge T.
\ee
Then, for any $\lambda\geq 1$, there exists a constant $\overline \kappa_{0}$, depending only on 
$\lambda$, such that, for
any 
$\kappa \geq \overline \kappa_{0}$
and 
 any $i\in\dd$,
\be 
\label{expboundmu}
\E\bigg[\exp\bigg\{\int_0^{\overline{\tau}_N} \frac{\lambda }{\overline \mu_t[i]} dt \bigg\}\bigg] \leq \frac{C}{N^{-\epsilon}+\overline \mu_0[i]},
\ee 
for a constant $C$ depending on $\delta, 
\kappa,T,d,M$, but independent of $N$
and of the initial condition. Even more, for any $t\in[0,T]$,
\be 
\label{eq1/mu}
\E \bigg[  \frac{1}{N^{-\epsilon}+\overline \mu_{t\wedge \overline{\tau}_N}[i]} 
\exp \biggl( \int_{0}^{t\wedge \overline \tau_{N}}
\frac{\lambda}{\overline \mu_{s}[i]}
ds\biggr)
\bigg] \leq \frac{C}{N^{-\epsilon}+\overline\mu_0[i] }.
\ee
\end{thm}

\begin{rem}
Combining 
\eqref{expboundmu} and H\"older's inequality, we deduce that, for any $\lambda\geq 1$, there exists a constant $\overline \kappa_{0}$, depending only on $\lambda$, such that, for
any 
$\kappa \geq \overline \kappa_{0}$
and 
 any $i\in\dd$,
\be\label{bound_calE}
\E\bigg[\exp\bigg\{\sum_{i\in\dd}  \int_0^{\overline{\tau}_N} \frac{\lambda }{\overline \mu_t[i]} dt \bigg\}\bigg] \leq 
\prod_{i \in \ES}
\biggl(
\frac{C}{N^{-\epsilon}+\overline \mu_0[i]} \biggr)^{1/d} =
C \prod_{i \in \ES}
\bigl(
N^{-\epsilon}+\overline \mu_0[i] \bigr)^{-1/d},
\ee 
for a constant $C$ depending on $\delta, 
\kappa,T,d,M$, but independent of $N$. 

Similarly, for any $t \in [0,T]$ (and by enlarging $\overline\kappa_0$ if necessary),  
\be 
\label{eq1/muexp}
\E \bigg[ 
\prod_{i \in \ES}
 \Bigl( N^{-\epsilon}+\overline \mu_{t\wedge \overline{\tau}_N}[i]
 \Bigr)^{-1/d}
\exp \biggl( \int_{0}^{t\wedge \overline \tau_{N}}
\frac{\lambda}{\overline \mu_{s}[i]}
ds\biggr)
\bigg] \leq 
C \prod_{i \in \ES}
\bigl(
N^{-\epsilon}+\overline \mu_0[i] \bigr)^{-1/d}.
\ee
\end{rem}

Notice that all the above estimates are true if $\overline{\tau}_N=0$.

\subsubsection{Moment bounds}
\label{subse:4:2}

We can now state several bounds on the moments of $(\overline Y^{l})_{l \in \NN}$ uniformly in $N$, recalling that the index 
$n$ 
and the parameter 
$\iota \in \{0,1\}$ are fixed as in 
\eqref{eq:markov:transition:4}.
We also work with the same general feedback function 
${\boldsymbol \alpha}=(\alpha^l)_{l \in \EN}$ as in \eqref{eq:markov:transition:4}, keeping in mind that it is bounded by $2 (T \|f \|_{\infty} + \| g\|_{\infty})$.

\begin{prop}\label{thm:bound_Y_tau}
Let $\overline \tau_N$ be as in Theorem \ref{thmexpbound}. Then for any integer $\ell \geq 1$, there exists a constant $\overline \kappa_{0}$, depending only on $\ell$, such that, for
any 
$\kappa \geq \overline \kappa_{0}$ 
and
any 
$\epsilon < 1/4$  (recalling that $\epsilon$ shows up in 
\eqref{deftau}), 
\be
\label{boundY_magenta}
\sup_{0\leq t \leq T} \E \biggl[ \frac1N \sum_{l \in \NN} |\overline Y^{l}_{t\wedge \overline \tau_N}|^\ell \biggr] \leq C
\biggl( \frac1N \sum_{{l\in\NN}}  |y^l_0|^{2\ell} \biggr)^{1/2} 
\prod_{i \in \ES}
\bigl(
N^{-\epsilon}+\overline \mu_0[i] \bigr)^{-1/(2d)},
\ee
for a constant $C$ depending on $\delta, 
\kappa,T,d,M$, but independent of $N$
and of the initial condition $\bar\mu_0$, the latter being assumed to be deterministic. 
Moreover, if $\ell  \geq 3$,
\be
\label{boundtau_magenta}
\P\bigl(\overline \tau_N <T\bigr)
\le\frac C{N^{\epsilon/d}}\bigg\{ 
\prod_{i \in \ES}
\bigl(
N^{-\epsilon}+\overline \mu_{0}[i] \bigr)^{-1/d}
+
\biggl( \frac{1}{N}\sum_{l\in\NN} |y^l_0|^{2\ell} \biggr)^{1/2} 
\prod_{i \in \ES}
\bigl(
N^{-\epsilon}+\overline \mu_{0}[i] \bigr)^{-1/(2d)}
\bigg\}.
\ee
\end{prop}

As for the inverse, we have 
\begin{prop}\label{thm:bound_Y_tau:-1}
Let $\overline \tau_N$ be as in Theorem \ref{thmexpbound}. Then for any $\ell \geq 1$,
there exists a constant $\overline \kappa_{0}$, depending only on $\ell$, such that, for
any 
$\kappa \geq \overline \kappa_{0}$ 
and
any 
$\epsilon < 1/4$  
(with $\epsilon$ showing up in 
\eqref{deftau}), 
\be
\label{boundY_magenta:-1}
\begin{split}
&\sup_{0\leq t \leq T} \E \biggl[\biggl( \frac1N \sum_{l \in \NN} |\overline Y^{l}_{t\wedge \overline \tau_N}|^\ell\biggr)^{-1} \biggr] 
\\
&\leq C
\biggl[\biggl(\frac1N
 \sum_{l \in \EN} \vert y_{0}^l \vert^{\ell} \biggr)^{-1/2}
 + \exp(-c N^{1-2\epsilon}) \biggr] 
\prod_{i \in \ES}
\bigl(
N^{-\epsilon}+\overline \mu_{0}[i] \bigr)^{-1/(2d)},
\end{split}
\ee
for a constant $C$ depending on $\delta,\kappa, 
T,d,M$, but independent of $N$
and of the initial condition $\bar\mu_0$, the latter being assumed to be deterministic. 
\end{prop}

Notice that, here as well, all the above estimates are true if $\overline{\tau}_N=0$.
The proofs 
of Proposition 
\ref{thm:bound_Y_tau}
and
\ref{thm:bound_Y_tau:-1} 
are given in Subsections \ref{subse:6:3}
and \ref{subse:6:4} respectively. 
Both
rely on the following lemma, which will be also used in Section \ref{sec:6} and the proof of which is given in Subsection \ref{subse:6:2}:
\begin{lem}
\label{lem:aux:moments:binomial}
For any integer $\ell \geq 1$, we can find a constant $C$ such that, 
for any probability measure $\mu \in \cP(\ES)$ with
$\mu[i]>0$ 
for all $i \in \ES$, the following two inequalities hold true
for all $i,j \in \ES$ and $N \geq 1$:
\begin{equation}
\label{eq:aux:moments:binomial:1}
\begin{split}
&\biggl\vert \EE\biggl[ \biggl( \frac{S_{\mu}[i]}{N \mu[i]} \biggr)^{\ell} -1\biggr]
\biggr\vert \leq 
 \frac{\ell(\ell-1)}{2N \mu[i]} + C \sum_{k=3}^{\ell}\frac{1}{N^{k/2} \mu[i]^{k-1}},
\\
&\biggl\vert \EE\biggl[ 
 \frac{S_{\mu}[j]}{N \mu[j]}
\biggl\{
\biggl( \frac{S_{\mu}[i]}{N \mu[i]} \biggr)^{\ell} -1\biggr\}\biggr]
\biggr\vert \leq 
\frac{\ell(\ell+1)}{2 N \min_{e \in \ES} \mu[e]}   + C \sum_{k=3}^{\ell+1}\frac{1}{N^{k/2} \min_{e \in \ES}\mu[e]^{k-1}},
\end{split}
\end{equation}
where we use the convention that $\sum_{k =3}^2=0$. 
Moreover, for any $p \geq 2$, we can find another constant $C$ such that, 
for any $i \in \ES$ and $N \geq 1$, 
\begin{equation}
\label{eq:aux:moments:binomial:2}
\EE\biggl[\biggl\vert  \biggl( \frac{S_{\mu}[i]}{N \mu[i]} \biggr)^{\ell} -1 \biggr\vert^p \biggr]^{1/p}
\leq 
C 
\sum_{k=1}^{\ell}  
\bigl(N \mu[i]\bigr)^{-k/2}.
\end{equation}
Lastly, for any $\eta >0$, there exists a constant $c>0$ such that, for all $i \in \ES$, 
\begin{equation}
\label{eq:ldp:moments:binomial:-1}
{\mathbf P} \biggl(
\biggl\vert  \biggl( \frac{S_{\mu}[i]}{N \mu[i]} \biggr)^{\ell} -1 \biggr\vert
\geq \eta  \biggr) \leq C \exp \bigl(- c N (\mu[i])^2\bigr).  
\end{equation}
\end{lem}

\subsection{The master equation as an approximation of the normalized Nash system}\label{sec:5}

The purpose of this subsection is to formulate the analogue of 
\cite[Proposition 6.1.3]{CardaliaguetDelarueLasryLions} (for continuous state case, and \cite[Proposition 4]{cec-pel2019} for finite state space), namely to 
regard finite dimensional projections of classical solutions to the master equation \eqref{eq:master:equation:intrinsic}
as almost solutions of the Nash system \eqref{nash}.

To make it clear, recall that we defined  in \eqref{wz} $z^{N,l}(t,\bm{x},\bm{y}) = U^{x^l}(t, \mu^N_{\bm{x},\bm{y}})$, where $U\in[\mathcal{C}^{1+\gamma'/2,2+\gamma'}_{\rm WF}([0,T] \times \mathcal{S}_{d-1})]^d$, for 
the same $\gamma' \in (0,1)$ as in 
\eqref{prop:previous}, is the classical solution to the master equation. Then, we want to show that 
\begin{align}\label{eq:uN}
u^{N,l}(t,\bm{x},\bm{y}):= y^l 
z^{N,l}(t,\bm{x},\bm{y}) \ 
\end{align}
almost solves the Nash system \eqref{nash}, at least when
\begin{equation}
\label{eq:mathcal:TN}
(\blx,\bly)\in{\mathcal T}_{N}:=\Bigl\{(\blx,\bly) \in \ES^N \times \setQ : 
\min_{i\in\dd} \overline \mu^N_{\blx,\bly}[i] \geq N^{-\epsilon}, 
\
\max_{l\in\NN} y^l \leq \frac12 N^{1-\epsilon}\Bigr\},
\end{equation}
for some $\epsilon \in (0,1/4)$, which has to be understood as the same 
$\epsilon$ as in 
\eqref{deftau}.
In comparison with the proof performed in 
\cite[Proposition 6.1.3]{CardaliaguetDelarueLasryLions}, 
one difficulty comes from the definition of $\mathcal{C}^{1+\gamma'/2,2+\gamma'}_{\rm WF}$, as the second-order derivatives of elements of the latter space may be singular at the boundary of the simplex.

Our result takes the following form: 
\begin{prop}\label{prop:u}
Let $\epsilon < 1/4$  be as 
in the definition of 
\eqref{eq:mathcal:TN}
Then, the function $u^{N,l}$ defined in \eqref{eq:uN} solves
\be 
\label{eq:approximate:nash:u}
\begin{split}
&\frac{d}{dt}u^{N,l}
+
\sum_{m \in \EN } 
\varphi(\mu^N_{\bm{x},\bm{y}}[\bullet])
\cdot \Delta^m u^{N,l}[\bullet] 
+\sum_{m\neq l} a^*\Big(x^m, \frac{1}{y^m} u^{N,m}_{\bullet}\Big)\cdot \Delta^m u^{N,l} [\bullet] 
\\
&\quad +y^l H\Big(x^l, \frac{1}{y^l}u^{N,l}_\bullet \Big) + y^lf(x^l, \mu^N_{\bm{x},\bm{y}})
\\
& \quad +\varepsilon N \EE 
\biggl[ u^{N,l}\biggl(t,\bm{x}, y^1\frac{S_{\mu^N_{\bm{x},\bm{y}}}[x^1]}{N\mu^N_{\bm{x},\bm{y}}[x^1]},\dots, y^N\frac{S_{\mu^N_{\bm{x},\bm{y}}}[x^N]}{N\mu^N_{\bm{x},\bm{y}}[x^N]}\biggr)
 -u^{N,l}(t,\bm{x},\bm{y})
\biggr]
\\
&= y^l r^{N,l}(t, \bm{x}, \bm{y}),  
\\
& u^{N,l}(T,\bm{x},\bm{y})= y^l g(x^l, \mu^N_{\bm{x},\bm{y}}),
\end{split}
\ee
and there exist
an exponent $\eta$, only depending on $\gamma'$ and $\epsilon$, and 
a constant $C$, only depending 
on $d$ and on
  the norm of $U$ in the space 
  $[\mathcal{C}^{1+\gamma'/2,2+\gamma'}_{\rm WF} ([0,T] \times \mathcal{S}_{d-1})]^d$, such that the rest
  $r^{N,l}(t, \bm{x}, \bm{y})$ is bounded as follows
\be
\label{eq:remainder:estimate:nash:u} 
|r^{N,l}(t, \bm{x}, \bm{y})| \leq \frac{C}{N^\eta}, 
\ee
  for  $t \in [0,T]$ and $(\blx,\bly) \in {\mathcal T}_{N}$.
\end{prop}

\begin{rem}
 Importantly,
Proposition 
\ref{prop:u}
may be reformulated in a similar result for the function $z^{N,l}$ 
defined in 
\eqref{wz}. In short, $z^{N,l}$ solves \eqref{newnash} plus the same rest $r^{N,l}$
as in \eqref{eq:approximate:nash:u} (but without the leading factor 
$y^l$).
We felt better to formulate 
Proposition 
\ref{prop:u}
as 
it makes a direct connection with the Nash system \eqref{nash}, but
we mostly use 
the version of 
\eqref{eq:approximate:nash:u}
for $z^{N,l}$
in the core of the proof of Theorem 
\ref{thm:convergence_value}. Anyway, the reader must be convinced that there is no difficulty in passing from one version 
to the other.
\end{rem}


\subsubsection{Recovering the common noise}
The strategy of proof of 
Proposition \ref{prop:u} consists in identifying the various terms of the master equation with the terms of 
\eqref{nash}.
In this respect, the most subtle term to deal with is certainly the term associated with the common noise in 
\eqref{eq:approximate:nash:u}. The following lemma makes with the connection between the latter and the second-order term in the master equation, the proof 
of this connection being given in Subsection \ref{subse:7:1}.  


\begin{lem}
\label{Nash:approximate:common:noise}
Under the assumption of 
Proposition \ref{prop:u}, we have, 
for $l \in \NN$, 
  \be
  \label{eq:5.4}
  \begin{split}
  &\frac{1}{y^l}  N \EE  
\biggl[ u^{N,l}\biggl(\bm{x}, y^1\frac{S_{\mu^N_{\bm{x},\bm{y}}}[x^1]}{N\mu^N_{\bm{x},\bm{y}}[x^1]},\dots, y^N\frac{S_{\mu^N_{\bm{x},\bm{y}}}[x^N]}{N\mu^N_{\bm{x},\bm{y}}[x^N]}\biggr)
 -u^{N,l}(t,\bm{x},\bm{y})
\biggr]
\\
&= \sum_{j\in\dd} \bigl(\delta_{j,x_l} -\mu^N_{\bm{x},\bm{y}}[j]\bigr) \fd_{j} U\bigl(t,x^l, \mu^N_{\bm{x},\bm{y}}\bigr)
\\
&\quad
+\tfrac{1}{2}
\sum_{j,k\in\dd} \bigl(\mu^N_{\bm{x},\bm{y}}[j] \delta_{j,k} -\mu^N_{\bm{x},\bm{y}}[j]\mu^N_{\bm{x},\bm{y}}[k]\bigr)
\fd^2_{j,k} U\bigl(t,x^l, \mu^{N}_{\bm{x},\bm{y}}\bigr)
+   r_{1}^{N,l}\bigl(t,\bm{x},\bm{y}\bigr),
\end{split}
  \ee
  where the rest $r_{1}^{N,l}$ is such that,
  for  $t \in [0,T]$ and $(\blx,\bly) \in {\mathcal T}_{N}$,
  \be \label{RNl}
 \bigl| r_{1}^{N,l}\bigl(t,\bm{x},\bm{y}\bigr) \bigr| \leq \frac{C}{N^\eta},  
  \ee
  for a constant $C$, only depending on $d$, 
  the norm of $U$ in the space 
  $[\mathcal{C}^{1+\gamma'/2,2+\gamma'}_{\rm WF} ([0,T] \times \mathcal{S}_{d-1})]^d$,
   and on an exponent $\eta$, which in turn only depending on $\gamma'$ and $\epsilon$.
\end{lem}

\subsubsection{Other terms}
Back to the statement of 
Proposition \ref{prop:u}, we now address the terms of the master equation that are not associated with the common noise. 
In this respect, we have the following sequence of lemmas, the first of which is of independent interest as it allows to connect first-order variations and first-order derivatives on the simplex. 
All these three lemmas are proved in Subsection \ref{subse:7:2}.

\begin{lem} 
\label{lem:exp:first:order:variation}
Under the assumption of 
Proposition \ref{prop:u}, we have
\begin{equation}
\label{deltanl}  
\begin{split}
\frac{1}{y^l} \Delta^m u^{N,l}(t,\bm{x}, \bm{y})[j] &= \frac{y^m}{N} \Big(\fd_{j} U^{x^l}(t, \mu^N_{\bm{x},\bm{y}}) - \fd_{{x^m}} U^{x^l}(t, \mu^N_{\bm{x},\bm{y}}) \Big) 
 + \varrho^{N,l,m}(t,\bm{x},\bm{y})[j],
 \quad m\neq l,
\end{split}
\end{equation}
and
\begin{equation}
\label{deltall}
\begin{split}
\frac{1}{y^l} \Delta^l u^{N,l}(t,\bm{x}, \bm{y})[j] &=
U^{j}(t, \mu^N_{\bm{x},\bm{y}}) - U^{x^l}(t, \mu^N_{\bm{x},\bm{y}}) 
\\
&+ \frac{y^l}{N} \Big(\fd_{j} U^{x^l}(t, \mu^N_{\bm{x},\bm{y}}) - \fd_{{x^l}} U^{x^l}(t, \mu^N_{\bm{x},\bm{y}}) \Big) + \varrho^{N,l,l}(t,\bm{x},\bm{y})[j], 
\end{split}
\end{equation}
where, for 
$t \in [0,T]$, $(\bm{x},\bm{y}) \in \ES^N \times \setQ$
and $m,l\in\NN$,
\be
\label{eq:conclusion:lem:exp:first:order:variation} 
\sup_{j \in \ES}|\varrho^{N,l,m}(t,\bm{x},\bm{y})[j]| \leq C \frac{ (y^m)^{1+\gamma'/2}}{N^{1+\gamma'/2}},
\ee
  for a constant $C$, only depending on 
  $d$ and on
  the norm of $U$ in the space 
  $[\mathcal{C}^{1+\gamma'/2,2+\gamma'}_{\rm WF} ([0,T] \times \mathcal{S}_{d-1})]^d$.
\end{lem}

The following lemma permits to handle the first-order terms in 
the Nash system and in the master equation. 

\begin{lem}
\label{prop:nash:1st:order}
Under the assumption of 
Proposition \ref{prop:u}, the analogue of the drift term
in the Nash system \eqref{nash}, but for $u^{N,l}$, has the following expansion
\begin{equation}
\label{eq:nash:1st:order}
\begin{split}
&\frac{1}{y^l}  \sum_{m \in \EN } 
\varphi(\mu^N_{\bm{x},\bm{y}}[\bullet])
\cdot \Delta^m u^{N,l}[\bullet] 
+ \frac{1}{y^l} \sum_{m \neq l} a^*\Big(x^m, \frac{1}{y^m} u^{N,m}_{\bullet}\Big)\cdot \Delta^m u^{N,l}  [\bullet]
\\
&=  
\sum_{k,j \in \dd} \mu_{\bm{x},\bm{y}}^{N}[k] \varphi \bigl( \mu_{\bm{x},\bm{y}}^{N}[j] \bigr)
\big({\mathfrak d}_{j} U^{x^l}(t, \mu_{\bm{x},\bm{y}}^{N}) 
- {\mathfrak d}_{k} U^{x^l}(t, \mu_{\bm{x},\bm{y}}^{N})\big)
\\
&\hspace{5pt}+ \sum_{k,j\in\dd}  \mu_{\bm{x},\bm{y}}^{N}[k]\big(U^k(t, \mu_{\bm{x},\bm{y}}^{N})-U^j(t, \mu_{\bm{x},\bm{y}}^{N})\big)_+ \big({\mathfrak d}_{j}U^{x^l}(t, \mu_{\bm{x},\bm{y}}^{N}) 
- {\mathfrak d}_{k}U^{x^l}(t, \mu_{\bm{x},\bm{y}}^{N})\big)
\\
&\hspace{5pt} + 
\sum_{j \in \dd} 
\varphi(\mu^N_{\bm{x},\bm{y}}[j])
\bigl( U^{j}(t, \mu^N_{\bm{x},\bm{y}}) - U^{x^l}(t, \mu^N_{\bm{x},\bm{y}}) \bigr)
+
r_{2}^{N,l}(t, \bm{x}, \bm{y}),
\end{split}
\end{equation}
where 
the functions $u^{N,m}$ and $\Delta^m u^{N,l}$
are evaluated at point $(t,\bm{x},\bm{y}) \in [0,T] \times {\mathcal T}_{N}$,
and where 
\begin{equation}
\label{eq:rn2l} 
\bigl|r_{2}^{N,l}(t, \bm{x}, \bm{y})\bigr| \leq \frac{C}{N^\eta}, 
\end{equation}
 for an exponent $\eta$, only depending on $\gamma$ and $\epsilon$, and 
a constant $C$, only depending 
on $d$ and on
  the norm of $U$ in the space 
  $[\mathcal{C}^{1+\gamma'/2,2+\gamma'}_{\rm WF} ([0,T] \times \mathcal{S}_{d-1})]^d$.
\end{lem}

It now remains to deal with the Hamiltonian part of the Nash system. 
\begin{lem}
\label{lem:expansion:H}
We fix $t \in [0,T]$, $(\bm{x},\bm{y}) \in {\mathcal T}_{N}$
and $l \in \NN$. 
Under the assumption of 
Proposition \ref{prop:u}, we have
	\be
	H\Big(x^l, \frac{1}{y^l}u^{N,l}_\bullet \Big) = H\bigl(x^l,U (t,\mu^N_{\bm{x},\bm{y}}) \bigr) + r^{N,l}_{3}(t, \bm{x},\bm{y}),
	\ee 
	where
	\[
	 \bigl|r^{N,l}_{3} (t, \bm{x},\bm{y})\bigr| \leq  C \frac{y^l}{N},
	\]
for	a constant $C$ that only depends 
on $d$ and on
  the norm of $U$ in the space 
  $[\mathcal{C}^{1+\gamma'/2,2+\gamma'}_{\rm WF} ([0,T] \times \mathcal{S}_{d-1})]^d$.	
\end{lem}

\subsubsection{Conclusion}   
We now complete the proof of 
Proposition
\ref{prop:u} by 
replacing the various terms on the left-hand side of 
\eqref{eq:approximate:nash:u}
by the expansions obtained in 
Lemmas 
\ref{Nash:approximate:common:noise}, 
\ref{prop:nash:1st:order}
and
\ref{lem:expansion:H}. 
Using the fact that $y^l \leq N^{1-\epsilon}$, we easily complete the proof, thanks to the fact that $U$ satisfies the master equation \eqref{eq:master:equation:intrinsic}.

\section{Proofs of Theorems \ref{thm:convergence_value} and \ref{thm:convergence}}
\label{sec:6}

We now prove the main results of the paper.

\subsection{Proof of Theorem \ref{thm:convergence_value}}
{ \ }
\vskip 5pt

\noindent We elaborate on the idea developed in 
\cite{CardaliaguetDelarueLasryLions} 
on a continuous state space, and then employed in \cite{bay-coh2019, cec-pel2019} for finite state spaces, paying attention to the fact that our setting here requires some care. 
The main noticeable difference with these references -- as explained in Remark \ref{rem2} \textit{(1)}-- is that we cannot provide a direct estimate for\footnote{ We here use the letter $n$ to denote the generic label of a player in the population, whilst we have used $l$ so far. 
This is to stay consistent with the notations introduced in 
Subsection 
\ref{subse:interpretation:nash}, which is used systematically in the sequel of this section. } 
\begin{equation*}
\sup_{\bm{x},\bm{y}}
\vert (\u^{N,n} - \v^{N,n})(t,\bm{x},\bm{y}) \vert, \quad t \in [0,T],\ n \in \NN,
\end{equation*}
the supremum being taken 
over $\bm{x} \in \ES^N$ and 
$\bm{y} \in \mathbb{Y}$. 

Here, instead, we must introduce a suitable weight and focus on the 
normalized quantity:
\begin{equation}
\label{eq:thetan:t}
\theta^n_{t}
:=
\sup_{\bm{x},\bm{y}}
\Bigl[ \prod_{i \in \ES}
\bigl(
N^{-\epsilon}+\mu_{\bm{x},\bm{y}}^N[i]
\bigr)^{1/d} \Phi^N(\bm{y}) \bigl(z^{N,n}-w^{N,n} \bigr)^2(t,\bm{x},\bm{y}) 
\Bigr],
\quad t \in [0,T], \ n \in \NN,
\end{equation}
with $\epsilon=1/8$ and $\ell=3$, 
where 
\begin{equation*}
\Phi^N(\bm{y}) := ( N^{-1} \sum_{m \in \EN} \vert y^m  \vert^{\ell} )^{-1}.
\end{equation*}
Observe in particular that the leading factor inside the supremum in 
the definition of 
\eqref{eq:thetan:t} decays as the 
$m$-moment of $\bm{y}$ increases or as 
the empirical distribution of one of the states decreases. 
In other words, the accuracy of our estimate for 
$\vert (\u^{N,n}_{t} - \v^{N,n}_{t})(\bm{x},\bm{y}) \vert$
becomes rather bad 
as $\mu^N_{\blx,\bly}$
gets closer to the boundary of the simplex or as
the $\ell$-moment of 
$\bly$ 
tends to $\infty$.

A key fact in the proof is that
$\Phi^N(\bm{y}) \leq 1$ for any 
$\bm{y} \in \Y$, since $\ell \geq 1$ and 
$ N^{-1} \sum_{m \in \EN}  y^m    =1$.
 \vskip 4pt

\noindent \textit{First Step.}
For $\ell$ as before and
for
a
 fixed index $n \in \EN$, we let 
 $$\Psi(t,\bm{x},\bm{y}):=
\Phi^N(\bm{y})
 \bigl[\u^{N,l}(t,\bm{x},\bm{y})-\v^{N,l}(t,\bm{x},\bm{y})\bigr]^2.$$ 
We then 
remind the reader of the definition of 
$(\widetilde{\boldsymbol X}^{\sqbullet,l},\widetilde{\boldsymbol Y}^{\sqbullet,l})$
in 
Subsection \ref{subse:interpretation:nash} (see 
\eqref{gen2} and
\eqref{eq:markov:transition:3}),  
with $\alpha$ in \eqref{gen2} being given by 
$\alpha^*$ as in the statement of Proposition 
\ref{prop:tilde:Nash}. 
Also, we denote by $(e_{t})_{0 \le t \le T}$ some real-valued (adapted) absolutely continuous non-decreasing  process, whose precise form will be specified later on in the proof. 
Following 
\eqref{eq:ito:generic}, 
It\^o's lemma implies that, for any $t\in[0,T]$, 
\begin{align}
&d\Bigl[ e_{t}\Psi\bigl(t,\wt{\bm X}^{\sqbullet,n}_t,\wt{\bm Y}^{\sqbullet,n}_t\bigr) \Bigr]
=
\Bigl[ e_{t} \partial_t\Psi
(t,\wt{\bm X}^{\sqbullet,n}_t,\wt{\bm Y}^{\sqbullet,n}_t) +
\dot{e}_{t} \Psi
(t,\wt{\bm X}^{\sqbullet,n}_t,\wt{\bm Y}^{\sqbullet,n}_t) \Bigr]dt 
\nonumber
\\
&+ 
\sum_{l\in\NN}\int_{[0,M]^d} e_{t}\biggl( \Psi\Bigl(t,\bigl( \wt{X}^{l,n}_{t-}+\tilde f_{t-}^{l}(\theta),\wt{\bm{X}}^{-l,n}_{t-}\bigr),\wt{\bm{Y}}^{\sqbullet,n}_{t-}\Bigr)
-\Psi\Bigl(t,\wt{\bm{X}}^{\sqbullet,n}_{t-},\wt{\bm{Y}}^{\sqbullet,n}_{t-}\Bigr)\biggr) \N^l(d\theta,dt)
\nonumber
\\
&+ \int_{[0,\varepsilon N]^{\NN^d}}e_{t} \biggl(\Psi\Bigl(t,\wt{\bm{X}}^{\sqbullet,n}_{t-},\wt{Y}^{1,n}_{t-}+\tilde g^{1}_{t-}(\theta),\ldots,\wt {Y}^{N,n}_{t-}+\tilde g^N_{t-}(\theta)\Bigr)
\label{ItoPsi}
\\
&\hspace{200pt}-\Psi\bigl(t,\wt{\bm{X}}^{\sqbullet,n}_{t-},\wt{\bm{Y}}^{\sqbullet,n}_{t-}\bigr)\biggr)\N^0(d\theta,dt),
\nonumber
\end{align}
where 
we used the same representation as 
in \eqref{tilde_processes}, with $(\overline{\bm{X}},\overline{\bm{Y}})$ therein being 
understood as $(\wt{\bm{X}}^{\sqbullet,n},\wt{\bm{Y}}^{\sqbullet,n})$. 
Equivalently,  
$({\boldsymbol \beta}^{l})_{l \in \EN}$
and ${\boldsymbol \beta}^0$, as originally defined in 
\eqref{eq:beta:overline:process}, now read
\begin{equation*}
\begin{split}
&\beta_{t}^l(j) :=
\beta^l\bigl(t,\wt{\bm{X}}^{\sqbullet,n}_{t},\wt{\bm{Y}}^{\sqbullet,n}_{t}\bigr)[j] \ ; \ 
\beta^l (t,\bm{x},\bm{y})[j] :=
\varphi(\mu^N_{\bm{x},\bm{y}}[j])
+
a^*\bigl(x^l,w^{N,l}_{\bullet}(t,\bm{x},\bm{y})\bigr),
\\
&\beta_t^0({\boldsymbol k}):=
\beta^0\bigl(t,\wt{\bm{X}}^{\sqbullet,n}_{t},\wt{\bm{Y}}^{\sqbullet,n}_{t}\bigr)[{\boldsymbol k}] \ ; \ 
\beta^0(t,\bm{x},\bm{y})[{\boldsymbol k}] :=
\varepsilon N \Bigl( \frac{k^{x^{n}}}{N  \mu^N_{\bm{x},\bm{y}}[x^n]} \Bigr) 
{\mathcal M}_{N,\mu^N_{\bm{x},\bm{y}}}({\boldsymbol k}),
\end{split}
\end{equation*}
$\iota$ being here equal to 1. 
Following 
\eqref{eq:f,g:overline:process}, we let
\begin{equation*}
\begin{split}
\tilde f^{l}(t,\bm{x},\bm{y},\theta) &=
\sum_{j\in \dd} (j-x^{l})\one_{(0,\beta^l(t,\bm{x},\bm{y})[j]]}(\theta^j),
\\
\tilde g^l(t,\bm{x},\bm{y},\theta) &=
\sum_{\bm{k}\in \NN^d} \Bigl( y^{l} \frac{k^{x^{l}}}{N \mu^N_{\bm{x},\bm{y}} [x^{l}]}
 - y^{l} \Bigr)
 \one_{\{\mu^N_{\bm{x},\bm{y}}[x^l] \neq 0\}}
 \one_{(0,\beta^0(t,\bm{x},\bm{y})[{\boldsymbol k}]]} (\theta^{\bm{k}}).
\end{split}
\end{equation*}
Then, 
the processes 
$(\tilde{f}_{t}^l(\theta))_{l \in \EN}$
and 
$(\tilde{g}_{t}^l(\theta))_{l \in \EN}$
in 
\eqref{ItoPsi}
are defined as (compare if needed with 
\eqref{eq:overline:f:overline:g})
\begin{equation*}
\tilde{f}_{t}(\theta)=f\bigl(t,\wt{\bm{X}}^{\sqbullet,n}_{t},\wt{\bm{Y}}^{\sqbullet,n}_{t},\theta\bigr),
\quad
\tilde{g}_{t}(\theta)=g\bigl(t,\wt{\bm{X}}^{\sqbullet,n}_{t},\wt{\bm{Y}}^{\sqbullet,n}_{t},\theta\bigr).
\end{equation*}
Now, for any $s\in[0,T]$, denote 
\begin{align*}
\Phi^n_{s} &:= \Phi^N\bigl( \wt{\bm Y}^{\sqbullet,n}_{s}\bigr),
\\
\u^n_s&:=\u^{N,n}\bigl(s,\wt{\bm{X}}^{\sqbullet,n}_s,\wt{\bm{Y}}^{\sqbullet,n}_s\bigr), \qquad \partial_t\u^n_s:=\partial_t \u^{N,n}\bigl(s,\wt{\bm{X}}^{\sqbullet,n}_s,\wt{\bm{Y}}^{\sqbullet,n}_s\bigr),
\\
\v^n_s&:=\v^{N,n}\bigl(s,\wt{\bm{X}}^{\sqbullet,n}_s,\wt{\bm{Y}}^{\sqbullet,n}_s\bigr), \qquad \partial_t\v^n_s:=\partial_t \v^{N,n}\bigl(s,\wt{\bm{X}}^{\sqbullet,n}_s,\wt{\bm{Y}}^{\sqbullet,n}_s\bigr).
\end{align*} 
Let $t_0\in[0,T]$ be the initial time of the process 
$(\wt{\bm{X}}^{\sqbullet,n},\wt{\bm{Y}}^{\sqbullet,n})$, for some initial condition 
$(\bm{x},\bm{y}) \in \ES^N \times \setQ$. 
Letting 
$(\tilde \mu_{t}^{\sqbullet,n}[i]:=\mu^N_{\wt{\bm{X}}_{t}^{\sqbullet,n},\wt{\bm{Y}}_{t}^{\sqbullet,n}}[i])_{t_{0} \le t \le T}$,  
following \eqref{deftau} with $\iota=1$ therein, 
we introduce the stopping time
\begin{equation}
\label{deftilde:bb}
\tilde \tau_N := \inf \Big\{ t\geq t_{0} : \min_{i\in\dd} \tilde \mu_t[i] < N^{-\epsilon} \mbox{ or } 
\max_{l\in\NN} \tilde Y^{l,n}_t > \frac12 N^{1-\epsilon}
\Big\} \wedge T.
\end{equation}
{Letting $t\in[t_0,T]$, by} integrating both sides of \eqref{ItoPsi} on the interval $[{t\wedge \tilde \tau_N},\tilde \tau_{N}]$ and recalling that $\u^n_T=\v^n_T$, we get
\begin{align*}
&e_{\tilde{\tau}_{N}}
\Phi_{\tilde{\tau}_{N}}^n
\bigl(\u^n_{\tilde \tau_{N}}-\v^n_{\tilde \tau_{N}}\bigr)^2
=
e_{t} \Phi_{t}^n (\u^n_t-\v^n_t)^2
\\
&+ \int_{t\wedge \tilde \tau_N}^{\tilde \tau_{N}}
\Phi^n_{s}\Bigl[
2 e_{s}  (\u^n_s-\v^n_s)(\partial_t\u^n_s-\partial_t\v^n_s)
+ \dot{e}_{s} 
(\u^n_s-\v^n_s)^2
\Bigr] ds
\\
&
+
\sum_{l\in\NN}\int_{t\wedge \tilde \tau_N}^{\tilde{\tau}_{N}}
\int_{[0,M]^d} e_{s} \Phi^n_{s-}\Big[ \big(\u^{N,n}-v^{N,n})^2
\Bigl(s, \bigl( \wt{X}^{l,n}_{s-}+\tilde f_{s-}^{l}(\theta),\wt{\bm{X}}^{-l,n}_{s-}\bigr),
\wt{\bm Y}^{\sqbullet,n}_{s-}\Bigr)
\\
&\hspace{150pt}
-(\u^n_{s-}-\v^n_{s-})^2\Big]\N^l(d\theta,ds)
\\
& + \int_{t\wedge \tilde \tau_N}^{\tilde \tau_{N}}
\int_{[0,\varepsilon N]^{\NN^d}}e_{s} \Big[ 
- \Phi^n_{s-} (\u^n_{s-}-\v^n_{s-})^2
\\
&\hspace{60pt} + 
\Phi^N \Bigl(\wt{Y}^{1,n}_{s-}+\tilde g^{1}_{s-}(\theta),\ldots \Bigr)
(\u^{N,n}-v^{N,n})^2
\Bigl(s,\wt{\bm{X}}^{\sqbullet,n}_{s-},\wt{Y}^{1,n}_{s-}+\tilde g^{1}_{s-}(\theta),\ldots\Bigr)
\Big]\N^0(d\theta,ds).
\end{align*}
We now take expectations (recalling that  
$(\wt{\bm{X}}^{\sqbullet,n}_{t_{0}},\wt{\bm{Y}}^{\sqbullet,n}_{t_{0}})=(\bm{x},\bm{y})$). We get
\begin{align}
&\E \bigl[e_{\tilde \tau_{N}} \Phi_{\tilde{\tau}_{N}}^n \bigl(\u^n_{\tilde \tau_{N}}-\v^n_{\tilde \tau_{N}}\bigr)^2
\bigr] \nonumber
\\
&=
\E\bigl[ e_{t}  \Phi_{t}^n \bigl(\u^n_t-\v^n_t\bigr)^2\bigr]+\E \int_{t\wedge \tilde \tau_N}^{\tilde \tau_{N}}
\Phi^n_{s} \Bigl[ 2 e_{s}
(\u^n_s-\v^n_s)(\partial_t\u^n_s-\partial_t\v^n_s)
+ 
\dot{e}_{s} (\u^n_s-\v^n_s)^2 \Bigr] ds \nonumber
\\
&\quad
+
\sum_{l\in\NN}\E \int_{t\wedge \tilde \tau_N}^{\tilde \tau_{N}}
e_{s} \Phi^n_{s} \biggl[
\Bigl(\widetilde \varphi^{\sqbullet,n}_s[\bullet]+
a^*\bigl(\wt X_{s}^{l,n},w^{l}_{s}\bigr)[\bullet]
\Bigr)
\cdot\Delta^l \bigl\{ (\u^n_s-\v^n_s)^2\bigr\}[\bullet] \biggr] ds  \label{ExpPsi}
\\
 &\quad+\varepsilon N \E 
 \int_{t\wedge \tilde \tau_N}^{\tilde \tau_{N}}
 e_{s} {\mathbf E}
  \biggl\{ 
  \frac{S_{\widetilde \mu^{\sqbullet,n}_s}[\widetilde X^{n,n}_{s}]}{N \widetilde \mu^{\sqbullet,n}_s[\widetilde X^{n,n}_{s}]}
\bigg[
 - \Phi^n_{s} (\u^{n}_s-\v^{n}_s)^2 \nonumber
\\
&\hspace{60pt} +
 \Phi^N\bigl( h(S_{\widetilde \mu^{\sqbullet,n}_s},\wt{\bm{X}}^{\sqbullet,n}_s,\wt{\bm{Y}}^{\sqbullet,n}_s)\bigr)
(\u^{N,n}-\v^{N,n})^2  \bigl(s,\wt{\bm X}^{\sqbullet,n}_s, h(S_{\widetilde \mu^{\sqbullet,n}_s},\wt{\bm X}^{\sqbullet,n},\wt{\bm Y}^{\sqbullet,n})\bigr) 
\biggr]\biggr\}
ds,
\nonumber
\end{align}
where we used the notation $\widetilde \varphi^{\sqbullet,n}_s=(\widetilde \varphi^{\sqbullet,n}_s[j]:j\in\ES)$, $\widetilde \mu^{\sqbullet,n}_s=(\widetilde \mu^{\sqbullet,n}_s[j]:j\in\ES)$, with 
\be\notag
\widetilde{\mu}_{s}^{\sqbullet,n}[j]=\mu^N_{\widetilde{X}_{s}^{\sqbullet,n},\widetilde{Y}_{s}^{\sqbullet,n}}[j], 
\qquad\text{and}\qquad \widetilde{\varphi}^{\sqbullet,n}_s[j]:=\varphi\bigl(\widetilde{\mu}^{\sqbullet,n}_s[j]\bigr).
\ee
and
$$
h(k,\bm{x},\bm{y}):
=\left(\bm{y}^1\frac{k[x^1]}{N\mu^N_{\bm{x},\bm{y}}[x^1]},\ldots, \bm{y}^N\frac{k[x^N]}{N\mu^N_{\bm{x},\bm{y}}[x^N]}\right).
$$
Recalling the notation from 
\eqref{eq:approximate:nash:u}, 
set ${r}^{n}_s:=r^{N,n}(s,\wt{\bm{X}}^n_s,\wt{\bm{Y}}^n_s)$ and 
\begin{align*}
  \alpha^{\v,l,n}_s[\bullet]&:=\widetilde{\varphi}^{\sqbullet,n}_s[\bullet]+ 
a^*\bigl(\wt X_{s}^{l,n},w^{l}_{s}\bigr)[\bullet],\qquad
H^{\v,n}_s:=H\Big(\wt X^{n,n}_s,\v^{N,n}_{\bullet}\bigl(s,\wt{\bm{X}}^{\sqbullet,n}_s,\wt{\bm{Y}}^{\sqbullet,n}_s\bigr)\Big),
\\
  \alpha^{\u,l,n}_s[\bullet]&:=\widetilde{\varphi}^{\sqbullet,n}_s[\bullet]+ 
a^*\bigl(\wt X_{s}^{l,n},\u^{l}_{s}\bigr)[\bullet]
,\qquad
H^{\u,n}_s:=H\Big(\wt X^{n,n}_s,\u^{N,n}_{\bullet}(s,\wt{\bm{X}}^{\sqbullet,n}_s,\wt{\bm{Y}}^{\sqbullet,n}_s)\Big),
\end{align*}
and using equations \eqref{nash} and \eqref{eq:approximate:nash:u}, we obtain
\begin{align}
&\E \bigl[e_{\tilde{\tau}_{N}} \Phi_{\tilde{\tau}_{N}}^n \bigl(\u^n_{\tilde \tau_{N}}-\v^n_{\tilde \tau_{N}}\bigr)^2
\bigr]
=
\E \bigl[ e_{t}  \Phi_{t}^n \bigl(\u^n_t-\v^n_t)^2\bigr]
+ 
\E \int_{t\wedge \tilde \tau_N}^{\tilde \tau_{N}} \dot{e}_{s} \Phi_{s}^n (\u^n_s-\v^n_s)^2 ds \nonumber
\\
&\quad
+2\E \biggl[\int_{t\wedge \tilde \tau_N}^{\tilde \tau_{N}} e_{s} \Phi^n_{s} (\u^n_s-\v^n_s)\Big\{\sum_{l\ne n}\Big[
\alpha^{\v,l,n}_s[\bullet] \cdot\Delta^l\v^n_s[\bullet]
-\alpha^{\u,l,n}_s[\bullet]\cdot\Delta^l\u^n_s[\bullet]
\Big]\nonumber
\\
&\hspace{4.5cm} + \widetilde \varphi^{\sqbullet,n}_s[\bullet] \cdot\Delta^n(\v^n_s-\u^n_s) [\bullet]
+H^{\v,n}_s-H^{\u,n}_s-r^n_s
\Big\}ds\biggr]\nonumber
\\
&\quad
+
\sum_{l\in\NN}\E \biggl[\int_{t\wedge \tilde \tau_N}^{\tilde \tau_{N}}
e_{s}\Phi^n_{s} \Bigl\{
\alpha^{\v,l,n}_s[\bullet] \cdot\Delta^l(\u^n_s-\v^n_s)^2[\bullet]\Bigr\} ds\biggr]\nonumber
\\
 &\quad
+2\varepsilon N {\mathbb E} \biggl[ \int_{t\wedge \tilde \tau_N}^{\tilde \tau_{N}}
e_{s} \Phi^n_{s} {\mathbf E} \Big\{
\frac{S_{\widetilde \mu^{\sqbullet,n}_s}[\widetilde X^{n,n}_{s}]}{N \widetilde \mu^{\sqbullet,n}_s[\widetilde X^{n,n}_{s}]}
(\u^n_s-\v^n_s) \label{ExpPsi2}
\\
&\hspace{60pt}
\times \Big(
(\v^{N,n}-\u^{N,n})\bigl(s,\wt{\bm{X}}^{\sqbullet,n}_s,h(S_{\widetilde \mu^{\sqbullet,n}_s},\wt{\bm{X}}^{\sqbullet,n}_s,\wt{\bm{Y}}^{\sqbullet,n}_s)\bigr)
 -(\v^n_s-\u^n_s)\Big)\Big\}
ds
\biggr]\nonumber
\\
&\quad+\varepsilon N \E \biggl[ 
\int_{t\wedge \tilde \tau_N}^{\tilde{\tau}_{N}} e_{s} {\mathbf E} \biggl\{
\frac{S_{\widetilde \mu^{\sqbullet,n}_s}[\widetilde X^{n,n}_{s}]}{N \widetilde \mu^{\sqbullet,n}_s[\widetilde X^{n,n}_{s}]}
 \biggl[   - \Phi^n_{s} (\u^n_s-\v^n_s)^2 \nonumber
 \\
&\hspace{60pt}  +   \Phi^N\bigl( h(S_{\widetilde \mu^{\sqbullet,n}_s},\wt{\bm{X}}^{\sqbullet,n}_s,\wt{\bm{Y}}^{\sqbullet,n}_s)\bigr)
(\u^{N,n}-\v^{N,n})^2  \bigl(s,\wt{\bm X}^{\sqbullet,n}_s, h(S_{\widetilde \mu^{\sqbullet,n}_s},\wt{\bm X}^{\sqbullet,n},\wt{\bm Y}^{\sqbullet,n})\bigr)  \biggr] \biggr\} ds 
\biggr]\nonumber
\\
&=:T_{1}+T_{2}+T_{3}+T_{4}+T_{5}+T_{6}.
\nonumber
\end{align}

\noindent \textit{Second Step.}
We first treat the sum of the last two terms $T_{5}$ and $T_{6}$:
\begin{align}
&T_{5}+T_{6}
\nonumber
\\
&= 
\varepsilon N \E \biggl[ 
\int_{t\wedge \tilde \tau_N}^{\tilde{\tau}_{N}} e_{s} \Phi^n_{s} 
{\mathbf E} \biggl\{
\frac{S_{\widetilde \mu^{\sqbullet,n}_s}[\widetilde X^{n,n}_{s}]}{N \widetilde \mu^{\sqbullet,n}_s[\widetilde X^{n,n}_{s}]}
 \biggl[   (\u^{N,n}-\v^{N,n})^2  \bigl(s,\wt{\bm X}^{\sqbullet,n}_s, h(S_{\widetilde \mu^{\sqbullet,n}_s},\wt{\bm X}^{\sqbullet,n}_{s},\wt{\bm Y}^{\sqbullet,n}_{s})\bigr)
 \nonumber
 \\
&\hspace{30pt} 
- 2 
(\u^n_s-\v^n_s)
(\u^{N,n}-\v^{N,n})  \bigl(s,\wt{\bm X}^{\sqbullet,n}_s, h(S_{\widetilde \mu^{\sqbullet,n}_s},\wt{\bm X}^{\sqbullet,n}_{s},\wt{\bm Y}^{\sqbullet,n}_{s})\bigr)
+  (\u^n_s-\v^n_s)^2\biggr] \biggr\} ds 
\biggr] \nonumber
\\
&\hspace{15pt}+
\varepsilon N \E \biggl[ 
\int_{t\wedge \tilde \tau_N}^{\tilde{\tau}_{N}} e_{s} {\mathbf E} \biggl\{
\frac{S_{\widetilde \mu^{\sqbullet,n}_s}[\widetilde X^{n,n}_{s}]}{N \widetilde \mu^{\sqbullet,n}_s[\widetilde X^{n,n}_{s}]}
  \Bigl[ \Phi^N  \bigl( h(S_{\widetilde \mu^{\sqbullet,n}_s},\wt{\bm X}^{\sqbullet,n}_{s},\wt{\bm Y}^{\sqbullet,n}_{s})\bigr)
  -
\Phi^n_{s} \Bigr] \label{eq:decomposition:T5+T6}
\\
&\hspace{120pt}
\times
(\u^{N,n}-\v^{N,n})^2  \bigl(s,\wt{\bm X}^{\sqbullet,n}_s, h(S_{\widetilde \mu^{\sqbullet,n}_s},\wt{\bm X}^{\sqbullet,n}_{s},\wt{\bm Y}^{\sqbullet,n}_{s})\bigr) \biggr\} ds 
\biggr]
\nonumber
\\
&:= T_0^{5,6}+T_{1}^{5,6}. 
\nonumber
\end{align}
We start with the analysis of $T_{0}^{5,6}$. 
We get 
\begin{align}\label{eq:T5+T6}
\begin{split}
&T^{5,6}_{0} =
{\mathbb E} \biggl[ \int_{t\wedge \tilde \tau_N}^{\tilde{\tau}_{N}}
e_{s} J^n_{s} ds \biggr]
\\
&\text{with} \quad J^n_{s} := \varepsilon N \Phi^n_{s} {\mathbf E}\biggl[ 
 \frac{S_{\widetilde \mu^{\sqbullet,n}_s}[\widetilde X^{n,n}_{s}]}{N \widetilde \mu^{\sqbullet,n}_s[\widetilde X^{n,n}_{s}]}
 \\
 &\hspace{60pt} \times \biggl(  \bigl(\u^{N,n}-\v^{N,n}\bigr)\bigl(s,\wt{\bm X}^{\sqbullet,n}_s, h(S_{\widetilde \mu^{\sqbullet,n}_s},\wt{\bm X}^{\sqbullet,n}_{s},\wt{\bm Y}^{\sqbullet,n}_{s})\bigr) 
- (\u^n_s-\v^n_s) \biggr)^2 \biggr].
\end{split}
\end{align}
Back to
\eqref{eq:decomposition:T5+T6}, 
we now adress the term 
$T_{1}^{5,6}$. 
We write
\begin{equation*}
\begin{split}
&T_{1}^{5,6} = \E \int_{t\wedge \tilde \tau_N}^{\tilde \tau_{N}}
e_{s} T_{2}^{5,6}(s) ds,
\\
&\textrm{with} \quad
T_{2}^{5,6}(s) := 
 \varepsilon N \mathbf{E} \biggl[ 
\frac{S_{\widetilde \mu^{\sqbullet,n}_s}[\widetilde X^{n,n}_{s}]}{N \widetilde \mu^{\sqbullet,n}_s[\widetilde X^{n,n}_{s}]}
  \Bigl[ \Phi^N  \bigl( h(S_{\widetilde \mu^{\sqbullet,n}_s},\wt{\bm X}^{\sqbullet,n}_{s},\wt{\bm Y}^{\sqbullet,n}_{s})\bigr)
  -
\Phi^n_{s} \Bigr] 
\\
&\hspace{120pt}
\times
(\u^{N,n}-\v^{N,n})^2  \bigl(s,\wt{\bm X}^{\sqbullet,n}_s, h(S_{\widetilde \mu^{\sqbullet,n}_s},\wt{\bm X}^{\sqbullet,n}_{s},\wt{\bm Y}^{\sqbullet,n}_{s})\bigr) \biggr\}  
\biggr].
\end{split}
\end{equation*}
We now recall that $\Phi({\bm y}) = ( N^{-1} \sum_{l \in \EN} \vert y^l  \vert^{\ell} )^{-1}$. 
By the convexity of the mapping $x\mapsto x^{-1}$, it is well-checked that 
\begin{equation*}
\Phi\bigl(
h(S_{\widetilde \mu^{\sqbullet,n}_s},\wt{\bm X}^{\sqbullet,n},\wt{\bm Y}^{\sqbullet,n})
\bigr) - 
\Phi\bigl(\wt{\bm Y}_{s}^{\sqbullet,n} \bigr)
 \geq - 
\Phi^2\bigl(\wt{\bm Y}_{s}^{\sqbullet,n} \bigr) 
D  \wt{\bm Y}_{s}^{\sqbullet,n},
\end{equation*}
where
\begin{equation} 
\label{eq:DeltaY}
\begin{split}
D 
 \wt{\bm Y}_{s}^{\sqbullet,n}
 &:=  \frac1N  \sum_{l \in \EN}
\vert \wt{\bm Y}_{s}^{l,n}  \vert^{\ell}
\biggl[ \biggl(
 \frac{S_{\widetilde \mu^{\sqbullet,n}_s}[\widetilde X^{l,n}_{s}]}{N \widetilde \mu^{\sqbullet,n}_s[\widetilde X^{l,n}_{s}]}
\biggr)^{\ell}
-
1\biggr].
\end{split}
\end{equation}
from which we deduce that 
\begin{equation*}
\begin{split}
&T_{2}^{5,6}(s) 
\\
&\geq-
\varepsilon N 
 {\mathbf E}
\biggl[   \frac{S_{\widetilde \mu^{\sqbullet,n}_s}[\widetilde X^{n,n}_{s}]}{N \widetilde \mu^{\sqbullet,n}_s[\widetilde X^{n,n}_{s}]}
\bigl( z^{N,n} - w^{N,n} \bigr)^2\bigl(s,\wt{\bm{X}}^{\sqbullet,n}_s,h(S_{\widetilde \mu^{\sqbullet,n}_s},\wt{\bm{X}}^{\sqbullet,n}_s,\wt{\bm{Y}}^{\sqbullet,n}_s)\bigr)
(\Phi^n_{s})^2 
D 
 \wt{\bm{Y}}^{\sqbullet,n}_s
\biggr]
\\
&= - 
T_{2,1}^{5,6}(s) - 
T_{2,2}^{5,6}(s)+ 
T_{2,3}^{5,6}(s) ,
\end{split}
\end{equation*}
with 
\begin{align*}
&T_{2,1}^{5,6}(s) :=
\varepsilon N 
 {\mathbf E}
\biggl[ \frac{S_{\widetilde \mu^{\sqbullet,n}_s}[\widetilde X^{n,n}_{s}]}{N \widetilde \mu^{\sqbullet,n}_s[\widetilde X^{n,n}_{s}]} 
\biggl( 
\bigl( z^{N,n} - w^{N,n} \bigr)\bigl(s,\wt{\bm{X}}^{\sqbullet,n}_s,h(S_{\widetilde \mu^{\sqbullet,n}_s},\wt{\bm{X}}^{\sqbullet,n}_s,\wt{\bm{Y}}^{\sqbullet,n}_s)\bigr)
\\
&\hspace{130pt}
- 
\bigl( z^{n}_{s} - w^{n}_{s} \bigr)
\biggr)^2
(\Phi^n_{s})^2 
D 
 \wt{\bm{Y}}^{\sqbullet,n}_s
\biggr],
\\
&
T_{2,2}^{5,6}(s) :=
2 \varepsilon N 
{\mathbf E}
\biggl[ 
\frac{S_{\widetilde \mu^{\sqbullet,n}_s}[\widetilde X^{n,n}_{s}]}{N \widetilde \mu^{\sqbullet,n}_s[\widetilde X^{n,n}_{s}]} 
 \biggl( 
 \bigl( z^{N,n} - w^{N,n} \bigr)\bigl(s,\wt{\bm{X}}^{\sqbullet,n}_s,h(S_{\widetilde \mu^{\sqbullet,n}_s},\wt{\bm{X}}^{\sqbullet,n}_s,\wt{\bm{Y}}^{\sqbullet,n}_s)\bigr)
\\
&\hspace{120pt}
- 
\bigl( z^{n}_{s} - w^{n}_{s} \bigr)
\biggr)  \bigl( z^{n}_{s} - w^{n}_{s} \bigr)
(\Phi^n_{s})^2 D 
 \wt{\bm{Y}}^{\sqbullet,n}_s
\biggr]\\
&T_{2,3}^{5,6}(s) :=
3\varepsilon N 
 {\mathbf E}
\biggl[  
\frac{S_{\widetilde \mu^{\sqbullet,n}_s}[\widetilde X^{n,n}_{s}]}{N \widetilde \mu^{\sqbullet,n}_s[\widetilde X^{n,n}_{s}]}
\bigl( z^{n}_{s} - w^{n}_{s} \bigr)^2
(\Phi^n_{s})^2  D  
 \wt{\bm{Y}}^{\sqbullet,n}_s
\biggr]. 
\end{align*}
We start with the analysis of $T_{2,3}^{5,6}(s) $. 
 Denote $\mumin_s:=\min_{e\in\ES}\widetilde\mu^{\sqbullet,n}_s[e]$. 
By the definition of $D 
 \wt{\bm{Y}}^{\sqbullet,n}$, the uniform bounds of $z^n_s$ and $w^n_s$, and the second part of \eqref{eq:aux:moments:binomial:1}, we have
\begin{equation}\label{eq:T2356:bis}
\begin{split}
| T_{2,3}^{5,6}(s)| 
&\le3\varepsilon\bigl( z^n_{s} - w^{n}_{s} \bigr)^2
 (\Phi^n_{s})^2
 \sum_{l \in \EN}
\vert \wt{\bm Y}_{s}^{l,n}  \vert^{\ell}
\biggl|\mathbf{E}\biggl[
\frac{S_{\widetilde \mu^{\sqbullet,n}_s[\widetilde X^{n,n}_{s}]}}{N \widetilde \mu^{\sqbullet,n}_s[\widetilde X^{n,n}_{s}]}
  \biggl(\biggl(
 \frac{S_{\widetilde \mu^{\sqbullet,n}_s}[\widetilde X^{l,n}_{s}]}{N \widetilde \mu^{\sqbullet,n}_s[\widetilde X^{l,n}_{s}]}
\biggr)^{\ell}
-
1\biggr)\biggr]\biggr|
\\
&\le
c N\bigl( z^n_{s} - w^{n}_{s} \bigr)^2 \Phi^n_{s}
\biggl(
\frac{1}{N\mumin_s}+\frac{1}{N^{\frac{3}{2}-2\epsilon}}
\biggr)
\\
&\le
\frac{c}{\mumin_s} \bigl( z^n_{s} - w^{n}_{s} \bigr)^2 \Phi^n_{s}
+\frac{1}{N^{\frac{1}{2}-2\epsilon}},
\end{split}
\end{equation}
for a $c$ whose value is allowed to vary from line to line as long as it only depends on 
$T$, $\| f \|_{\infty}$ and $\|g \|_{\infty}$. 
We now turn to the analysis of $T^{5,6}_{2,1}(s)$. 
Generally speaking, the strategy is to prove that it is smaller (up to a small remainder) than a small fraction of 
$J^n_{s}$ in 
\eqref{eq:T5+T6} (since the latter is positive, it will hence permit to absorb 
$-T^{5,6}_{2,1}$ in the expansion of $T^{5,6}_{2}$). 
The proof is as follows: On the event 
\begin{equation*}
F:= \biggl\{
\biggl\vert \biggl(
 \frac{S_{\widetilde \mu^{\sqbullet,n}_s}[\widetilde X^{l,n}_{s}]}{N \widetilde \mu^{\sqbullet,n}_s[\widetilde X^{l,n}_{s}]}
\biggr)^{\ell}
-
1\biggr\vert \leq \frac12
 \;:\; l\in\{1,2,\cdots,N\}\biggr\},
\end{equation*}
we have $\vert D \widetilde{\bm Y}_{s}^{\sqbullet,n} \vert
\leq (\Phi^n_{s})^{-1}/2$, 
so that (recalling the definition of $J^n_{s}$ in \eqref{eq:T5+T6})
\begin{equation*}
\begin{split}
&T_{2,1}^{5,6}(s)
 \leq \frac12  J^n_{s}
+
c 
{\mathbf E}
\biggl[   {\mathbbm 1}_{F^{\complement}}  \frac{S_{\widetilde \mu^{\sqbullet,n}_s}[\widetilde X^{n,n}_{s}]}{N \widetilde \mu^{\sqbullet,n}_s[\widetilde X^{n,n}_{s}]} 
(\Phi^n_s)^2
\bigl\vert  D \widetilde{\bm Y}_{s}^{\sqbullet,n} 
\bigr\vert
\biggr].
\end{split}
\end{equation*}
Using 
\eqref{eq:aux:moments:binomial:2}
  and recalling from 
  \eqref{eq:ldp:moments:binomial:-1}
  that ${\mathbb P}(F^{\complement}) \leq CN\exp(-c N^{1-2 \epsilon})$, 
  we deduce that 
\begin{equation}
\label{eq:T2156}
\begin{split}
&T_{2,1}^{5,6}(s)
 \leq \frac12  J^n_{s}
+
\frac{c}{{N^{1/2-2\epsilon}}}.
\end{split}
\end{equation}
 It then remains to tackle $T_{2,2}^{5,6}(s)$. 
 We use Young's inequality, Jensen's inequality, Cauchy--Schwartz inequality, and finally  \eqref{eq:aux:moments:binomial:2}:
 \begin{align}
\notag
  T_{2,2}^{5,6}(s) 
  &\leq \frac14  J^n(s) +  c N 
  {\mathbf E}
\biggl[
 \frac{S_{\widetilde \mu^{\sqbullet,n}_s}[\widetilde X^{n,n}_{s}]}{N \widetilde \mu^{\sqbullet,n}_s[\widetilde X^{n,n}_{s}]} \bigl( z^{n}_{s} - w^{n}_{s} \bigr)^2 
(\Phi^n_{s})^3 \bigl( D 
 \wt{\bm{Y}}^{\sqbullet,n}_s
\bigr)^2 \biggr]\\\notag
&=
\frac14  J^n(s) +  c N \bigl( z^{n}_{s} - w^{n}_{s} \bigr)^2 
\Phi^n_{s}\\\notag
&\quad\times
{\mathbf E}
\biggl[
 \frac{S_{\widetilde \mu^{\sqbullet,n}_s}[\widetilde X^{n,n}_{s}]}{N \widetilde \mu^{\sqbullet,n}_s[\widetilde X^{n,n}_{s}]}\biggl\{\sum_{l\in\EN}\frac{\tfrac 1N\vert \wt{\bm Y}_{s}^{l,n}  \vert^{\ell}
}{(\Phi^n_s)^{-1}}\sum_{i\in\ES}{\mathbbm 1}_{\{ \wt X_{s}^{l,n}=i\}}\biggl(\biggl(
 \frac{S_{\widetilde \mu^{\sqbullet,n}_s}[i]}{N \widetilde \mu^{\sqbullet,n}_s[i]}
\biggr)^{\ell}
-
1\biggr)\biggr\}^2
\biggr]
\\\notag
&\le
\frac14  J^n(s) +  c  N \bigl( z^{n}_{s} - w^{n}_{s} \bigr)^2 
\Phi^n_{s}\\\notag
&\quad\times
\sum_{l\in\EN}\frac{\tfrac 1N\vert \wt{\bm Y}_{s}^{l,n}  \vert^{\ell}
}{(\Phi^n_s)^{-1}}\sum_{i\in\ES}{\mathbf E}
\biggl[
 \frac{S_{\widetilde \mu^{\sqbullet,n}_s}[\widetilde X^{n,n}_{s}]}{N \widetilde \mu^{\sqbullet,n}_s[\widetilde X^{n,n}_{s}]}\biggl(\biggl(
 \frac{S_{\widetilde \mu^{\sqbullet,n}_s}[i]}{N \widetilde \mu^{\sqbullet,n}_s[i]}
\biggr)^{\ell}
-
1\biggr)^2
\biggr]\\\notag
&\le
\frac14  J^n(s) +  c  N \bigl( z^{n}_{s} - w^{n}_{s} \bigr)^2 
\Phi^n_{s}\\\notag
&\quad\times
\sum_{i\in\ES}{\mathbf E}
\biggl[
\biggl( \frac{S_{\widetilde \mu^{\sqbullet,n}_s}[\widetilde X^{n,n}_{s}]}{N \widetilde \mu^{\sqbullet,n}_s[\widetilde X^{n,n}_{s}]}\biggr)^2\biggr]^{1/2}{\mathbf E}
\biggl[\biggl(\biggl(
 \frac{S_{\widetilde \mu^{\sqbullet,n}_s}[i]}{N \widetilde \mu^{\sqbullet,n}_s[i]}
\biggr)^{\ell}
-
1\biggr)^4
\biggr]^{1/2}\\\label{eq:T2256}
&\le
\frac14  J^n(s) +
\frac{c}{\mumin_s} 
\Phi^n_{s}  \bigl( z^{n}_{s} - w^{n}_{s} \bigr)^2,
 \end{align}
where we indeed applied 
 \eqref{eq:aux:moments:binomial:2}
 with $p=4$ in the penultimate line  (noticing that the leading exponent of the last term therein is $1/2$).
 Moreover, 
the last inequality follows since
 $N\mumin\ge 1$ on the event $\{t\le \tilde\tau_N\}$. As a consequence,  
 invoking once again
 \eqref{eq:aux:moments:binomial:2}, we have a universal bound for
 $${\mathbf E}
\biggl[
\biggl( \frac{S_{\widetilde \mu^{\sqbullet,n}_s}[\widetilde X^{n,n}_{s}]}{N \widetilde \mu^{\sqbullet,n}_s[\widetilde X^{n,n}_{s}]}\biggr)^2\biggr]^{1/2}.$$
Collecting all the terms 
\eqref{eq:T2356:bis}, 
\eqref{eq:T2156}
and
\eqref{eq:T2256}, we finally have 
\begin{equation*}
T_{1}^{5,6} \geq
- \frac{3}{4} 
{\mathbb E} \int_{t\wedge \tilde \tau_N}^{\tilde \tau_{N}} e_{s} J^n_{s} ds
-c  
{\mathbb E} \biggl[ \int_{t\wedge \tilde \tau_N}^{\tilde{\tau}_{N}}
\frac{e_{s} \Phi^n_{s}}{\mumin_s} (\u^n_s-\v^n_s)^2 ds\biggr]
- \frac{c}{N^{1/2-2\epsilon}} {\mathbb E} \biggl[ \int_{t\wedge \tilde \tau_N}^{\tilde{\tau}_{N}}
e_{s} ds\biggr].
\end{equation*}
Combining with 
\eqref{eq:T5+T6}, we get
\begin{equation*}
T_{5}
+ T_{6} \geq
-  c
{\mathbb E} \biggl[ \int_{t\wedge \tilde \tau_N}^{\tilde{\tau}_{N}}
\frac{e_{s} \Phi^n_{s}}{\mumin_s} (\u^n_s-\v^n_s)^2 ds\biggr]
- \frac{c}{N^{1/2-2\epsilon}} {\mathbb E} \biggl[ \int_{t\wedge \tilde \tau_N}^{\tilde{\tau}_{N}}
e_{s} ds\biggr].
\end{equation*}

\noindent \textit{Third Step.}
We now simplify the sum of $T_{3}$ and $T_{4}$ in \eqref{ExpPsi2}. Simple algebraic manipulations imply that for any $l\ne n$, 
\begin{align*}
&2(\u^n_s-\v^n_s)\bigl\{\alpha^{\v,l,n}_s[\bullet]\cdot\Delta^l\v^n_s[\bullet]-\alpha^{\u,l,n}_s[\bullet]\cdot\Delta^l\u^n_s[\bullet]\bigr\}+\alpha^{\v,l,n}_s[\bullet]\cdot\Delta^l(\u^n_s-\v^n_s)^2[\bullet]
\\
&\quad=
2(\u^n_s-\v^n_s)\bigl\{\alpha^{\v,l,n}_s[\bullet]\cdot\Delta^l(\v^n_s-\u^n_s)[\bullet]+(\alpha^{\v,l,n}_s-\alpha^{\u,l,n}_s)[\bullet]\cdot\Delta^l\u^n_s[\bullet]\bigr\}
\\
&\qquad+\alpha^{\v,l,n}_s[\bullet]\cdot\Delta^l(\u^n_s-\v^n_s)^2[\bullet]
\\
&\quad =
2(\u^n_s-\v^n_s)\bigl(\alpha^{\v,l,n}_s[\bullet]-\alpha^{\u,l,n}_s[\bullet]\bigr)\cdot\Delta^l\u^n_s[\bullet]
\\
&\qquad+
\alpha^{\v,l,n}_s[\bullet]\cdot \Bigl[\Delta^l(\u^n_s-\v^n_s)[\bullet]\odot \Delta^l(\u^n_s-\v^n_s)[\bullet]\Bigr],
\end{align*}
where $a\odot b=(a_ib_i:i\in\ES)$ is the element by element product between vectors. Similarly,
\begin{align*}
\alpha^{\v,n,n}_s[\bullet] \cdot\Delta^n(\u^n_s-\v^n_s)^2[\bullet]&=\alpha^{\v,n,n}_s\cdot \Bigl[\Delta^n(\u^n_s-\v^n_s)[\bullet]\odot \Delta^n(\u^n_s-\v^n_s)[\bullet]\Bigr]\\
&\quad
+
2(\u^n_s-\v^n_s)\alpha^{\v,n,n}_s[\bullet] \cdot\Bigl[\Delta^n(\u^n_s-\v^n_s)[\bullet]
\Bigr]
\\
&\geq 
2(\u^n_s-\v^n_s)\alpha^{\v,n,n}_s[\bullet] \cdot\Bigl[\Delta^n(\u^n_s-\v^n_s)[\bullet]
\Bigr],
\end{align*}
where, in the last line, 
we used the fact that the off-diagonal components of $\alpha^{\v,n}_s$ are non-negative

Back to 
\eqref{ExpPsi2}, we obtain
\begin{equation}
\label{eq:conclusion:third:step}
\begin{split}
&\E \bigl[e_{\tilde{\tau}_{N}} \Phi^n_{\tilde{\tau}_{N}} \bigl(\u^n_{\tilde \tau_{N}}-\v^n_{\tilde \tau_{N}}\bigr)^2
\bigr]
\geq
\E \bigl[ e_{t} \Phi^n_{t} \bigl(\u^n_t-\v^n_t)^2\bigr]
+ 
\E \int_{t\wedge \tilde \tau_N}^{\tilde \tau_{N}} \dot{e}_{s} \Phi^n_{s} (\u^n_s-\v^n_s)^2 ds 
\\
&\quad
+2\E \biggl[\int_{t\wedge \tilde \tau_N}^{\tilde \tau_{N}} e_{s} \Phi^n_{s} (\u^n_s-\v^n_s)\Big\{\sum_{l\ne n}\Big[
\bigl(\alpha^{\v,l,n}_s[\bullet]-\alpha^{\u,l,n}_s[\bullet]\bigr)\cdot\Delta^l\u^n_s[\bullet]
\Big]
\\
&\hspace{2cm} + \bigl( \alpha^{\v,n,n}_s[\bullet] - \widetilde \varphi^{\sqbullet,n}_s[\bullet] \bigr) \cdot\Delta^n(\u^n_s-\v^n_s) [\bullet]
+H^{\v,n}_s-H^{\u,n}_s-r^n_s
\Big\}ds\biggr]
\\
&\quad
+ \sum_{l \ne n} \E \biggl[\int_{t\wedge \tilde \tau_N}^{\tilde \tau_{N}} e_{s}
\Phi^n_{s}
\alpha^{\v,l,n}_s[\bullet]\cdot \Bigl[\Delta^l(\u^n_s-\v^n_s)[\bullet]\odot \Delta^l(\u^n_s-\v^n_s)[\bullet]\Bigr]
ds \biggr]
\\
 &\quad
-  c  
{\mathbb E} \biggl[ \int_{t\wedge \tilde \tau_N}^{\tilde{\tau}_{N}}
\frac{e_{s} \Phi^n_{s}}{\mumin_s} (\u^n_s-\v^n_s)^2 ds\biggr]
- c \frac1{N^{1/2-2\epsilon}} {\mathbb E} \biggl[ \int_{t\wedge \tilde \tau_N}^{\tilde{\tau}_{N}}
e_{s} ds\biggr].
\end{split}
\end{equation}
\vskip 4pt
%

\noindent \textit{Fourth Step.} 
We now handle the 
third expectation on the right hand side of  
\eqref{eq:conclusion:third:step}.
We start with the term on the second line
of \eqref{eq:conclusion:third:step}.
From \eqref{deltanl},
\begin{equation*}
\begin{split}
\Delta^l z^{n}_{s}[j]
&=\Delta^l\u^{N,n}\bigl(s,\wt{\bm X}^{\sqbullet,n}_s,\wt{\bm Y}^{\sqbullet,n}_s\bigr)[j]
\\
&= 
\frac{\wt{Y}^{l,n}_s}{N} \Bigl[\fd_{j} U\Bigl(t, X^{n,n}_s, \mu^N_{\wt{\bm X}^{\sqbullet,n}_s,\wt{\bm Y}^{\sqbullet,n}_s}\Bigr) - \fd_{{\wt{X}^{l,n}_s}} U\Bigl(s,X^{n,n}_s, \mu^N_{\wt{\bm X}^{\sqbullet,n}_s,\wt{\bm Y}^{\sqbullet,n}_s}\Bigr) \Bigr]
\\ 
&\hspace{15pt}+ \varrho^{N,n,l}\bigl(s,\wt{\bm X}^{\sqbullet,n}_s,\wt{\bm Y}^{\sqbullet,n}_s\bigr)[j], 
\end{split}
\end{equation*}
with $|\varrho^{N,n,l}(s,\wt{\bm X}^{\sqbullet,n}_s,\wt{\bm Y}^{\sqbullet,n}_s)|\le C(\wt Y^{l,n}_s)^{1+\gamma}/N^{1+\gamma}
\leq C \wt Y^{l,n}_s/N$, the last inequality following from the fact that 
$\wt Y^{l,n}_{s} \leq N$ and the constant $C$ depending only on the various parameters in the assumption (including $\kappa$). 
%
Moreover, by the regularity of $U$, its gradient is uniformly bounded. Together with the Lipschitz property of $\alpha^*$, we obtain that, for $l\ne n$,
\begin{equation*}
\begin{split}
\Bigl\vert \bigl(\alpha^{\v,l,n}_s[\bullet]-\alpha^{\u,l,n}_s[\bullet]\bigr) \cdot\Delta^l\u^n_s[\bullet]
\Bigr\vert
&\le C  \frac {\wt Y^{l,n}_s}N \Bigl\vert \bigl(- \Delta^l \v^n_{s}[\bullet] \bigr)_{+}- \bigl(- \Delta^l \u^n_{s}[\bullet]\bigr)_{+} \Bigr\vert,
\end{split}
\end{equation*}
and then, recalling that 
$\sum_{l \not =n}
(\wt Y^{l,n}_s/N) \leq 1$, we get
(by Jensen's inequality) 
\begin{equation*}
\begin{split}
&\biggl( \sum_{l \not =n}
\Bigl\vert \bigl(\alpha^{\v,l,n}_s[\bullet]-\alpha^{\u,l,n}_s[\bullet]\bigr) \cdot\Delta^l\u^n_s[\bullet]
\Bigr\vert
\biggr)^2 
\\
&\le C \sum_{l \not =n}  
\frac {\wt Y^{l,n}_s}N
\Bigl\vert \bigl(- \Delta^l \v^n_{s}[\bullet] \bigr)_{+}- \bigl(- \Delta^l \u^n_{s}[\bullet]\bigr)_{+} \Bigr\vert^2
\\
&= 
C \sum_{l \not =n}  
\sum_{i \in \ES}
\biggl[
\frac {\wt Y^{l,n}_s}N
\Bigl\vert \bigl( - \Delta^l \v^n_{s}[i] \bigr)_{+}- \bigl(- \Delta^l \u^n_{s}[i]\bigr)_{+} \Bigr\vert^2
{\mathbbm 1}_{\{  (- \Delta^l \v^n_{s}[i] )_{+} \leq \wt Y^{l,n}_s/N \}}\biggr]
\\
&\hspace{15pt}
+ C \sum_{l \not =n}  
\sum_{i \in \ES}
\biggl[
\frac {\wt Y^{l,n}_s}N
\Bigl\vert \bigl(- \Delta^l \v^n_{s}[i] \bigr)_{+}- \bigl(- \Delta^l \u^n_{s}[i]\bigr)_{+} \Bigr\vert^2
{\mathbbm 1}_{\{  ( - \Delta^l \v^n_{s}[i] )_{+} > \wt Y^{l,n}_s/N \}}\biggr]
\\
&\leq 
C \sum_{l \not =n}  
\frac {(\wt Y^{l,n}_s)^3}{N^3}
+
C \sum_{l \not =n}  
\sum_{i \in \ES}
\biggl[
\bigl( - \Delta^l \v^n_{s}[i] \bigr)_{+}
\Bigl\vert \bigl(- \Delta^l \v^n_{s}[i] \bigr)_{+}- \bigl(- \Delta^l \u^n_{s}[i]\bigr)_{+} \Bigr\vert^2\biggr],
\end{split}
\end{equation*}
where, in order to pass from the third to the last line, we used once again the fact that 
$(- \Delta^l \u^n_{s}[i] )_{+} \leq C \wt Y^{l,n}_{s}/N$. 
Recalling that $s < \tilde{\tau}_{N}$, 
we have
\begin{equation*}
\sum_{l \not =n}  
\frac {(\wt Y^{l,n}_s)^3}{N^3}
\leq 
\Bigl( \frac{N^{1-\epsilon}}{N} \Bigr)^2
\sum_{l \not =n}  
\frac {\wt Y^{l,n}_s}{N} = N^{-2 \epsilon}.
\end{equation*}
Using in addition the upper bound $( - \Delta^l \v^n_{s}[i] )_{+}
\leq \alpha_{s}^{w,l,n}[i]$
if $i \not = \wt X_{s}^{l,n}$, we deduce that 
\begin{equation}
\label{eq:deltalzn:deltalwn}
\begin{split}
&\biggl( \sum_{l \not =n}
\Bigl\vert \bigl(\alpha^{\v,l,n}_s[\bullet]-\alpha^{\u,l,n}_s[\bullet]\bigr) \cdot\Delta^l\u^n_s[\bullet]
\Bigr\vert
\biggr)^2 
\\
&\hspace{15pt} \leq C N^{-2 \epsilon} + C 
\sum_{l \not = n}
\alpha^{\v,l,n}_s[\bullet]\cdot \Bigl[\Delta^l(\u^n_s-\v^n_s)[\bullet]\odot \Delta^l(\u^n_s-\v^n_s)[\bullet]\Bigr].
\end{split}
\end{equation}
Hence, by Young's inequality and since $\epsilon=1/8$, we deduce from 
\eqref{eq:conclusion:third:step}
that 
 \begin{equation}
\label{eq:conclusion:third:step:2}
\begin{split}
&\E \bigl[e_{\tilde{\tau}_{N}} \Phi^n_{\tilde{\tau}_{N}} \bigl(\u^n_{\tilde \tau_{N}}-\v^n_{\tilde \tau_{N}}\bigr)^2
\bigr]
\geq
\E \bigl[ e_{t} \Phi^n_{t} \bigl(\u^n_t-\v^n_t)^2\bigr]
+ 
\E \int_{t\wedge \tilde \tau_N}^{\tilde \tau_{N}} \dot{e}_{s} \Phi^n_{s} (\u^n_s-\v^n_s)^2 ds 
\\
&\quad
+2\E \biggl[\int_{t\wedge \tilde \tau_N}^{\tilde \tau_{N}} e_{s} \Phi^n_{s} (\u^n_s-\v^n_s)\Big\{
- 
C (\u^n_s-\v^n_s)
\\
&\hspace{2cm} +  \bigl( \alpha^{\v,n,n}_s[\bullet] - \widetilde \varphi^{\sqbullet,n}_s[\bullet] \bigr) \cdot\Delta^n(\u^n_s-\v^n_s) [\bullet]
+H^{\v,n}_s-H^{\u,n}_s-r^n_s
\Big\}ds\biggr]
\\
 &\quad
-   c
{\mathbb E} \biggl[ \int_{t\wedge \tilde \tau_N}^{\tilde{\tau}_{N}}
\frac{e_{s} \Phi^n_{s}}{\mumin_s} (\u^n_s-\v^n_s)^2 ds\biggr]
- C \frac1{N^{2\epsilon}} {\mathbb E} \biggl[ \int_{t\wedge \tilde \tau_N}^{\tilde{\tau}_{N}}
e_{s} ds\biggr].
\end{split}
\end{equation}
We now turn to the analysis of the
third line 
on the right hand side of \eqref{eq:conclusion:third:step:2}. 
%
%
%
Recalling that 
$\tilde \varphi^{\sqbullet,n}$, $\alpha^{\v,l,n}$, $\u^{N,n},$ and $\v^{N,n}$  are bounded, 
we get, by the definition of the Hamiltonian,
\be\notag
|H^{\v,n}_t-H^{\u,n}_t|+|\wt \varphi^{\sqbullet,n}_{s}-\alpha^{\v,n,n}_s||\Delta^n(\u^n_s-\v^n_s)|\le C|\Delta^n(\u^n_s-\v^n_s)|\le C\Theta^n_s,
\ee
where 
\begin{equation}
\label{eq:Theta:ns}
\Theta^n_s:=\sum_{j \in\ES}\bigl|\bigl(\u^{N,n}-\v^{N,n}\bigr)\bigl(s,(j,\wt{\bm X}^{\sqbullet,-n}_{s}),\wt{\bm Y}^n_s\bigr)\bigr|,
\end{equation}
where $(j,\wt{\bm X}^{\sqbullet,-n}_{s})=(\wt{X}^{1,n}_{s},\dots,\wt{X}^{n-1,n}_{s},j,\wt{X}^{n+1,n}_{s},\dots)$.
Finally, from Proposition \ref{prop:u} (with $\epsilon=1/8$ therein),
\be \notag
\bigl| r^{N,n}\bigl(s, \wt{\bm X}^n_s, \wt{\bm Y}^n_s\bigr) \bigr| \leq \frac{C}{N^\eta}.
\ee
%
%
%
%
%
Using Young's inequality $ab\le(a^2+b^2)/2$
in 
\eqref{eq:conclusion:third:step:2}
and choosing 
\begin{equation}
\label{eq:definition:e}
e_{s} = \exp \biggl(\int_{t_{0}}^t \frac{c}{\mumin_s}
ds \biggr), 
\end{equation} 
we obtain 
\begin{equation}
\label{bound_exp_ul-vl}
\begin{split}
\E \bigl[e_{\tilde{\tau}_{N}}
\Phi^n_{\tilde{\tau}_{N}} \bigl(\u^n_{\tilde \tau_{N}}-\v^n_{\tilde \tau_{N}}\bigr)^2
\bigr]
&\geq
\E \bigl[ e_{t}
\Phi^n_{t} \bigl(\u^n_t-\v^n_t)^2\bigr] 
\\
&\hspace{-35pt} - C \E \biggl[\int_{t\wedge \tilde \tau_N}^{\tilde \tau_{N}} e_{s} \Phi^n_{s} \bigl( \Theta^n_s \bigr)^2 
ds \biggr]
- \frac{C}{N^{2\epsilon}} {\mathbb E} \biggl[ \int_{t\wedge \tilde \tau_N}^{\tilde{\tau}_{N}}
e_{s} ds\biggr] - \frac{C}{N^{2 \eta}}.
\end{split}
\end{equation}
Here, it is worth recalling that 
$c$ in 
\eqref{eq:definition:e}
only depends on 
$T$, $\| f \|_{\infty}$
and $\| g \|_{\infty}$.
\vskip 4pt

\noindent \textit{Fifth Step.}
We now proceed with the analysis of the various terms on the right-hand side of 
\eqref{bound_exp_ul-vl}. 
We start with the first term. 
By
\eqref{bound_calE}, 
\eqref{boundtau_magenta}, 
and 
\eqref{boundY_magenta:-1} 
and for $\kappa$ large enough (we comment more on this requirement right after the inequality), we observe that\footnote{Note that, even in case $\tilde\tau_N=t_0$ (meaning that the initial condition is outside the domain underpinning the definition of 
$\tilde \tau_{N}$ in \eqref{deftilde:bb}), the bound from \eqref{boundtau_magenta} still holds, which implies \eqref{eq:6thstep:term0} in this case.  } 
(using the fact that 
$\Phi^n_{\tilde{\tau}_{N}} \leq 1$)
\begin{equation}
\label{eq:6thstep:term0}
\begin{split}
&\E \bigl[e_{\tilde{\tau}_{N}}
\Phi^n_{\tilde{\tau}_{N}} \bigl(\u^n_{\tilde \tau_{N}}-\v^n_{\tilde \tau_{N}}\bigr)^2
\bigr] 
\\
&\leq C 
\E \bigl[ \bigl( e_{\tilde{\tau}_{N}} \bigr)^4 \bigr]^{1/4}
\E \bigl[
\bigl( 
 \Phi^n_{\tilde{\tau}_{N}} \bigr)^2 \bigr]^{1/2}
 {\mathbb P} \bigl( \tilde \tau_{N} < T\bigr)^{1/4}
\\
&\leq 
C 
N^{-\epsilon/(4d)}
\prod_{i \in \ES}
\bigl(
N^{-\epsilon}+\wt \mu_{t_{0}}[i]
\bigr)^{-5/(8d)}
\biggl[\biggl(\frac1N
 \sum_{l \in \EN} \vert y_{0}^l \vert^{\ell} \biggr)^{-1/4}
 + \exp(-c N^{1-2\epsilon}) \biggr] 
 \\
 &\hspace{15pt}
 \times 
 \biggl[ 
\prod_{i \in \ES}
\bigl(
N^{-\epsilon}+\wt \mu_{t_{0}}[i] \bigr)^{-1/(8d)} +
\biggl( \frac{1}{N}\sum_{l\in\NN} |y^l_0|^{\ell} \biggr)^{1/8}
\biggr].
\end{split}
\end{equation}
Notice that, in order to get an estimate for 
 $\E [ ( e_{\tilde{\tau}_{N}} )^4 ]^{1/4}$ (using \eqref{bound_calE}), we must assume that 
   $\kappa$ is large enough
   with respect to $c$, but this makes sense since 
the constant $c$ in the definition of $e_{s}$, see \eqref{eq:definition:e}, 
depends only on $\|f \|_{\infty}$, $\|g \|_{\infty}$
and $T$ and hence
does not depend on 
$\kappa$. It also interesting to observe that the polynomial decay in ${\boldsymbol y}$ in the penultimate line permits to 
balance the polynomial growth in the last line. 

If
$(N^{-1} \sum_{l \in \EN} \vert y_{0}^l \vert^{\ell} )^{-1/4}
 \geq \exp(-c N^{1-2\epsilon})$, then, using the fact that 
$(N^{-1} \sum_{l \in \EN} \vert y_{0}^l \vert^{\ell} )^{-1/4}$
is always less than 1,
we get that  
  the right-hand side in \eqref{eq:6thstep:term0} can be bounded by 
 $C 
N^{-\epsilon/(4d)}
\prod_{i \in \ES}
(
N^{-\epsilon}+\wt \mu_{t_{0}}[i])^{-3/(4d)}$, that is 
\begin{equation}
\label{eq:6thstep:term0:2}
\begin{split}
&\E \bigl[e_{\tilde{\tau}_{N}}
\Phi^n_{\tilde{\tau}_{N}} \bigl(\u^n_{\tilde \tau_{N}}-\v^n_{\tilde \tau_{N}}\bigr)^2
\bigr] 
\leq C 
N^{-\epsilon/(4d)}
\prod_{i \in \ES}
\bigl(
N^{-\epsilon}+\wt \mu_{t_{0}}[i]
\bigr)^{-3/(4d)}.
\end{split}
\end{equation}
If 
$ (N^{-1} \sum_{l \in \EN} \vert y_{0}^l \vert^{\ell} )^{-1/4}
 \leq \exp(-c N^{1-2\epsilon})$, that is 
 $N^{-1}
 \sum_{l \in \EN} \vert y_{0}^l \vert^{\ell}
 \geq \exp(4c N^{1-2\epsilon})$, then $\tilde \tau_{N}=t_{0}$ with probability 1.  
 But, in the latter case, 
 \begin{equation*}
 \E \bigl[e_{\tilde{\tau}_{N}}
\Phi^n_{\tilde{\tau}_{N}} \bigl(\u^n_{\tilde \tau_{N}}-\v^n_{\tilde \tau_{N}}\bigr)^2
\bigr] \leq C \Phi^n_{t_{0}}
= C 
 \biggl(\frac1N \sum_{l \in \EN} \vert y_{0}^l \vert^{\ell} \biggr)^{-1}
 \leq C 
 \exp(-4c N^{1-2\epsilon}),
 \end{equation*}
 and obviously \eqref{eq:6thstep:term0:2} also holds true. 
 So, the latter holds true in any case. 
Similarly,  by 
\eqref{bound_calE}
again and also for $\kappa$ large enough as before,
\begin{equation}
\label{eq:6thstep:term0:3}
\begin{split}
N^{-2\epsilon} {\mathbb E} \biggl[ \int_{t\wedge \tilde \tau_N}^{\tilde{\tau}_{N}}
e_{s} ds\biggr] 
&\leq C 
N^{-2\epsilon}\prod_{i \in \ES}
\bigl(
N^{-\epsilon}+\wt \mu_{t_{0}}[i]
\bigr)^{-1/d}
\\
&\leq C 
N^{-\epsilon/(4d)}
\prod_{i \in \ES}
\bigl(
N^{-\epsilon}+\wt \mu_{t_{0}}[i]
\bigr)^{-3/(4d)}.
\end{split}
\end{equation}
To finish with, we 
recall from 
\eqref{eq:thetan:t}
that
\begin{equation}
\theta^n_{t}
=
\sup_{\bm{x},\bm{y}}
\Bigl[ \prod_{i \in \ES}
\bigl(
N^{-\epsilon}+\mu_{\bm{x},\bm{y}}^N[i]
\bigr)^{1/d} \Phi^N(\bm{y}) \bigl(z^{N,n}-w^{N,n} \bigr)^2(t,\bm{x},\bm{y}) 
\Bigr],
\quad t \in [0,T],
\end{equation}
This allows us to 
write, using the notation from 
\eqref{eq:Theta:ns},
\begin{equation*}
\begin{split}
\Phi^n_{s} \bigl( \Theta^n_{s} \bigr)^2 \leq
C \sum_{j \in \ES}
\prod_{i \in \ES}
\Bigl(
N^{-\epsilon}+\mu_{
(j,\wt{\bm X}^{\sqbullet,-n}_{s}),\wt{\bm Y}^{\sqbullet,n}_s}^N[i]
\Bigr)^{-1/d}
\theta^n_{s}. 
\end{split}
\end{equation*}
Now, since $s < \tilde{\tau}_{N}$, we have, for all $i \in \ES$, 
\begin{equation*}
N^{-\epsilon}
+
\mu_{
(j,\wt{\bm X}^{\sqbullet,-n}_{s}),\wt{\bm Y}^{\sqbullet,n}_{s}}^N[i]
\geq 
N^{-\epsilon} + \mu_{
\wt{\bm X}^{\sqbullet,n}_{s},\wt{\bm Y}^{\sqbullet,n}_{s}}^N[i]
- \frac{Y_{t}^{n,n}}{N}
\geq 
\frac12 \Bigl( 
N^{-\epsilon} + 
\widetilde \mu^{\sqbullet,n}_s[i] \Bigr), 
\end{equation*}
from which we deduce that 
\begin{equation*}
\begin{split}
&\E \biggl[\int_{t\wedge \tilde \tau_N}^{\tilde \tau_{N}} e_{s} \Phi^n_{s} \bigl( \Theta^n_s \bigr)^2 
ds \biggr]
\leq 
C \E \biggl[\int_{t\wedge \tilde \tau_N}^{\tilde \tau_{N}} 
\theta^n_{s} e_{s} 
\prod_{i\in\ES} 
\bigl(
N^{-\epsilon}+\wt{\mu}_s^{\sqbullet,N}[i]
\bigr)^{-1/d}
ds \biggr].
\end{split}
\end{equation*}
By 
\eqref{eq1/muexp}, the above is less than
(again, provided that $\kappa$ is chosen large enough)  
\begin{equation}
\label{eq:6thstep:term0:4}
\begin{split}
&\E \biggl[\int_{t\wedge \tilde \tau_N}^{\tilde \tau_{N}} e_{s} \Phi^n_{s} \bigl( \Theta^n_s \bigr)^2 
ds \biggr]
\leq 
C 
\prod_{i \in \ES}
\bigl(
N^{-\epsilon}+\wt \mu_{t_{0}}[i]
\bigr)^{-1/d}
\int_{t}^{T} 
\theta^n_{s} 
ds.
\end{split}
\end{equation}
So,
taking $t=t_{0}$ in  
\eqref{bound_exp_ul-vl}, 
multiplying 
the whole 
by 
$\prod_{i \in \ES}
(
N^{-\epsilon}+\wt \mu_{t_{0}}[i]
)^{1/d}$ and 
collecting the four bounds 
\eqref{bound_exp_ul-vl},
\eqref{eq:6thstep:term0:2}, 
\eqref{eq:6thstep:term0:3}
and 
\eqref{eq:6thstep:term0:4}, 
we deduce that there exists an exponent $\chi >0$ such that, for any initial condition 
$(t_{0},\bm{x},\bm{y})$, 
\begin{equation*}
\prod_{i \in \ES}
\bigl(
N^{-\epsilon}+\wt \mu_{t_{0}}[i]
\bigr)^{1/d}
 \Phi(\bm{y}) \bigl(z^{N,n}-w^{N,n} \bigr)^2(t_{0},\bm{x},\bm{y}) 
 \leq C N^{-\chi} + 
C 
\int_{t_{0}}^{T} 
\theta^n_{s} 
ds.
\end{equation*}
Taking the supremum over $\bm{x}$ and $\bm{y}$, we
get $\theta^n_{t_{0}}$ on the left-hand side. 
Since $t_{0}$ is arbitrary, we can  complete the proof 
of \eqref{bound1}
by Gronwall's lemma.
\qed
%
%
%

\subsection{Proof of Theorem \ref{thm:convergence}}

Now that Theorem 
\ref{thm:convergence_value}
has been proved, 
the proof is pretty straightforward. It is based on a standard diffusion approximation theorem, 
see for instance \cite[Chapter 7, Theorem 4.1]{EthierKurtz}, 
and  
our plan is thus to check the assumption of this latter statement. 
Throughout the proof, we make use of the following stopping time
\begin{equation*}
\sigma_{K}^N := 
\inf \Bigl\{ t \in [0,T] : \min_{e \in \ES} \mu^N_{t}[e] \leq 1/K \ \text{or}
\ \frac1N \sum_{l \in \ES} \vert Y_{s}^{N,l} \vert^{\ell} \geq K \Bigr\},
\quad \inf \emptyset = +\infty,
\end{equation*}
for $\ell=3$ and any real $K \geq 0$. 
Above, we denoted by $Y^{N,l}$ the $l$th coordinate of ${\bm Y}^N$.
A key fact for us is that, by
\eqref{eq1/muexp}
and
\eqref{boundY_magenta}, with $\iota=0$ therein,
\begin{equation*}
\lim_{K \rightarrow \infty} \sup_{N \geq 1}
{\mathbb P} \Bigl( \sigma^N_{K} \leq T \Bigr) = 0. 
\end{equation*}
Note that the initial condition satisfies \eqref{xNyN} because of the convergence in \eqref{conv:init} and because 
$\bm{y}^N=(1,\dots,1)$.
 The above result is important in our analysis. It says that we can easily localize
the various conditions appearing in the statement of 
 \cite[Chapter 7, Theorem 4.1]{EthierKurtz}
 and just check them up to the stopping time $\sigma^N_{K}$
 (which is not so different from the fact that, in the latter statement, 
 all the conditions are localized with respect to the stopping $\tau_{n}^r$, 
 using the same notation as therein).

The next step is to provide a semi-martingale expansion of 
$(\mu^N_{t}[i])_{0 \leq t \leq T}$, for any $i \in \ES$. 
This is here possible by It\^o's lemma (see for instance
the proof of 
Theorem \ref{thmexpbound}
or expansion \eqref{ItoPsi}). For convenience, we remove the exponent $N$ from most of the notations 
(for instance we merely write $({\boldsymbol X},{\boldsymbol Y})$
for $({\boldsymbol X}^N,{\boldsymbol Y}^N)$), except when 
this is clearly needed (say for instance when we take a limit over $N$). 
With this convention,  we get 
\begin{align*}
d \mu^N_{t}[i] = d \biggl( \frac1N \sum_{l \in \NN} Y_{t}^l {\mathbbm 1}_{\{X_{t}^l=i\}}
\biggr)
&= \frac1N \sum_{l \in \NN} 
\int_{[0,M]^d}
Y_{t-}^l \bigl( {\mathbbm 1}_{\{X_{t-}^l + f^l_{t-}(\theta)=i\}}
-
{\mathbbm 1}_{\{X_{t-}^l =i\}}
\bigr) {\mathcal N}^l(d \theta,dt)
\\
&\hspace{15pt}
+  \frac1N \sum_{l \in \NN} 
\int_{[0,\varepsilon N]^{\NN^d}}
g^l_{t-}(\theta) 
{\mathbbm 1}_{\{X_{t-}^l =i\}} {\mathcal N}^0(d \theta,dt)
\\
&= \frac1N \sum_{l \in \NN} 
\int_{[0,M]^d}
Y_{t-}^l \bigl( {\mathbbm 1}_{\{X_{t-}^l + f^l_{t-}(\theta)=i\}}
-
{\mathbbm 1}_{\{X_{t-}^l =i\}}
\bigr) \nu^l(d \theta) dt
\\
&\hspace{15pt}+ \frac1N \sum_{l \in \NN} 
\int_{[0,M]^d}
Y_{t-}^l \bigl( {\mathbbm 1}_{\{X_{t-}^l + f^l_{t-}(\theta)=i\}}
-
{\mathbbm 1}_{\{X_{t-}^l =i\}}
\bigr) \overline{\mathcal N}^l(d \theta,dt)
\\
&\hspace{15pt}
+  \frac1N \sum_{l \in \NN} 
\int_{[0,\varepsilon N]^{\NN^d}}
g^l_{t-}(\theta) 
{\mathbbm 1}_{\{X_{t-}^l =i\}} \overline{\mathcal N}^0(d \theta,dt)
\\
&=: dB_{t} + d M_{t} + dM^0_{t}, 
\end{align*}
where $\overline{\mathcal N}^0, \overline{\mathcal N}^1,\cdots, \overline{\mathcal N}^N$ denote
the compensated Poisson measures and where, in the penultimate line, we used the fact that 
$\int_{[0,\varepsilon N]^{\NN^d}}
g^l_{t-}(\theta) d \nu^0(\theta)=0$. 
By 
\eqref{integral}, we know that 
\begin{equation*}
\begin{split}
dB_{t} &=
 \frac1N \sum_{l \in \NN}
\sum_{j \in \ES} Y_{t-}^l  \Bigl( \delta_{i,j}
-
{\mathbbm 1}_{\{X_{t-}^l =i\}} \Bigr) \Bigl( \varphi \bigl( \mu^N_{t-}[j] \bigr) 
+ \bigl( - \Delta^l w^{l}_{t-}[j] \bigr)_{+} \Bigr) dt
\\
&= 
 \frac1N \sum_{l \in \NN}
 Y_{t-}^l  
\Bigl[
\Bigl( \varphi \bigl( \mu^N_{t-}[i] \bigr) 
+ \bigl( - \Delta^l w^{l}_{t-}[i] \bigr)_{+} \Bigr)
-
\sum_{j \in \ES} {\mathbbm 1}_{\{X_{t-}^l =i\}} 
\Bigl( \varphi \bigl( \mu^N_{t-}[j] \bigr) 
+ \bigl( - \Delta^l w^{l}_{t-}[j] \bigr)_{+} \Bigr)
\Bigr]  dt. 
\end{split}
\end{equation*}
As long as $t < \sigma_{K}$, we have that
\begin{equation}
\label{eq:convergence:bound:by:K:for:Yl}
\frac{Y_{t-}^{l}}{N} \leq \frac1N N^{1/m}
\biggl( \frac1N \sum_{n \in \NN}
\vert Y_{t-}^n \vert^{m} \biggr)^{1/m} \leq K^{1/m} N^{-1+1/m}.
\end{equation}Hence, we get 
from \eqref{deltall}, \eqref{eq:conclusion:lem:exp:first:order:variation} and  
from Theorem \ref{thm:convergence_value},
\begin{equation*}
\begin{split} 
\Delta^l w^{l}_{t-}[j]
&= \Delta^l z^{l}_{t-}[j] + O \bigl( N^{-\chi}
\bigr) 
\\
&= 
U^j(t-,\mu^N_{t-})
- 
U^{X_{t-}^l}(t-,\mu^N_{t-}) + O \bigl( N^{-\chi}
\bigr),
\end{split}
\end{equation*}
for a possibly new value of $\chi$. Above, $O(\cdot)$ is the standard 
Landau notation, it being understood that the underlying constant 
is deterministic and independent of $l$, $t$ and $j$. 

We thus end up with 
\begin{equation}
\label{eq:B:b}
\begin{split}
dB_{t}&= 
\sum_{k \in \ES} \mu_{t-}^N[k]
\Bigl( \varphi \bigl( \mu^N_{t-}[i] \bigr) 
+ \bigl( U^k - U^j
 \bigr)_{+}(t,\mu_{t-}^N) \Bigr)
\\
&\hspace{15pt}-
\sum_{j \in \ES}
\mu_{t-}^N[i]
\Bigl( \varphi \bigl( \mu^N_{t-}[j] \bigr) 
+ \bigl( U^i - U^j
  \bigr)_{+}(t,\mu_{t-}^N) \Bigr) 
\Bigr] dt + O \bigl( N^{-\chi} \bigr) dt
\\
&=: b(t, \mu_{t-}^N) dt + O \bigl( N^{-\chi} \bigr) dt
\end{split}
\end{equation}
We deduce that (we put an exponent $N$ in $B$ in order to emphasize the dependence on $N$) 
\begin{equation*}
\sup_{t \in [0,T \wedge \sigma_{K}]}
\biggl\vert B_{t}^N 
- \int_{0}^t b(s,\mu_{s}^N) ds \biggr\vert 
\end{equation*}
tends to $0$ in probability as $N$ tends to $\infty$, which fits \cite[Chapter 7, (4.6)]{EthierKurtz}. 
By the way,  notice that 
\cite[Chapter 7, (4.4)]{EthierKurtz}
follows in a straightforward manner. 

We now handle the martingale part. 
By independence of the noises ${\mathcal N}^1,\cdots,{\mathcal N}^N$, it is easy to see that 
\begin{equation*}
\lim_{N \rightarrow \infty}
{\mathbb E} \bigl[ \sup_{0 \le t \leq \sigma_{K}} \vert M_{t} \vert^2 \bigr] = 0. 
\end{equation*}
Then, the compensator of $(M_{t}[i]M_{t}[j])_{0 \le t \le T}$, which we denote by $(A^{i,j}_{t})_{0 \le t \le T}$, satisfies 
\begin{equation*}
\lim_{N \rightarrow \infty}
{\mathbb E} \bigl[ \sup_{0 \le t \leq \sigma_{K}} \vert A_t\vert \bigr] = 0. 
\end{equation*}
Next, the compensator of 
$(M^0_{t}[i]M^0_{t}[j])_{0 \le t \le T}$ is given by 
$(A^{0,i,j}_{t})_{0 \le t \le T}$, defined by 
\begin{equation*}
\begin{split}
dA_{t}^{0,i,j} &= \frac1{N^2} \sum_{l,n \in \NN}
\int_{[0,\varepsilon N]^{\NN^d}}
g_{t-}^l(\theta) 
g_{t-}^n(\theta) {\mathbbm 1}_{\{ X_{t-}^l = i\}}
{\mathbbm 1}_{\{ X_{t-}^n = j\}}
\nu^0(d\theta) dt
\\
&=\frac{\varepsilon}{N} \sum_{l,n \in \NN}
\biggl(
Y^l_{t-} Y^n_{t-}
{\mathbbm 1}_{\{ X_{t-}^l = i\}}
{\mathbbm 1}_{\{ X_{t-}^n = j\}}
{\mathbf E}
\biggl[\biggl( \frac{S_{\mu^N_{t-}}[i]}{N \mu^N_{t-}[i]} - 1
\biggr)
\biggl( \frac{S_{\mu^N_{t-}}[j]}{N \mu^N_{t-}[j]} - 1
\biggr) \biggr] \biggr) dt
\\
&= \varepsilon \Bigl(   \mu_{t-}^N[i] \delta_{i,j}
- \mu_{t-}^N[i] \mu_{t-}^N[j] \Bigr) dt. 
\end{split}
\end{equation*}
By combining the last two results, we deduce that 
the compensator of $((M_t+M^0_{t})[i](M_t+M^0_{t})[j])_{0 \le t \le T}$, which we denote by 
$(\bar A^{i,j}_{t})_{0 \le t \le T}$, satisfies
\begin{align*}
&\lim_{N\to\infty}\E\Big[\sup_{t\in[0,T\wedge\sigma_K]}|\bar A^{i,j}_t-\bar A^{i,j}_{t-}|\Big]=0,\\ 
&\lim_{N\to\infty}\E\Big[\sup_{t\in[0,T\wedge\sigma_K]}|\bar A^{i,j}_t-\int_0^t \varepsilon \Bigl(   \mu_{s-}^N[i] \delta_{i,j}
- \mu_{s-}^N[i] \mu_{s-}^N[j] \Bigr) ds|\Big]=0,
\end{align*}
which are, respectively, 
\cite[Chapter 7, (4.5)]{EthierKurtz}
and
\cite[Chapter 7, (4.7)]{EthierKurtz}. 
The last assumption we need to verify is 
\cite[Chapter 7, (4.3)]{EthierKurtz}.
Since the process 
$(\mu_{t}^N)_{0 \leq t \leq T}$
takes values in the simplex, it suffices to prove that
\begin{equation}
\label{eq:convergence:jumps}
\forall r>0, 
\quad \lim_{N \rightarrow \infty}
{\mathbb P}
\biggl( \sup_{0 \leq t \leq T \wedge \sigma_{K}^N}
\vert \mu_{t}^N - \mu_{t-}^N \vert 
\geq r \biggr) = 0. 
\end{equation}
We may split the jumps into two parts: Those that are due to 
the idiosyncratic noises
and those that are due to the common noise. 
To make it clear, for any $i \in \ES$, 
\begin{equation*}
\begin{split}
\vert \mu_{t}^N[i] - \mu_{t-}^N[i] \vert
\leq \frac1N \sum_{l \in \NN} \vert Y_{t}^l - Y_{t-}^{l}
\vert +
 \frac1N \sum_{l \in \NN} Y_{t-} \vert X_{t}^l - X_{t-}^{l} \vert.
\end{split}
\end{equation*}
Since at most one of all the $(X^l)_{l \in \NN}$ may
jump at a given time, the second term on the right-hand side can be upper bounded in the following way:
\begin{equation*}
 \frac1N \sum_{l \in \NN} Y_{t-} \vert X_{t}^l - X_{t-}^{l} \vert
 \leq \frac{d}N \max_{l \in \NN} Y_{t-}^l 
 \leq 
  K^{1/m} N^{-1+1/m},
\end{equation*}
 the last bound following from 
\eqref{eq:convergence:bound:by:K:for:Yl}.

Therefore, 
we can just focus on the jumps induced by ${\boldsymbol Y}$, which is more subtle.  
The idea is to 
represent the latter ones as follows. 
By 
\eqref{eq:markov:transition:2}, 
we may indeed 
represent the 
jump times of $({\boldsymbol Y}_{t})_{0 \le t \le T}$
through the 
jump times 
$(\varrho_{n})_{n \geq 0}$
of a Poisson process of intensity $\varepsilon N$
on the axis $[0,+\infty)$ (with $\varrho_{0}=0$). The counting process
on $[0,T]$
is denoted by  
$(R_{t} = \sum_{n\geq 0} {\mathbbm 1}_{[0,t]}(\tau_{n}))_{0 \leq t \leq T}$, which
is a Poisson process 
of intensity $\varepsilon N$. 
Then, we use the fact that, when the exponential clock rings (namely at 
some time $\varrho_{n}$), 
the jump of ${\boldsymbol Y}$ is given by 
a multinomial distribution of parameters 
$N$ and $(\mu^N_{\varrho_{n}-}[i])_{i \in \ES}$. 
 Writing
\begin{equation*}
\sup_{0 \leq t \leq T \wedge \sigma_{K}}
\frac1N \sum_{l \in \NN}
\vert Y_{t}^l - Y_{t-}^l \vert
\leq \sup_{n \in \{1,\cdots, R_{T}\}}
\biggl[ \frac1N \sum_{l \in \NN}
\vert Y_{\varrho_{n}}^l - Y^l_{\varrho_{n}-} \vert
{\mathbbm 1}_{\{ \min_{e \in \ES} \mu^N_{\varrho_{n}-}[e] >1/K\}} \biggr],
\end{equation*}
we deduce that, for any $r >0$, 
\begin{equation*}
\begin{split}
&{\mathbb P}
\Bigl( 
\sup_{0 \leq t \leq T \wedge \sigma_{K}}
\frac1N \sum_{l \in \NN}
\bigl\vert Y_{t}^l - Y_{t-}^l \bigr\vert
\geq r \Bigr)
\\
&\leq {\mathbb P}
\Bigl( R_{T} \geq N r^{-1} \Bigr)
+
{\mathbb P}
\biggl( 
\sup_{n \in \{1,\cdots, \lfloor Nr^{-1} \rfloor \}}
\biggl[ \frac1N \sum_{l \in \NN}
\bigl\vert Y_{\varrho_{n}}^l - Y^l_{\varrho_{n}-} \bigr\vert
{\mathbbm 1}_{\{ \min_{e \in \ES} \mu^N_{\varrho_{n}-}[e] >1/K\}} \biggr]
\geq r \biggr)
\\
&\leq \varepsilon T r + Nr^{-1} \sup_{\mu : 
\min_{e \in \ES} \mu[e] >1/K} {\mathbb P} \biggl( \sup_{i \in \ES} \Bigl\vert \frac{S_{\mu}[i]}{N}
- \mu[i] \Bigr\vert \geq r \mu[i] \biggr) 
\\
&\leq \varepsilon T r + 2Nr^{-1} \exp \bigl( - 2 N r^2 K^{-2} \bigr), 
\end{split}
\end{equation*}
the last line following from Hoeffding's inequality. The above bound tends to $0$ as 
$N$ tends to $\infty$ first and then $r$ tends to $0$, from which we deduce that, 
for any $r>0$, the term on the first line tends to $0$ as $N$ tends to $\infty$. 
We get 
\eqref{eq:convergence:jumps}, which completes the proof. 
\qed

\section{Proofs of the auxiliary estimates of the weight process}
\label{sec:6:b}
In this section, we prove the results stated in Subsection \ref{sec:4}.
\subsection{Proof of Theorem \ref{thmexpbound}}
\label{subse:6:1}
{ \ }
\vskip 5pt

\noindent \textit{First Step.}
Recall It\^o's formula
\eqref{eq:ito:generic}
together with the notation
\eqref{eq:overline:f:overline:g}.
Given $\varpi>0$, we apply this formula to the function $v(\bm{x},\bm{y})=\log(\varpi +   \mu^N_{\bm{x},\bm{y}}[i])$. 
We get, for each $i\in\dd$, 
letting 
$\overline\Lambda_t :=
\log \bigl(\varpi + \overline\mu_t[i]\bigr) $, 
\begin{align*}
&d \overline{\Lambda}_{t} 
= \sum_{l\in\NN} \int_{[0,M]^d} 
\biggl[ \log\left(\varpi + \overline \mu_{t-}[i] + \frac1N \overline Y^{l}_{t-} \Big(\one_{\{\overline X^{l}_{t-}+ \overline f_{t-}^{l}(\theta)=i\}} - \one_{\{\overline X^{l}_{t-}=i\}}\Big)\right)
\\
&\hspace{15pt} -\log\bigl(\varpi + \overline \mu_{t-}[i]\bigr) \biggr] \N^l(d\theta,dt)
 \\
&\quad+ \int_{[0,\varepsilon N]^{\NN^d}} 
\biggl[
\log\bigg(\varpi + \overline\mu_{t-}[i] + \frac1N  \sum_{l\in\NN} \overline g^{l}_{t-}(\theta)  \one_{\{\overline X^{l}_{t-}=i\}}\bigg)-\log\bigl(\varpi + \overline \mu_{t-}[i]\bigr)
\biggr]
 \N^0(d\theta,dt)
 \\
 &= \sum_{l\in\NN} \int_{[0,M]^d} 
\log\Biggl(1+\frac{ \frac1N \overline Y^{l}_{t-} \Big[\one_{\{\overline X^{l}_{t-}+ \overline f^{l}_{t-}(\theta)=i\}} - \one_{\{\overline X^{l}_{t-}=i\}}\Big]}{\varpi + \overline \mu_{t-}[i]}\Biggr) \N^l(d\theta,dt)
 \\
&\quad+ \int_{[0,\varepsilon N]^{\NN^d}} 
\log\Bigg(1 + \frac{\frac1N  \sum_{l\in\NN} \overline g^{l}_{t-}(\theta)  \one_{\{\overline X^{l}_{t-}=i\}}}{\varpi + \overline\mu_{t-}[i]}\Bigg)
 \N^0(d\theta,dt),
\end{align*}
and we have
\begin{align*}
-d\overline\Lambda_t
 &= \sum_{l\in\NN} \int_{[0,M]^d} 
\log\Bigg(\frac{ \varpi + \overline\mu_{t-}[i] }{\varpi + \overline \mu_{t-}[i]+ \frac1N \overline Y^{l}_{t-} \big[\one_{\{\overline X^{l}_{t-}+ \overline f^{l}_{t-}(\theta)=i\}} - \one_{\{ \overline X^{l}_{t-}=i\}}\big]}\Bigg) \N^l(d\theta,dt)
 \\
&\quad+ \int_{[0,\varepsilon N]^{\NN^d}} 
\log\Bigg(\frac{ \varpi + \overline \mu_{t-}[i] }{\varpi + \overline \mu_{t-}[i] + \frac1N  \sum_{l\in\NN} \overline g^{l}_{t-}(\theta)  \one_{\{\overline X^{l}_{t-}=i\}}}\Bigg)
 \N^0(d\theta,dt)
 \\
 &= \sum_{l\in\NN} \int_{[0,M]^d} 
\log\Bigg(1-\frac{ \frac1N \overline Y^{l}_{t-} \big[ \one_{\{\overline X^{l}_{t-}+ \overline f^{l}_{t-}(\theta)=i\}} - \one_{\{\overline X^{l}_{t-}=i\}}\big] }{\varpi + \overline\mu_{t-}[i] + \frac1N {\overline Y^{l}_{t-}} \big[\one_{\{\overline X^{l}_{t-}+ \overline f^{l}_{t-}(\theta)=i\}} - \one_{\{X^{l}_{t-}=i\}}\big]}\Bigg) \N^l(d\theta,dt)
 \\
&\quad+ \int_{[0,\varepsilon N]^{\NN^d}} 
\log\Bigg(1- \frac{  \frac1N  \sum_{l\in\NN} {\overline g^{l}_{t-}(\theta)}  \one_{\{\overline X^{l}_{t-}=i\}} }{\varpi + \overline\mu_{t-}[i] + \frac1N  \sum_{l\in\NN} {\overline g^{l}_{t-}(\theta)}  \one_{\{\overline X^{l}_{t-}=i\}}}\Bigg)
 \N^0(d\theta,dt).
\end{align*}
This prompts us to let 
\begin{align*}
\overline\lambda_t &:= \sum_{l\in\NN} \int_{[0,M]^d} 
-\frac{ \frac1N {\overline Y^{l}_{t}} \big[\one_{\{\overline X^{l}_{t-}+ \overline f^{l}_{t}(\theta)=i\}} - \one_{\{\overline X^{l}_t=i\}}\big] }{\varpi + \overline \mu_{t-}[i]+ \frac1N {\overline Y^{l}_{t-}} \big[\one_{\{\overline X^{l}_{t}+ \overline f^{l}_{t}(\theta)=i\}} - \one_{\{\overline X^{l}_t=i\}}\big]}
 \nu(d\theta) 
 \\
&\quad+ \int_{[0,\varepsilon N]^{\NN^d}} 
- \frac{  \frac1N  \sum_{l\in\NN} {\overline g^{l}_{t}(\theta)}  \one_{\{\overline X^{l}_t=i\}} }{\varpi + \overline \mu_t[i] + \frac1N  \sum_{l\in\NN} {\overline g^{l}_{t}(\theta)}  \one_{\{\overline X^{l}_t=i\}}}
 \nu^0(d\theta),
\end{align*}
As a result, we claim that $(Z_t= \exp \{ -\overline \Lambda_t - \int_{0}^t \overline\lambda_s ds \})_{0\leq t\leq T}$ is a supermartingale. This follows from the fact that, given a tuple of bounded predictable processes $((F_{t}^l)_{0 \leq t \leq T})_{l=0,\cdots,N}$, the process 
\begin{equation*}
\begin{split}
Z_t &= \exp \Biggl\{ 
\sum_{l=1}^N \int_0^t \int_{[0,M]^d}\log (1+  F^l_{t}) \N^l(d\theta,dt) +
\int_0^t \int_{[0,\varepsilon N]^{\NN^d}}\log (1+  F^0_{t}) \N^0(d\theta,dt)
\\
&\hspace{45pt} -  
\sum_{l=1}^N 
 \int_0^t \int_{[0,M]^d} F_{t}^l \nu(d\theta) dt 
- 
 \int_0^t \int_{[0,\varepsilon N]^{\NN^d}}  F^0_{t} \nu^0(d\theta)dt
 \Biggr\},
\end{split}
\end{equation*}
is a supermartingale since it solves the SDE
\[ 
d Z_t = \sum_{l=1}^N 
\int_{[0,M]^d} Z_{t-} F_{t}^l \big[ {\mathcal N}^l(d\theta, dt) - \nu(d\theta) dt\big]
+ 
\int_{[0,\varepsilon N]^{\NN^d}} Z_{t-} F_{t}^0 \big[ {\mathcal N}^0(d\theta, dt) - \nu^0(d\theta) dt\big].
\]
 In particular, 
 $\E[Z_t] \leq (\varpi+\bar\mu^i_0)^{-1}$ for any $t$. 
\vskip 5pt

\noindent \textit{Second Step.}
By definition of $\Lambda$
\be 
\label{exp}
(\varpi+ \overline\mu_{t\wedge \overline\tau_N}[i])\exp\bigg( 
- \overline{\Lambda}_{t\wedge \overline\tau_N}
-
\int_0^{t\wedge \overline\tau_N} \overline\lambda_s ds\bigg) = \exp\bigg( \int_0^{t\wedge \overline \tau_N} -\overline
\lambda_s ds \bigg),
\ee
and then, using $0<\varpi<1$ and $\overline\mu_{\overline \tau_N}[i] \leq 1$, taking expectation and letting 
$t$ tend to $T$, we obtain
\be 
\label{eqexp}
\E \bigg[ \exp\bigg\{ - \int_0^{ \overline \tau_N} \overline \lambda_t dt \bigg\}\bigg] \leq \frac{2}{\varpi + 
\overline \mu_0[i]},
\ee
thus it remains to bound $-\overline\lambda_t$ from below. By \eqref{eq:intensity:measures}, we obtain, letting 
$\alpha_{t}^l(j) :=
\alpha^l(t,\bm{\overline X}_{t},\bm{\overline Y}_{t})[j]$,
\begin{align}
-\overline \lambda_t &= \frac 1N \sum_{l\in\NN} \sum_{j\in\dd}  
\bigl[\varphi\bigl(\overline \mu_t[j]\bigr) + \alpha_{t}^{l}[j]\bigr]\frac{  {\overline Y^{l}_t} \big(\one_{\{j=i\}} - \one_{\{\overline X^{l}_t=i\}}\big) }{\varpi + \overline \mu_t[i] + \frac1N {\overline Y^{l}_t} \big(\one_{\{j=i\}} - \one_{\{\overline X^{l}_t=i\}}\big)} \nonumber
 \\
&\quad+ \varepsilon N  \sum_{{\boldsymbol k}\in \NN^d} 
\Bigl( {\frac{k^{\overline X^{n}_t}}{N\overline\mu_t[\overline X^{n}_t]}}
\Bigr)^{\iota}
{\mathcal M}_{N,\overline \mu_{t}}({\boldsymbol k}) 
\frac{  \frac1N  \sum_{l\in\NN} \overline Y^{l}_t \one_{\{\overline X^{l}_t=i\}} \one_{\{\overline \mu_t[i] \neq 0\}} 
\Big( \frac{k^i}{N\overline\mu_t[i]} - 1\Big)
}{\varpi + \overline\mu_t[i] + \frac1N  \sum_{l\in\NN} \overline Y^{l}_t  \one_{\{\overline X^{l}_t=i\}} \Big( \frac{k^i}{N\overline\mu_t[i]} - 1\Big) }  \nonumber
\\
&= \frac 1N \sum_{l\in\NN} \overline Y^{l}_t \sum_{j\in\dd}  \one_{\{\overline X^{l}_t=j\}}
\frac{\varphi(\overline \mu_t[i]) + \alpha_t^l[i]}{ \varpi + \overline \mu_t[i]+ \frac1N {\overline Y^{l}_t}  - \frac1N \overline Y^{l}_t \one_{\{j=i\}}}  \nonumber
\\
&\quad -  \frac 1N \sum_{l\in\NN} \overline Y^{l}_t \one_{\{\overline X^{l}_t=i\}} \sum_{j\in\dd}  \frac{\varphi(\overline\mu_t[j]) + \alpha_t^{l}[j]}{ \varpi+ \overline \mu_t[i]+ \frac1N {\overline Y^{l}_t} \one_{\{j=i\}}  - \frac1N \overline Y^{l}_t } \nonumber
\\
&\quad + \varepsilon N  \one_{\{\overline \mu_t[i] \neq 0\}}  \EE  \bigg[ \biggl(  {\frac{S_{\overline \mu_t}[\overline X^{n}_t]}{N\overline \mu_t[\overline X^{n}_t]}}
\biggr)^{\iota}
\frac{  N^{-1}{S_{{\overline \mu_t}}[i]} - \overline\mu_t[i]}
{\varpi + \overline{\mu}_t[i] + N^{-1} S_{\overline\mu_t}[i] - \overline \mu_t[i]} \bigg]  \nonumber
\\
&= \frac 1N \sum_{l\in\NN} \overline Y^{l}_t \sum_{j\neq i} \bigg( \one_{\{\overline X^{l}_t=j\}}
\frac{\varphi(\overline \mu_t[i]) + \alpha_t^l[i]}{ \varpi + \overline \mu_t[i]+ \frac1N \overline Y^{l}_t  } 
- \one_{\{\overline X^{l}_t=i\}}  \frac{\varphi(\overline\mu_t[j]) + \alpha_t^{l}[j]}
{ \varpi + \overline \mu_t[i] - \frac1N \overline Y^{l}_t } \bigg)  \nonumber
\\
&\quad + \varepsilon   \one_{\{\overline \mu_t[i] \neq 0\}}  \sum_{j\in\ES}
\one_{\{\overline X^{n}_t=j\}}
\EE  \bigg[ \biggl( 
{\frac{S_{\overline\mu_t}[j]}{N\overline \mu_t[j]}}
\biggr)^{\iota}
\frac{ S_{\overline \mu_t}[i] - N \overline \mu_t[i]}
{\varpi + N^{-1} S_{{\overline \mu_t}}[i]}    \bigg]  \nonumber
\\
&=: \bigl( \mbox{I} + \mbox{II} \bigr) +  \mbox{III}, \label{eq:I+II+III}
\end{align}
with the same convention as before that the ratio 
$S_{\overline\mu_t}[j]/(N\overline \mu_t[j])$ is understood as $1$
if $\overline \mu_t[j]=0$. Importantly, note that the denominators in expressions I and II are positive; indeed, on 
the event $\{ \overline X_{t}^l =i\}$, $\overline \mu_t[i] - N^{-1} \overline Y^{l}_t
= \overline\mu_t[i] - N^{-1} \overline Y^{l}_t \one_{\{\overline X^{l}_t=i\}} = 
N^{-1} \sum_{j\neq l} \overline Y^{j}_t \one_{\{\overline X^{j}_t=i\}}$, which is non-negative. 
\vskip 5pt

\noindent \textit{Third Step.}
We let $\varpi= N^{-\epsilon}$ and consider times $t < \overline\tau_N$,
so that $\overline \mu_t[i] \geq N^{-\epsilon}$ and $\overline Y^{l}_t \leq \tfrac12 N^{1-\epsilon}$ for any $l\in\NN$. We have 
\begin{align*}
&\varpi + \overline \mu_t[i] + \frac1N \overline Y^{l}_t \leq N^{-\epsilon} + \overline \mu_t[i] + \frac12 N^{-\epsilon} \leq \frac52 \overline\mu_t[i],
\\
& \varpi + \overline \mu_t[i] - \frac1N \overline Y^{l}_t \geq  \overline \mu_t[i] - \frac12 N^{-\epsilon} \geq \frac12 \overline\mu_t[i],
\end{align*}
and thus, using the  {definition} of $\varphi$ and the bounds $0\leq\alpha_t^{i,j} \leq 2(T\|f\|_\infty + \|g\|_\infty)$, we obtain
(recall \eqref{eq:M} for the definition of $M$)
\begin{equation}
\label{eq:I+II}
\begin{split}
\mbox{I} &\geq  
\frac{2\kappa}{5} \frac{1}{\overline \mu_t[i]}
  \one_{\{\overline\mu_t[i] \leq \delta\}}\sum_{j\neq i} \overline\mu_t[j] 
\geq \frac25 \kappa (1-\delta) \one_{\{\overline\mu_t[i] \leq \delta\}} \frac{1}{ \overline \mu_t[i]}, 
\\
\mbox{II} &\geq - (d-1) M \overline \mu_t[i]  \frac{2}{\overline \mu_{t}[i]}  = - 2 (d-1) M. 
\end{split}
\end{equation}
As to the {$j$-th} term {in} III, we note that, for any $l \in \NN$, ${S_{\overline\mu_t}[l]} \sim \mathrm{Bin}(N,\overline\mu_t[l])$ and, applying Hoeffding's inequality, we get
\be \label{Hoeffding}
\PP\bigg( \bigl| S_{\overline{\mu}_{t}}[l]-N\overline \mu_t[l]\bigr| \geq \frac12 N\overline\mu_t[l] \bigg) \leq 2\exp \bigg\{ -\frac{N (\overline \mu_t[i])^2}{2} \bigg\} \leq 2\exp \bigg\{-\frac{N^{1-2\epsilon}}{2} \bigg\}.
\ee  
For any $j \in \ES$, we let
\begin{equation*}
\mbox{III}_{j}
= \varepsilon N
\EE  \bigg[ \biggl( 
{\frac{S_{\overline\mu_t}[j]}{N\overline \mu_t[j]}}
\biggr)^{\iota}
\frac{ S_{\overline \mu_t}[i] - N \overline \mu_t[i]}
{\varpi + N^{-1} S_{{\overline \mu_t}}[i]}    \bigg].
\end{equation*}
Observing that, whether $\iota=0$ or $\iota=1$, it holds that 
$\EE  [ ( {S_{\overline\mu_t}[j]}/{N\overline \mu_t[j]}
)^{\iota} ]=1$, 
we may define the new probability measure 
\begin{equation*}
\label{eq:overline:pj}
\overline{\mathbf P}^j := \biggl( 
{\frac{S_{\overline\mu_t}[j]}{N\overline \mu_t[j]}}
\biggr)^{\iota} \cdot {\mathbf P},
\end{equation*}
and, then denoting by $\overline {\mathbf E}^j$ 
the related expectation, we obtain
\begin{equation}
\label{eq:IIIj}
\begin{split} 
\mbox{III}_{j}
&=\varepsilon N
\overline\EE^j  \bigg[ 
\frac{ S_{\overline \mu_t}[i] - N \overline \mu_t[i]}
{\varpi + N^{-1} S_{\overline \mu_t}[i]}    \biggr] 
\\
&= \varepsilon N \overline\EE^j   \bigg[
\frac{ S_{\overline \mu_t}[i] - N \overline\mu_t[i]}
{ N^{1-\epsilon}   + S_{\overline \mu_t}[i] }  \one_{\{ | S_{\overline \mu_{t}}[i]-N\overline \mu_t[i]| > \frac12 N\overline \mu_t[i] \}}\bigg]
\\ 
&\hspace{15pt}+ 
 \varepsilon N \overline\EE^j \bigg[
 \frac{ S_{\overline \mu_t}[i] - N \overline \mu_t[i]}
{  N^{1-\epsilon} + S_{\overline \mu_t}[i]}  \one_{\{ | S_{\overline \mu_{t}}[i]-N\overline \mu_t[i]| \le \frac12 N\overline\mu_t[i] \}}\bigg].
\end{split}
\end{equation}
Noticing that ${\mathbf E}[ 
( {S_{\overline\mu_t}[j]}/{N\overline \mu_t[j]}
)^{2\iota}
] \leq 1 + (1-\overline \mu_{t}[j])/(N \overline \mu_{t}[j])
\leq 1 + N^{\epsilon-1}\le 2$ 
(the worst case is $\iota=1$). In particular, for any event $A \in (\Xi,{\mathcal G},{\mathbf P})$, 
$\overline{\mathbf P}^j(A) \leq \sqrt{2} {\mathbf P}(A)^{1/2}$ and, in particular, 
we have a variant of 
\eqref{Hoeffding}
under $\overline{\mathbf P}^j$ (up to a multiplicative constant). 

Now, 
the first term in \eqref{eq:IIIj} can be bounded as follows 
\begin{equation}
\label{eq:IIIjb}
\begin{split}
&\varepsilon N  \biggl\vert \overline\EE^j   \bigg[
\frac{ S_{\overline \mu_t}[i] - N \overline\mu_t[i]}
{ N^{1-\epsilon}   + S_{\overline \mu_t}[i] }  \one_{\{ | S_{\overline \mu_{t}}[i]-N\overline \mu_t[i]| > \frac12 N\overline \mu_t[i] \}}\bigg] \biggr\vert
\\
&\le 
\varepsilon N \frac{N}{N^{1-\varepsilon}}\overline \PP^j \Bigl( | S_{\overline \mu_{t}}[i] -N\overline \mu_t[i]| \geq \frac12 N\overline \mu_t[i] \Bigr)
\\
&\leq  2\varepsilon N^{1+{\epsilon}} \exp \Bigl\{-\frac{  N^{1-2 \epsilon}}{4} \Bigr\}\leq C \varepsilon,
\end{split}
\end{equation}
with $C$ as in the statement. 
\vskip 5pt

\noindent \textit{Fourth Step.}
In order to bound the second term in III$_j$, 
we denote, for $x\geq -N\overline \mu_t[i] $,
\[
\psi (x) = \frac{x}{  (N^{-\epsilon} +\overline \mu_t[i]) N + x}.
\]
We note that $\psi$ is increasing and concave, and split
\begin{align}
\varepsilon N \overline\EE^j &\Big[ \psi\bigl( S_{{\overline \mu_t}}[i] - N\overline\mu_t[i]\bigr) \one_{\{ | {S_{\overline \mu_t}[i]}-N\overline \mu_t[i]| \leq \frac12 N\overline \mu_t[i] \}} \Big] \nonumber
\\
&= \varepsilon N \overline\EE^j \Big[\Bigl( \psi\bigl( S_{{\overline \mu_t}}[i] - N\overline \mu_t[i]\bigr) + \psi\bigl(   N\overline \mu_t[i]- S_{\overline \mu_t}[i]\bigr) \Bigr) \one_{\{ | {S_{\overline \mu_t[i]}-N\overline \mu_t[i]| \leq \frac12 N\overline\mu_t[i] \}}} \Big]
\nonumber
\\
&\quad+\varepsilon N \overline\EE^j  \Big[ - \psi\bigl(   N\overline \mu_t[i] - S_{\overline \mu_t}[i]\bigr) \one_{\{ | \{S_{\overline \mu_t[i]}-N\overline \mu_t[i]| \leq \frac12 N\overline \mu_t[i] \}} \Big] \nonumber
\\
&=: (A) + (B). \label{eq:A+B}
\end{align}
Notice that, on the event 
$\{ | {S_{\overline \mu_t}[i]}-N\overline \mu_t[i]| \leq \frac12 N\overline \mu_t[i] \}$, 
it holds that  
$N\overline \mu_{t}[i] - {S_{\overline \mu_t}[i]} \geq -\tfrac12 N\overline \mu_t[i]$, which 
makes licit the composition by $\psi$. 
Thus Jensen's inequality, under the conditional probability given $\{
| {S_{\overline \mu_t}[i]}-N\overline \mu_t[i]| \leq \frac12 N\overline \mu_t[i] 
\}$, gives
\begin{align*}
\overline \EE^j &\Big[ - \psi\bigl(   N\overline \mu_t[i] - S_{\overline \mu_t[i]}\bigr) \Big| 
| {S_{\overline \mu_t}[i]}-N\overline \mu_t[i] | \leq \tfrac12 N\overline \mu_t[i]  \Big]
\\
&\geq  - \psi\biggl(   \overline \EE^j \Big[ N\overline \mu_t[i] - S_{{\overline \mu_t}}[i] \Big| 
| {S_{\overline \mu_t}[i]}-N\overline \mu_t[i] | \leq \tfrac12 N\overline \mu_t[i]  \Big] \biggr) 
\\
&= -\psi\Bigg( \frac{\overline \EE^j\big[ \bigl(N\overline \mu_t[i] - S_{{\overline \mu_t}}[i]\bigr) \one_{\{| {S_{\overline \mu_t[i]}-N\overline \mu_t[i]| \leq \tfrac12 N\overline \mu_t[i]\}}} \big]}{\overline \PP^j\big(| {S_{\overline \mu_t[i]}-N\overline \mu_t[i]| \leq \tfrac12 N\overline \mu_t[i]\big)}}    \Bigg).
\end{align*}
Now, notice that 
\begin{equation*}
\begin{split}
&\overline \EE^j\big[ N\overline \mu_t[i] - S_{{\overline \mu_t}}[i]\bigr]
 = \EE\Big[
\Bigl( {\frac{S_{\overline\mu_t}[j]}{N\overline \mu_t[j]}}\Bigr)^{\iota} \Bigl( N\overline \mu_t[i] - S_{{\overline \mu_t}}[i] \Bigr)\Bigr]
\end{split}
\end{equation*}
If $\iota=0$, then both terms in the above formula are obviously $0$. 
Otherwise, $\iota=1$ and then,  at least for $i \not =j$, 
\begin{equation*}
\begin{split}
&\overline \EE^j\big[ N\overline \mu_t[i] - S_{{\overline \mu_t}}[i]\bigr]
\\
&= \EE\Big[
  {\frac{S_{\overline\mu_t}[j]}{N\overline \mu_t[j]}} \Bigl( N\overline \mu_t[i] - S_{{\overline \mu_t}}[i] \Bigr)\Bigr]
\\
&= \frac1{N\overline \mu_t[j]}
\EE\Big[
\Bigl( S_{\overline\mu_t}[j] - N\overline \mu_t[j]\Bigr)  \Bigl( N\overline \mu_t[i] - S_{{\overline \mu_t}}[i] \Bigr)\Bigr]
= \frac{N \overline \mu_t[j]\overline \mu_t[i]}{N\overline \mu_t[j]}= \overline \mu_t[i]. 
\end{split}
\end{equation*}
Noticing that the first term on the third line is negative if $i=j$ and recalling 
our variant of 
\eqref{Hoeffding},
we get in any case (whether 
$i =j$ or not):
\begin{align*}
&\overline \EE^j\big[ \bigl(N\overline \mu_t[i] - S_{{\overline \mu_t}}[i]\bigr) \one_{\{| {S_{\overline \mu_t[i]}-N\overline \mu_t[i]| \leq \tfrac12 N\overline \mu_t[i]\}}} \big]
\\
&= \overline \EE^j\big[ \bigl(N\overline \mu_t[i] - S_{{\overline \mu_t}}[i]\bigr) \big]
-
\overline \EE^j\big[ \bigl(N\overline \mu_t[i] - S_{{\overline \mu_t}}[i]\bigr) \one_{\{| {S_{\overline \mu_t[i]}-N\overline \mu_t[i]| > \tfrac12 N\overline \mu_t[i]\}}} \big]
\\
&\leq \overline \mu_t[i]
-
\overline \EE^j\big[ \bigl(N\overline \mu_t[i] - S_{{\overline \mu_t}}[i]\bigr) \one_{\{| {S_{\overline \mu_t[i]}-N\overline \mu_t[i]| > \tfrac12 N\overline \mu_t[i]\}}} \big]
\\
&\leq
\overline \mu_t[i]
+ 2N 
\overline {\mathbf P}^j\Bigl( \bigl| {S_{\overline \mu_t[i]}-N\overline \mu_t[i]\bigr| > \tfrac12 N\overline \mu_t[i]\}} \Bigr)
\\
&\leq
\overline \mu_t[i]
+
2 N \exp \Big\{-\frac{ N^{1-2\epsilon}}4 \Bigr\}, 
\end{align*}
so that 
\begin{equation*}
\begin{split}
\frac{\overline \EE^j\big[ \bigl(N\overline \mu_t[i] - S_{{\overline \mu_t}}[i]\bigr) \one_{\{| {S_{\overline \mu_t[i]}-N\overline \mu_t[i]| \leq \tfrac12 N\overline \mu_t[i]\}}} \big]}{\overline \PP^j\big(| {S_{\overline \mu_t[i]}-N\overline \mu_t[i]| \leq \tfrac12 N\overline \mu_t[i]\big)}}
 &\leq \frac{\overline \mu_t[i]
+
2 N \exp \big\{-\tfrac14  N^{1-2\epsilon} \big\}}{1-2 N \exp \big\{-\tfrac14  N^{1-2\epsilon} \big\}}
\\
&\leq C \overline \mu_t[i] + \frac{C}{N^2},
\end{split}
\end{equation*}
at least for $N$ large enough (the underlying rank upon which the above bound is true 
only depending $\epsilon$); by noticing that the left-hand side is upper bounded by $N$, we can change the constant 
$C$ accordingly such that the above is always true (even for $N$ small). 
We deduce that 
\begin{equation*}
\begin{split}
&\overline \EE^j \Big[ - \psi\bigl(   N\overline \mu_t[i] - S_{\overline \mu_t[i]}\bigr) \Big| 
| {S_{\overline \mu_t}[i]}-N\overline \mu_t[i] | \leq \tfrac12 N\overline \mu_t[i]  \Big]
 \geq - \psi \Bigl( C \overline \mu_t[i] + \frac{C}{N^2} \Bigr).
\end{split}
\end{equation*}
Now, allowing the constant $C$ to vary from one inequality to another, 
\begin{equation*}
\begin{split}
 \psi \Bigl( C \overline \mu_t[i] + \frac{C}{N^2} \Bigr)
 &\leq 
C  \frac{\overline \mu_t[i] + N^{-2}}{  (N^{-\epsilon} +\overline \mu_t[i]) N + C \overline \mu_t[i]} 
\leq \frac{C}{N}.
\end{split}
\end{equation*}
%
%
As a result (recall \eqref{eq:A+B} for the definition of $(A)$ and $(B)$),
\begin{align*}
(B) &= \varepsilon N \overline\EE^j  \Big[ - \psi\bigl(   N\overline \mu_t[i] - S_{\overline \mu_t}[i]\bigr) \one_{\{ | \{S_{\overline \mu_t[i]}-N\overline \mu_t[i]| \leq \frac12 N\overline \mu_t[i] \}} \Big] 
\\
& =  \varepsilon N \overline \EE^j \Big[ 
- \psi\bigl(   N\overline \mu_t[i] - S_{\overline \mu_t}[i]\bigr) \Big|
 | \{S_{\overline \mu_t[i]}-N\overline \mu_t[i]| \leq \frac12 N\overline \mu_t[i] \}
 \Bigr] 
 \\
&\hspace{15pt} \times \overline \PP^j 
\Bigl( 
S_{\overline \mu_t[i]}-N\overline \mu_t[i]| \leq \frac12 N\overline \mu_t[i]
\Bigr)
\\
& \geq -\varepsilon N \frac{C}{N} = -C\varepsilon.
\end{align*}
The term $(A)$ is instead
\begin{align*}
(A)&=     
\varepsilon N \EE  \bigg[ \bigg(
\frac{ S_{{\overline \mu_t}}[i] - N \overline\mu_t[i]}
{  ({N^{-\epsilon}} +  \overline\mu_t[i]) N  + S_{{\overline\mu_t}}[i] - N\overline \mu_t[i] } 
\\
&\hspace{15pt}+ \frac{  N \overline \mu_t[i] - S_{{\overline \mu_t}}[i]}
{  ({N^{-\epsilon}} +  \overline \mu_t[i]) N   + N \overline \mu_t[i] - S_ {\overline \mu}[i]}   \bigg)
 \one_{\{ | {S_{\overline \mu_t}[i]}-N\overline \mu_t[i]| \leq \frac12 N\overline \mu_t[i] \}}\bigg] 
 \\
 &= \varepsilon N \EE  \bigg[ -
\frac{ 2(S_{{\overline\mu_t}}[i] - N \overline \mu_t[i])^2}
{  ({N^{-\epsilon}} +  \overline \mu_t[i])^2 N^2  - (S_{{\overline \mu_t}}[i] - N\overline\mu_t[i])^2 } 
 \one_{\{ | {S_{\overline \mu_t}[i]}-N\overline \mu_t[i]| \leq \frac12 N\overline \mu_t[i] \}}\bigg] 
 \\
 & \geq - \varepsilon N \EE  \bigg[
\frac{ 2 (S_{{\overline\mu_t}}[i] - N \overline \mu_t[i])^2}
{   \frac34 (\overline \mu_t[i])^2 N^2  } 
 \one_{\{ | S_{\overline\mu_t}[i]-N\overline \mu_t[i]| \leq \frac12 N\overline \mu_t[i] \}}\bigg] 
 \\
& \geq - \varepsilon N \frac83 \frac{N \overline \mu_t[i](1-\overline \mu_t[i])}{N^2 (\overline \mu_t[i])^2} 
 \geq - \frac83 \varepsilon   \frac{1}{\overline \mu_{t}[i]} 
 \geq  -  \frac83 \varepsilon  \frac{\one_{\{\overline \mu_{t}[i]\leq\delta\}}}{\overline \mu_{t}[i]} -\varepsilon \frac{8 }{3\delta}, 
\end{align*}
and so, combining with the lower bound for $(B)$, we obtain from 
\eqref{eq:IIIj}
and
\eqref{eq:IIIjb}:
\[
\mbox{III} \geq -C\varepsilon  
-  \frac83 \varepsilon \frac{\one_{\{\overline \mu_{t}[i]\leq\delta\}}}{\overline \mu_t[i]} 
-\varepsilon \frac{8}{3\delta} \geq -C_0 -  \frac83 \varepsilon \frac{\one_{\{\overline \mu_t[i]\leq\delta\}}}{\overline \mu_t[i]},
\]
for a  constant $C_0$ as in the statement, 
using $\varepsilon<1$. 
\vskip 5pt

\noindent \textit{Conclusion.}
Putting things together (see
\eqref{eq:I+II+III}
and 
\eqref{eq:I+II}), we obtain
\[
-\overline \lambda_t \geq \bigg(\frac25 (1-\delta) \kappa  - \frac83 \varepsilon  \bigg)\frac{\one_{\{\overline \mu_t[i]\leq\delta\}}}{\overline \mu_t[i]} -C_1,
\]
for another $C_1$ as in the statement. This inequality, applied in \eqref{eqexp}, gives
\[
\E \bigg[ \exp\bigg\{ \int_0^{\overline \tau_N} \bigg(\frac25 (1-\delta) \kappa  -  \frac83 \varepsilon\bigg)\frac{\one_{\{\overline\mu_t[i]\leq\delta\}}}{\overline\mu_t[i]} dt \bigg\}\bigg] \leq \frac{2}{N^{-\epsilon}+\overline\mu_0[i]} \exp(C_1 T).
\]
Therefore \eqref{expboundmu} follows if we choose $\kappa$ such that
\[
\frac25 (1-\delta) \kappa  - \varepsilon \frac83 \geq \lambda,
\]
and then use another value of $C$. 

In order to prove \eqref{eq1/mu}, we exploit \eqref{exp} to derive
\begin{equation*}
\begin{split}
& \frac{1}{N^{-\epsilon}+ \overline \mu_{t\wedge \overline \tau_N}[i]} \exp\biggl(
 \int_0^{ t\wedge \overline \tau_N}  \frac{\lambda}{\overline \mu_{s}[i]} {\mathbbm 1}_{\{\overline \mu_{s}[i] \leq \delta \} } ds
 - C_{1} T 
 \biggr)
 \leq 
 \exp\bigg(- \overline{\Lambda}_{ t\wedge \overline \tau_N} - \int_0^{ t\wedge \overline \tau_N}  \overline\lambda_s ds\bigg). 
\end{split}
\end{equation*}
Taking expectation, we conclude by recalling that the exponential on the right-hand side has expectation bounded by $1/(N^{-\epsilon}+ \overline \mu_{0}[i])$. \qed

\subsection{Proof of Lemma \ref{lem:aux:moments:binomial}}
\label{subse:6:2}
{ \ }
\vskip 5pt

\noindent \textit{First Step.}
We start with the first line in 
\eqref{eq:aux:moments:binomial:1}. By standard algebra, we get 
\begin{equation}
\label{eq:step:1:aux:moments:binomial:1}
\begin{split}
\EE\biggl[ \biggl( \frac{S_{\mu}[i]}{N \mu[i]} \biggr)^{\ell} -1\biggr] &=
 \frac{1}{N^{\ell}\mu[i]^{\ell}} \EE\Bigl[ \Bigl( S_{\mu}[i]- N\mu[i] +N\mu[i]\Bigr)^{\ell} \Bigr] -1
 \\
 &= \frac{1}{N^{\ell}  \mu[i]^{\ell}}\sum_{k=1}^{\ell} \binom{\ell}{k} \EE\Bigl[\Bigl(S_{\mu}[i]- N  \mu[i]\Bigr)^k\Bigr] \bigl(N \mu[i]\bigr)^{\ell-k}
 \\
 &= \binom{\ell}{2}\frac{N   \mu [i] (1-   \mu [i] )}{\bigl(N    \mu [i]\bigr)^2 } + \sum_{k=3}^{\ell} \binom{\ell}{k} \frac{\EE\bigl[\bigl(S_{ \mu }[i]- N  \mu [i]\bigr)^k\bigr]}{\bigl(N   \mu[i]\bigr)^k}.
\end{split} 
 \end{equation}
Recall Rosenthal's inequality (see \cite[Theorem 2.9]{petrov}) for independent integrable random variables $\{Z_k\}_{k=1}^n$ with $\E[Z_k]=0$ and for $p\ge 2$:
\begin{align*}
\E\Big|\sum_{k=1}^nZ_k\Big|^p\le c(p)\left(\sum_{k=1}^n\E[|Z_k|^p]+\Big(\sum_{k=1}^n\E[|Z_k|^2]\Big)^{p/2}\right),
\end{align*}
where $c(p)$ is a positive constant depending only on $p$.
Applying it to the centered sum $ S_{ \mu}[i]- N  \mu[i]$, we get, that for each real $p \geq 2$, 
 \begin{equation}
 \label{eq:rosenthal}
 \begin{split}
& \EE\bigl[\bigl\vert S_{ \mu}[i]- N  \mu[i]\bigr\vert^p\bigr] 
\\
&\leq C_{p}
 N \bigl[ \mu[i]^p\bigl(1- \mu[i] \bigr) + \bigl(1-  \mu [i]\bigr)^p  \mu [i]\bigr] +C_{p} N^{p/2}  \mu [i]^{p/2}\bigl(1-
  \mu[i]\bigr)^{p/2}
\\
&\leq C_{p}
 N  \mu [i] +C_{p} N^{p/2}  \mu [i]^{p/2}
\\
&\leq C_{p}
  N^{p/2}  \mu[i]. 
\end{split} 
 \end{equation}
 for a constant $C_{p}$ depending on $p$, the value of which may vary from line to line.  
And then, by combining the last two inequalities, we obtain the first line in 
\eqref{eq:aux:moments:binomial:1}. 

We now turn to the second line in 
\eqref{eq:aux:moments:binomial:1}.
Following the analysis of the first line, we have 
\begin{align*}
& \EE\biggl[ 
 \frac{S_{\mu}[j]}{N \mu[j]}
\biggl\{
\biggl( \frac{S_{\mu}[i]}{N \mu[i]} \biggr)^{\ell} -1\biggr\}\biggr]
\\
&= 
\frac{1}{N^{\ell+1}  \mu[i]^{\ell} \mu[j]}\sum_{k=1}^{\ell} \binom{\ell}{k} \EE\Bigl[
 \Bigl( S_{\mu}[j] 
 - N \mu[j] \Bigr)
\Bigl(S_{\mu}[i]- N \mu[i]\Bigr)^k\Bigr] \bigl(N  \mu[i]\bigr)^{\ell-k}
\\
&\hspace{15pt}+
\frac{1}{N^{\ell}  \mu[i]^{\ell}}\sum_{k=2}^{\ell} \binom{\ell}{k} \EE\Bigl[
\Bigl(S_{\mu}[i]- N \mu[i]\Bigr)^k\Bigr] \bigl(N  \mu[i]\bigr)^{\ell-k}.
\end{align*}
When $k=1$ in the first sum, the expectation therein is given by the correlation matrix of the multinomial distribution, 
namely 
$\EE[(S_{\mu}[j]- N \mu[j])(S_{ \mu}[i]- N  \mu[i])]=
N  \mu [i] \delta_{i,j}
- N   \mu[i]
  \mu[j]
$; when $k = 2$ in the second sum, the expectation therein is given by 
$\EE[(S_{\mu}[i]- N \mu[i])^2]=
N  \mu [i] (1- \mu[i])$. As for the other terms  (whatever the sum),  we may just invoke Cauchy--Schwarz inequality and then \eqref{eq:rosenthal} in order to bound the corresponding expectation. As a result, we can find a constant $C$, only depending on $\ell$, such that
\begin{equation*}
\begin{split}
&\biggl\vert \EE\biggl[ 
 \frac{S_{\mu}[j]}{N \mu[j]}
\biggl\{
\biggl( \frac{S_{\mu}[i]}{N \mu[i]} \biggr)^{\ell} -1\biggr\}\biggr]
\biggr\vert
\\
&\leq 
\frac{\ell}{N \mu[j]}
+  
\frac{\ell(\ell-1)}{2 N \mu[i]}
+
\frac{C}{N^{\ell+1} \mu[i]^{\ell}  \mu[j]}\sum_{k=2}^{\ell} N^{(k+1)/2} 
 \bigl(  \mu[i] \bigr)^{1/2} \bigl(  \mu[j] \bigr)^{1/2}
\bigl(N  \mu[i]\bigr)^{\ell-k}
\\
&\hspace{15pt}
+
\frac{C}{N^{\ell} \mu[i]^{\ell}}\sum_{k=3}^{\ell} N^{k/2} 
 \mu[i]   
\bigl(N  \mu[i]\bigr)^{\ell-k}
\\
&\leq 
\frac{\ell(\ell+1)}{2 N \min_{e \in \ES} \mu[e]}
 +
{C}\sum_{k=2}^{\ell} 
\frac1{
N^{(k+1)/2} 
  ( \mu[i]  )^{k-1/2}  ( \mu[j] )^{1/2}}
  +
 C \sum_{k=3}^{\ell} 
 \frac{1}{N^{k/2} 
 \bigl(N  \mu[i]\bigr)^{k-1}}
  \\
&\leq 
\frac{\ell(\ell+1)}{2 N \min_{e \in \ES} \mu[e]}
 +
{C}\sum_{k=3}^{\ell+1} 
\frac1{
N^{k/2} 
  \min_{e \in \ES} \mu[e]^{k-1}},
\end{split}
\end{equation*}
which fits the announced inequality. 
\vskip 4pt

\noindent \textit{Second Step.}
We now prove 
\eqref{eq:aux:moments:binomial:2}. 
Following 
\eqref{eq:step:1:aux:moments:binomial:1}, we have 
\begin{equation}
\label{eq:step:10:aux:moments:binomial:1}
\begin{split}
\biggl\vert \biggl( \frac{S_{\mu}[i]}{N \mu[i]} \biggr)^{\ell} -1\biggr\vert &= 
 \frac{1}{N^{\ell}\mu[i]^{\ell}} \Bigl\vert \Bigl( S_{\mu}[i]- N\mu[i] +N\mu[i]\Bigr)^{\ell} -N^{\ell}\mu[i]^{\ell}\Bigr\vert
 \\
 &\leq C \sum_{k=1}^{\ell}  
 \Bigl\vert S_{\mu}[i]- N  \mu[i] \Bigr\vert^k \bigl(N \mu[i]\bigr)^{-k}.
\end{split} 
 \end{equation}
By the third line in 
 \eqref{eq:rosenthal}, we get 
\begin{equation*}
\begin{split}
{\mathbf E}\biggl[ \biggl\vert \biggl( \frac{S_{\mu}[i]}{N \mu[i]} \biggr)^{\ell} -1\biggr\vert^{p} \biggr]^{1/p} &\leq C \sum_{k=1}^{\ell}  
\bigl( N \mu[i] + N^{kp/2} \mu[i]^{kp/2} \bigr)^{1/p}
\bigl(N \mu[i]\bigr)^{-k}
\\
&\leq C 
\biggl[ 
\bigl( N \mu[i]  \bigr)^{1/p}
\sum_{k=1}^{\ell}  
\bigl(N \mu[i]\bigr)^{-k}
+
\sum_{k=1}^{\ell}  
\bigl(N \mu[i]\bigr)^{-k/2}
\biggr], 
\end{split} 
 \end{equation*}
which completes the proof, since $p \geq 2$. 

It now remains to address the inequality \eqref{eq:ldp:moments:binomial:-1}. 
With $\eta>0$ as in the statement and with $C$ as in 
\eqref{eq:step:10:aux:moments:binomial:1}, choose 
$\eta'=\min(1,\eta/(\ell C))$ and deduce that
$
 \vert \left( {S_{\mu}[i]}/{N \mu[i]} \right)^{\ell} -1 \vert \leq \eta,
$
on the event 
$
E = \left\{\left\vert  {S_{\mu}[i]}/{N \mu[i]} - 1 \right\vert \leq \eta' \right\}. 
$
Now, 
Hoeffding's inequality 
says that 
${\mathbf P} ( E^{\complement}) \leq 
2 \exp (-  2 N \mu[i]^2 (\eta')^2)$, which completes the proof.  \qed

\subsection{Proof of Proposition \ref{thm:bound_Y_tau}}
\label{subse:6:3}
{\ }\vskip 5pt

\noindent For a given $\lambda \geq 1$, 
consider the exponential
\[ 
\mathcal{E}_t:=\exp\bigg\{\lambda \int_0^t \sum_{i\in\ES} \frac{ \one_{\{\overline \mu_s[i]\neq 0\}}}{\overline \mu_s[i]} ds \bigg\}, \quad t \in [0,T]. 
\]
Applying It\^o's formula to $(\mathcal{E}_{t \wedge \overline \tau_N}^{-1} \sum_{l \in \NN} |\overline Y^{l}_{t \wedge \overline \tau_N}|^{2\ell})_{0 \le t \le T}$ (recall for instance \eqref{integral}) and taking expectation, we obtain
\begin{equation}
\label{eq:proof:moments}
\begin{split}
\frac{d}{dt}&\E\bigg[ \mathcal{E}_{t \wedge \overline \tau_N}^{-1} \sum_{l\in\NN} |\overline Y^{l}_{t \wedge \overline \tau_N}|^{2\ell} \bigg] 
+\E\bigg[ {\mathbbm 1}_{\{t < \overline \tau_{N}\}} \mathcal{E}_t^{-1}  \sum_{i\in\dd} \frac{\lambda }{\overline \mu_t[i]} \sum_{l\in\NN} |\overline Y^{l}_t|^{2\ell} \bigg]\\
&\leq 
\varepsilon N \E\bigg[ {\mathbbm 1}_{\{t < \overline \tau_{N}\}} \mathcal{E}_t^{-1} \sum_{l\in\NN} |\overline Y^{l}_t|^{2\ell} \sum_{i\in\dd}  \one_{\{\overline X^{l}_t=i\}}
\\
&\qquad\quad
\times \sum_{j\in\ES}{\mathbbm 1}_{\{\overline X^{n}_t=j\}} 
\biggl\vert
\EE \bigg\{
\biggl( {\frac{S_{\overline \mu_t}[j]}{N\overline \mu_t[j]}}
\biggr)^{\iota}
\bigg(\biggl( \frac{S_{\overline\mu_t}[i]}{N \overline\mu_t[i]}\biggr)^{2\ell} -1\bigg)\bigg\} \biggr\vert \bigg].
\end{split}
\end{equation}
We first handle the expectation $\EE$ in the above right-hand side. 
The key point is to notice that
it can be estimated by 
the first line in
\eqref{eq:aux:moments:binomial:1}
when $\iota=0$
and by the second line in 
\eqref{eq:aux:moments:binomial:1}
when $\iota=1$. 
%
%
%
 In any case, we have
\begin{equation*}
\begin{split}
\biggl\vert
\EE \bigg\{
\biggl( {\frac{S_{\overline \mu_t}[j]}{N\overline \mu_t[j]}}
\biggr)^{\iota}
\bigg(\biggl( \frac{S_{\overline\mu_t}[i]}{N \overline\mu_t[i]}\biggr)^{2\ell} -1\bigg)\bigg\} \biggr\vert 
& \leq 
\frac{\ell(2\ell+1)}{N \min_{e \in \ES} \mu[e]} 
 +
{C_{\ell}}\sum_{k=3}^{2\ell+1} 
\frac1{
N^{k/2} 
  \min_{e \in \ES} \overline \mu_t[e]  ^{k-1}}
\\
& \leq 
\sum_{e \in \ES}
\frac{\ell(2\ell+1)}{N  \mu[e]} 
 +
{C_{\ell}}\sum_{k=3}^{2\ell+1} 
\frac1{
N^{k/2} 
  \min_{e \in \ES} \overline \mu_t[e]  ^{k-1}},  
\end{split}
\end{equation*}
for a constant $C_{\ell}$, only depending on $\ell$ (and the value of which is allowed to vary from line to line). Integrating
\eqref{eq:proof:moments}
 from 0 to $t$,  we have 
 (allowing $C_{\ell}$ to depend on $d$)
\begin{align*}
\E\bigg[ &\mathcal{E}_{t\wedge \overline \tau_N}^{-1} \sum_{l\in\NN} |\overline Y^{l}_{t \wedge \overline \tau_N}|^{2\ell} \bigg] 
+\lambda \E\bigg[ \int_0^{t \wedge \overline \tau_N} \mathcal{E}_s^{-1} \sum_{i\in\dd} \frac{1}{\overline \mu_s[i]}  \sum_{l\in\NN} |\overline Y^{l}_s|^{2\ell}ds\bigg]
\\
&\leq \sum_{l\in\NN} |y^l_0|^{2\ell} + \varepsilon  \ell(2\ell+1)  \E\biggl[ \int_0^{t \wedge  \overline\tau_N} \mathcal{E}_s^{-1} 
\sum_{i\in\dd}   \frac{1}{\overline \mu_s[i]}
\sum_{l\in\NN} |\overline Y^{l}_s|^{2\ell}  ds\biggr]
 \\
&+  \varepsilon C_\ell  \E\bigg[ \int_0^{t \wedge \overline \tau_N}\mathcal{E}_s^{-1}
\sum_{k=3}^{2\ell+1}
\frac{1}{N^{k/2}  (\min_{e \in \ES} \overline \mu_s[e]  )^{k-1}}
  \sum_{l\in\NN} |\overline Y^{l}_s|^{2\ell}  ds\bigg],
\end{align*}
which implies,  if $\lambda\geq \varepsilon  \ell(2\ell+1)$,  
\begin{align*}
\E\bigg[ \mathcal{E}_{t \wedge \overline \tau_N}^{-1} \sum_{l\in\NN} |\overline Y^{l}_{t \wedge \overline \tau_N}|^{2\ell} \bigg] 
&\leq \sum_{{l\in\NN}} |y^l_0|^{2\ell}  
\\
&+
 \varepsilon C_\ell 
 \sum_{k=3}^{2\ell+1}
\frac{N^{\epsilon(k-1)}}{N^{k/2}}
 \E\bigg[ \int_0^{t}\mathcal{E}_{s \wedge \overline \tau_N}^{-1} \sum_{l\in\NN} |\overline Y^{l}_{s \wedge \overline \tau_{N}}|^{2\ell} ds\bigg].
\end{align*}
Hence, if $\epsilon < 1/4$, Gronwall's inequality yields 
\[
\E\bigg[ \mathcal{E}_{t \wedge \overline \tau_N}^{-1} \sum_{l\in\NN} |\overline Y^{l}_{t \wedge \overline \tau_N}|^{2\ell} \bigg] \leq C \sum_{{l\in\NN}} |y^l_0|^{2\ell},
\]
for a constant $C$ as in the statement
and whose value may vary from line to line. 
Now, by Cauchy--Schwarz inequality, 
\begin{align*}
\E\bigg[ \frac1N \sum_{l\in\NN} |\overline Y^{l}_{t \wedge \overline \tau_N}|^{\ell} \bigg] 
&\leq 
\E\bigg[{\mathcal E}_{t \wedge \overline \tau_{N}}^{-1} 
\biggl( \frac1N \sum_{l\in\NN} |\overline Y^{l}_{t \wedge \overline \tau_N}|^{\ell} \biggr)^2 \bigg]^{1/2}
\E\big[{\mathcal E}_{t \wedge \overline \tau_{N}}   \bigr]^{1/2}
\\
&\leq 
\E\bigg[{\mathcal E}_{t \wedge \overline \tau_{N}}^{-1} 
\biggl( \frac1N \sum_{l\in\NN} |\overline Y^{l}_{t \wedge \overline \tau_N}|^{2\ell} \biggr) \bigg]^{1/2}
\E \bigl[ \mathcal{E}_{t \wedge \overline \tau_N} \bigr]^{1/2}
\\
&\leq C \biggl( \frac1N \sum_{{l\in\NN}} |y^l_0|^{2\ell} \biggr)^{1/2}
\E \bigl[ \mathcal{E}_{t \wedge \overline \tau_N} \bigr]^{1/2}.
\end{align*}
%
%
%
We invoke  \eqref{bound_calE}
(assuming throughout that $\kappa$ is large enough) in order to bound the last term on the right-hand side. 
We easily get 
\eqref{boundY_magenta}. 
%
Combining 
\eqref{eq1/muexp} with \eqref{boundY_magenta}, we get
\begin{align}
\label{bound_two_sums}
&\E \bigg[  
\prod_{i \in \ES}
\Bigl(
N^{-\epsilon}+\overline \mu_{t \wedge \overline \tau_N}[i] \Bigr)^{-1/d}
+ 
\frac1N \sum_{l\in\NN} |\overline Y^{l}_{t \wedge \overline \tau_N}|^{\ell}
\bigg]
\\
\notag
&\quad\le 
C \prod_{i \in \ES}
\bigl(
N^{-\epsilon}+\overline \mu_{0}[i] \bigr)^{-1/d}
+ C \biggl( \frac1N \sum_{l\in\NN} |y^l_0|^{2\ell} \biggr)^{1/2}
\prod_{i \in \ES}
\bigl(
N^{-\epsilon}+\overline \mu_{0}[i] \bigr)^{-1/(2d)}.
%
%
\end{align}
Choosing $t=T$ and invoking the definition of $\overline \tau_N$ in \eqref{deftau} (together with the right continuity of the trajectories), we get
\begin{equation*}
\begin{split}
&\E \bigg[  \bigg( 
\prod_{i \in \ES}
\Bigl(
N^{-\epsilon}+\overline \mu_{\overline \tau_N}[i] \Bigr)^{-1/d}
+ 
\frac1N \sum_{l\in\NN} |\overline Y^{l}_{\overline \tau_N}|^{\ell} \bigg) \one_{\{\overline \tau_N <T\}}
\bigg]
\\
&\geq \P(\overline \tau_N <T) \min \bigg\{\frac1{2} N^{\epsilon/d}, 
\frac1{2^\ell}
N^{\ell-1 - \ell\epsilon}\bigg\}, 
\end{split}
\end{equation*}
where we used the fact that, for any $i \in \ES$, 
$N^{-\epsilon}+\overline \mu_{\overline \tau_N}[i] \leq 2$. 
Noticing that, if $\ell \geq 3$ and $\epsilon < 1/4$, $\ell-1-\ell \epsilon \geq 3/2-1 \geq 1/2 \geq \epsilon$,  we obtain  \eqref{boundtau_magenta}.

\subsection{Proof of Proposition \ref{thm:bound_Y_tau:-1}}
\label{subse:6:4}
{ \ }
\vskip 5pt

\noindent \textit{First Step.} We start with a similar computation to 
\eqref{eq:proof:moments}. 
\begin{align}
&\frac{d}{dt} {\mathbb E} \biggl[ {\mathcal E}_{t \wedge \overline \tau_N}^{-1} 
\biggl(
 \sum_{l \in \EN} \vert \overline Y_{t \wedge \overline \tau_N}^l \vert^{\ell} \biggr)^{-1}
\biggr]
+
{\mathbb E} \biggl[ 
{\mathbbm 1}_{\{ t < \overline \tau_{N}\}}
{\mathcal E}_{t}^{-1}
\sum_{i \in \ES}
\frac{\lambda}{\overline \mu_{t}[i]}
\biggl( \sum_{l \in \EN} \vert \overline Y_{t}^l \vert^{\ell} \biggr)^{-1}
\biggr] \nonumber
\\
&=  \varepsilon N 
\sum_{j\in \ES}
 {\mathbb E} \biggl[ 
 {\mathbbm 1}_{\{ t < \overline \tau_{N}\}}
 {\mathcal E}_{t}^{-1}
 {\mathbbm 1}_{\{ \overline X_{t}^n=j\}} 
{\mathbf E}\biggl\{  \biggl( \frac{S_{\overline \mu_{t}}[j]}{N \overline \mu_{t}[j]}
\biggr)^{\iota} 
\biggl[
\biggl( \sum_{i \in \ES} 
\sum_{l \in \EN} \vert \overline Y_{t}^l \vert^{\ell}
{\mathbbm 1}_{\{ \overline X_{t}^l=i\}}
\Bigl( \frac{S_{\overline \mu_{t}}[i]}{N \overline \mu_{t}[i]}
\Bigr)^{\ell}
 \biggr)^{-1} \nonumber
 \\
&\hspace{30pt}-
\biggl( \sum_{l \in \EN} \vert \overline Y_{t}^l \vert^{\ell} \biggr)^{-1}
\biggr]\biggr\} 
\biggr] \label{eq:theorem:inverse:y:proof:1}
\\
&= - \varepsilon N 
\sum_{j\in \ES}
 {\mathbb E} \biggl[ 
 {\mathbbm 1}_{\{ t < \overline \tau_{N}\}}
 {\mathcal E}_{t}^{-1}
 {\mathbbm 1}_{\{ \overline X_{t}^n=j\}} 
{\mathbf E}\biggl\{  \biggl( \frac{S_{\overline \mu_{t}}[j]}{N \overline \mu_{t}[j]}
\biggr)^{\iota}  \nonumber
  \biggl(
\sum_{l \in \EN} \vert \overline Y_{t}^l \vert^{\ell}
+ D \overline Y_{t}
 \biggr)^{-1}\biggl( \sum_{l \in \EN} \vert \overline Y_{t}^l \vert^{\ell} \biggr)^{-1}
D\overline Y_{t}
\biggr\}
\biggr], \nonumber
\end{align}
with
$$\displaystyle D\overline Y_{t}
:=\frac{1}{N}
 \sum_{i \in \ES} 
\sum_{l \in \EN} \vert \overline Y_{t}^l \vert^{\ell}
{\mathbbm 1}_{\{ \overline X_{t}^l=i\}}
\Bigl[ \Bigl( \frac{S_{\overline \mu_{t}}[i]}{N \overline \mu_{t}[i]}
\Bigr)^{\ell}
-1 \Bigr].$$

Then, 
\begin{equation}
\label{eq:decomposition:inverse:weight}
\begin{split}
&{\mathbf E} \biggl\{ 
\biggl( \frac{S_{\overline \mu_{t}}[j]}{N \overline \mu_{t}[j]}
\biggr)^{\iota} 
\biggl(
\sum_{l \in \EN} \vert \overline Y_{t}^l \vert^{\ell}
+ ND \overline Y_{t}
 \biggr)^{-1}\biggl( \sum_{l \in \EN} \vert \overline Y_{t}^l \vert^{\ell} \biggr)^{-1}
ND\overline Y_{t}
\biggr\}
\\
&={\mathbf E} \biggl\{ 
\biggl( \frac{S_{\overline \mu_{t}}[j]}{N \overline \mu_{t}[j]}
\biggr)^{\iota} 
\biggl(
\sum_{l \in \EN} \vert \overline Y_{t}^l \vert^{\ell}
 \biggr)^{-2}
 ND\overline Y_{t}
\biggr\}
\\
&\hspace{15pt} -
{\mathbf E} \biggl\{ 
\biggl( \frac{S_{\overline \mu_{t}}[j]}{N \overline \mu_{t}[j]}
\biggr)^{\iota} 
\biggl(
\sum_{l \in \EN} \vert \overline Y_{t}^l \vert^{\ell}
+ ND \overline Y_{t}
 \biggr)^{-1}\biggl( \sum_{l \in \EN} \vert \overline Y_{t}^l \vert^{\ell} \biggr)^{-2}
\bigl( ND\overline Y_{t}
\bigr)^2
\biggr\}.
\end{split}
\end{equation}
By the second line in 
\eqref{eq:aux:moments:binomial:1} and since $\overline\mu_t[i]\ge N^{-\varepsilon}$ for any $i\in\ES$ and $t<\overline\tau_N$, we notice that, for any $i \in \ES$, 
\begin{align}
\label{eq:conclusion:1step:y-1}
&\biggl\vert {\mathbf E} \biggl\{ 
\biggl( \frac{S_{\overline \mu_{t}}[j]}{N \overline \mu_{t}[j]}
\biggr)^{\iota} 
\biggl(
\sum_{l \in \EN} \vert \overline Y_{t}^l \vert^{\ell}
 \biggr)^{-2}
 ND\overline Y_{t}
\biggr\}\biggr\vert
\\
&= \biggl(
\sum_{l \in \EN} \vert \overline Y_{t}^l \vert^{\ell}
 \biggr)^{-2}
 \biggl\vert
{\mathbf E} \biggl\{ 
\biggl( \frac{S_{\overline \mu_{t}}[j]}{N \overline \mu_{t}[j]}
\biggr)^{\iota} 
 ND\overline Y_{t}
\biggr\}\biggr\vert
 \leq  C \biggl(
\sum_{l \in \EN} \vert \overline Y_{t}^l \vert^{\ell}
 \biggr)^{-1} \biggl( \frac{1}{N \min_{e \in \ES} \overline\mu_{t}[e]}
 + \frac1{{N^{3/2-2\epsilon}}} \biggr), \nonumber
\end{align}
for a constant $C$ only depending on $\ell$ and the value of which may vary from line to line.  
\vskip 5pt

\noindent \textit{Second Step.} We now split the expectation ${\mathbf E}\{ \cdots\}$ on the last line of 
\eqref{eq:decomposition:inverse:weight} according to the two events 
$E:=\cap_{i \in \ES}
\{ \vert S_{\overline{\mu}_{t}[i]}/N - \overline{\mu}_{t}[i] \vert \leq \eta \overline{\mu}_{t}[i]\}$
and 
$E^{\complement}=\cup_{i \in \ES}
\{ \vert S_{\overline{\mu}_{t}[i]}/N - \overline{\mu}_{t}[i] \vert > \eta \overline{\mu}_{t}[i]\}$,
for some parameter $\eta >0$ whose value is chosen right below.  

On the event $E$, we have
$\vert 
S_{\overline{\mu}_{t}[i]}/(N \overline{\mu}_{t}[i])
- 1 \vert \leq \eta$, 
and we can choose $\eta$ small enough such that 
$1/2 \leq (1-\eta)^{\ell} \leq 1 - \eta \leq 1+ \eta  \leq (1+\eta)^{\ell} \leq 3/2$, from which we deduce that, on 
$E$, 
\begin{equation*}
ND \overline{Y}_{t} \geq - \frac12  \sum_{l \in \EN} \vert \overline Y_{t}^l \vert^{\ell}. 
\end{equation*}
So, {together with Jensen's inequality,} 
\begin{equation}
\label{eq:conclusion:2step:y-1}
\begin{split}
&{\mathbf E} \biggl\{ 
\biggl( \frac{S_{\overline \mu_{t}}[j]}{N \overline \mu_{t}[j]}
\biggr)^{\iota} 
\biggl(
\sum_{l \in \EN} \vert \overline Y_{t}^l \vert^{\ell}
+ ND \overline Y_{t}
 \biggr)^{-1}\biggl( \sum_{l \in \EN} \vert \overline Y_{t}^l \vert^{\ell} \biggr)^{-2}
\bigl( ND\overline Y_{t}
\bigr)^2
{\mathbbm 1}_{E}
\biggr\}
\\
&\leq C {\mathbf E} \biggl\{ 
\biggl( \sum_{l \in \EN} \vert \overline Y_{t}^l \vert^{\ell} \biggr)^{-3}
\bigl( ND\overline Y_{t}
\bigr)^2
\biggr\}
\\
&=  C \biggl(
\sum_{l \in \EN} \vert \overline Y_{t}^l \vert^{\ell}
 \biggr)^{-1}{\mathbf E}\biggl[\biggl( 
\sum_{l \in \EN} \tfrac{\vert \overline Y_{t}^l \vert^{\ell}}{\sum_{k \in \EN} \vert \overline Y_{t}^k \vert^{\ell}}
\sum_{i \in \ES} {\mathbbm 1}_{\{ \overline X_{t}^l=i\}}
\Bigl[ \Bigl( \frac{S_{\overline \mu_{t}}[i]}{N \overline \mu_{t}[i]}
\Bigr)^{\ell}
-1 \Bigr]\biggr)^2\biggr]
 \\
&\le  C \biggl(
\sum_{l \in \EN} \vert \overline Y_{t}^l \vert^{\ell}
 \biggr)^{-1} 
\sum_{l \in \EN} \tfrac{\vert \overline Y_{t}^l \vert^{\ell}}{\sum_{k \in \EN} \vert \overline Y_{t}^k \vert^{\ell}}
\sum_{i \in \ES} {\mathbbm 1}_{\{ \overline X_{t}^l=i\}}
{\mathbf E}\biggl[ \biggl(\Bigl( \frac{S_{\overline \mu_{t}}[i]}{N \overline \mu_{t}[i]}
\Bigr)^{\ell}
-1 \biggr)^2\biggr]
 \\
&\leq  C \biggl(
\sum_{l \in \EN} \vert \overline Y_{t}^l \vert^{\ell}
 \biggr)^{-1}  \frac{1}{N \min_{e \in \ES} \overline\mu_{t}[e]},
\end{split}
\end{equation}
the proof of the last line following from 
\eqref{eq:aux:moments:binomial:2}
with $p=2$ and from the fact that $N \min_{e \in \ES}\bar \mu_{t}[e] \geq 1$
for $t <\overline\tau_N$.
\vskip 5pt

\noindent \textit{Third Step.} We now proceed on the complementary event $E^{\complement}$. 
We observe that 
\begin{equation*}
\frac1N \sum_{l \in \EN} \vert \overline Y_{t}^l \vert^{\ell}
\geq 
\biggl( \frac1N \sum_{l \in \EN}  \overline Y_{t}^l  
\biggr)^{\ell}=1, 
\end{equation*}
and the same lower bound holds true 
for 
$
\frac1N \sum_{l \in \EN} \vert \overline Y_{t}^l \vert^{\ell} + D \overline Y_{t}$ 
since the global weight is preserved by the dynamics, see
\S \ref{subsubse:2:2:1}. 
And then
\begin{equation*}
\begin{split}
&{\mathbf E} \biggl\{ 
\biggl( \frac{S_{\overline \mu_{t}}[j]}{N \overline \mu_{t}[j]}
\biggr)^{\iota} 
\biggl(
\sum_{l \in \EN} \vert \overline Y_{t}^l \vert^{\ell}
+ ND \overline Y_{t}
 \biggr)^{-1}\biggl( \sum_{l \in \EN} \vert \overline Y_{t}^l \vert^{\ell} \biggr)^{-2}
\bigl( ND\overline Y_{t}\bigr)^2 
{\mathbbm 1}_{E^{\complement}}
\biggr\}
\\
&\leq C 
N^{-1} 
{\mathbf E}\biggl\{ 
\biggl( \frac{S_{\overline \mu_{t}}[j]}{N \overline \mu_{t}[j]}
\biggr)^{\iota}
\biggl( \sum_{l \in \EN} \vert \overline Y_{t}^l \vert^{\ell} \biggr)^{-2}
\bigl( ND\overline Y_{t}
\bigr)^2 {\mathbbm 1}_{E^{\complement}}
\biggr\}. 
\end{split}
\end{equation*}
We then use the fact that $t < \overline \tau_{N}$, which implies that 
$1/\min_{j \in \ES}\overline \mu_{t}[j] \leq N^{\epsilon}$
and $\max_{l \in \EN} \overline Y_{t}^{l} \leq N^{1-\epsilon}$. 
This implies that 
$\vert ND \overline Y_{t} \vert \leq (1+ N^{\ell \epsilon}) 
 \sum_{l \in \EN} \vert \overline Y_{t}^l \vert^{\ell}$, from which we get 
\begin{equation}
\label{eq:conclusion:3step:y-1}
\begin{split}
&{\mathbf E} \biggl\{ 
\biggl( \frac{S_{\overline \mu_{t}}[j]}{N \overline \mu_{t}[j]}
\biggr)^{\iota} 
\biggl(
\sum_{l \in \EN} \vert \overline Y_{t}^l \vert^{\ell}
+ ND \overline Y_{t}
 \biggr)^{-1}\biggl( \sum_{l \in \EN} \vert \overline Y_{t}^l \vert^{\ell} \biggr)^{-2}
 \bigl( 
ND\overline Y_{t}
\bigr)^2
{\mathbbm 1}_{E^{\complement}}
\biggr\}
\leq C 
N^{(2\ell+1)\epsilon -1} 
{\mathbf P}\bigl( E^{\complement}
\bigr). 
\end{split}
\end{equation}
By Hoeffding's inequality as in the proof of 
Lemma 
\ref{lem:aux:moments:binomial}, we can find a constant $c>0$ (depending on $\eta$)
so that 
${\mathbf P} ( E^{\complement}) \leq 
C \exp (- 2 N (\min_{i \in \ES} \overline \mu_{t}[i])^2)
\leq C \exp(-2 N^{1-2\epsilon})$.  
\vskip 5pt

\noindent \textit{Conclusion.}
By combining 
\eqref{eq:theorem:inverse:y:proof:1}, 
\eqref{eq:conclusion:1step:y-1}, 
\eqref{eq:conclusion:2step:y-1}
and
\eqref{eq:conclusion:3step:y-1} (multiplying the former by $N$), we end up with
\begin{equation*}
\begin{split}
&\frac{d}{dt} {\mathbb E} \biggl[ {\mathcal E}_{t \wedge \overline \tau_{N}}^{-1} 
\biggl(\frac1N
 \sum_{l \in \EN} \vert \overline Y_{t \wedge \overline \tau_{N}}^l \vert^{\ell} \biggr)^{-1}
\biggr]
+\lambda
{\mathbb E} \biggl[ 
{\mathbbm 1}_{\{ t < \overline \tau_{N}\}}
{\mathcal E}_{t}^{-1}
\sum_{i \in \ES}
\frac{1}{\overline \mu_{t}[i]}
\biggl(\frac1N \sum_{l \in \EN} \vert \overline Y_{t}^l \vert^{\ell} \biggr)^{-1}
\biggr]
\\
&\leq 
 C \varepsilon {\mathbb E} \biggl[ 
 {\mathbbm 1}_{\{ t < \overline \tau_{N}\}}
  {\mathcal E}_{t}^{-1} \biggl( \frac1N 
\sum_{l \in \EN} \vert \overline Y_{t}^l \vert^{\ell}
 \biggr)^{-1} \biggl( \frac{1}{ \min_{e \in \ES} \overline\mu_{t}[e]}
 + \frac1{N^{(1-\epsilon)/2}} \biggr) \biggr]
 + C \exp(- N^{1-2\epsilon}).
\end{split}
\end{equation*}
Notice that, since {$\epsilon < 1/4$}, we have $1/N^{(1-\epsilon)/2} \leq 1/N^{1/3} \leq 1/N^{\epsilon} \leq 
1/ \min_{e \in \ES} \overline\mu_{t}[e]$ (recalling that $t$ is here less than 
$\overline \tau_{N}$). Hence, by choosing $\lambda$ large enough
with respect to $C$
(which is indeed possible since $C$ only depends on $m$), we deduce that 
\begin{equation*}
\begin{split}
&{\mathbb E} \biggl[ {\mathcal E}_{t \wedge \overline \tau_{N}}^{-1} 
\biggl(\frac1N
 \sum_{l \in \EN} \vert \overline Y_{t \wedge \overline \tau_{N}}^l \vert^{\ell} \biggr)^{-1}
\biggr]
\leq 
\biggl(\frac1N
 \sum_{l \in \EN} \vert y_{0}^l \vert^{\ell} \biggr)^{-1}
 + C \exp(-c N^{1-2\epsilon}).
\end{split}
\end{equation*}
It now remains to insert 
\eqref{expboundmu}--\eqref{bound_calE} ($\kappa$ being implicitly taken large enough), from which we get 
\begin{align*}
&{\mathbb E} \biggl[ 
\biggl(\frac1N
 \sum_{l \in \EN} \vert \overline Y_{t \wedge \overline \tau_{N}}^l \vert^{\ell} \biggr)^{-1}
\biggr]
\\
& \leq 
{\mathbb E} \biggl[ 
{\mathcal E}_{t \wedge \overline \tau_{N}}^{-1}
\biggl(\frac1N
 \sum_{l \in \EN} \vert \overline Y_{t \wedge \overline \tau_{N}}^l \vert^{\ell} \biggr)^{-2}
\biggr]^{1/2}
{\mathbb E} \bigl[ 
{\mathcal E}_{t \wedge \overline \tau_{N}} 
\bigr]^{1/2}
\\
&\leq 
{\mathbb E} \biggl[ 
{\mathcal E}_{t \wedge \overline \tau_{N}}^{-1}
\biggl(\frac1N
 \sum_{l \in \EN} \vert \overline Y_{t \wedge \overline \tau_{N}}^l \vert^{\ell} \biggr)^{-1}
\biggr]^{1/2}
{\mathbb E} \bigl[ 
{\mathcal E}_{t}
\bigr]^{1/2}
\\
&\leq C
\biggl[\biggl(\frac1N
 \sum_{l \in \EN} \vert y_{0}^l \vert^{\ell} \biggr)^{-1/2}
 + \exp(-c N^{1-2\epsilon}) \biggr] 
  \prod_{i \in \ES}
\bigl(
N^{-\epsilon}+\overline \mu_{0}[i] \bigr)^{-1/(2d)},
\end{align*}
where, in the second line, we used the (already proved) fact that  $N^{-1}
 \sum_{l \in \EN} \vert \overline Y_{t}^l \vert^{\ell} \geq 1$. \qed

\section{Proofs of the estimates connecting the Nash system with the master equation}
\label{se:7:b}
This section is devoted to the proofs of the various lemmas that enter the demonstration of Proposition \ref{prop:u}. We recall that, 
$U$ being defined as an element of ${\mathcal C}^{1+\gamma'/2,2+\gamma'}$, we just know that 
$\sqrt{p^j p^k} \fd^2_{i,j} U^i$ is bounded and $\gamma'$-H\"older continuous in space (for the Wright--Fischer distance), 
for each $i,j,k \in \dd$. By the way, we recall that $\gamma'$-H\"older continuity in space for the Wright--Fischer distance implies $\gamma'/2$-H\"older continuity in space for the standard Euclidean distance. 
We use the latter property quite often in the section. 
We refer if needed to the monograph \cite{EpsteinMazzeo} for a complete review on all these facts

\subsection{Proof of Lemma \ref{Nash:approximate:common:noise}}
\label{subse:7:1}
\begin{proof} 
Throughout the proof, we fix $t \in [0,T]$ and $(\bm{x},\bm{y}) \in {\mathcal T}_{N}$. This permits us to let 
$\mu = \mu^N_{\bm{x},\bm{y}}$, 
$\mu^j = \mu^N_{\bm{x},\bm{y}}[j]$ and $S^j=S_{\mu}[j]$.
Also, fixing $l \in \NN$, we may denote $x^l$ by $i$, namely 
$i:=x^l$. Then, notice that 
\[
u^{N,l}\biggl(t,\bm{x}, y^1\frac{S_{\mu^N_{\bm{x},\bm{y}}}[x^1]}{N\mu^N_{\bm{x},\bm{y}}[x^1]},\dots, y^N\frac{S_{\mu^N_{\bm{x},\bm{y}}}[x^N]}{N\mu^N_{\bm{x},\bm{y}}[x^N]}\biggr) 
= y^l \frac{S^i}{N\mu^i} U^i\Big(t, \frac{S}{N}\Big),
\]
where, on the right-hand side, we identified 
the probability measure 
$\mu^N_{\bm{x},\bm{z}}$, 
for
$\bm{z}$ being given by 
$z^l = y^l 
S_{\mu^N_{\bm{x},\bm{y}}}[x^l]/(N\mu^N_{\bm{x},\bm{y}}[x^l])$, 
with 
$S/N$, which is made licit by 
the identity \eqref{eq:sum:proba} (which guarantees that 
$\mu^N_{\bm{x},\bm{z}}$ is indeed a probability measure). 
We expand by Taylor's formula in the form 
\begin{align*}
U^i\bigl(t,\frac{S}{N}\bigr) & = U^i(t,\mu) + \sum_{j\in\dd} \fd_{j} U^i(t,\mu) \bigg(\frac{S^j}{N} -\mu^j\bigg) \\
&+ \sum_{j,k\in\dd} 
\bigg(\frac{S^j}{N} -\mu^j\bigg) \bigg(\frac{S^k}{N} -\mu^k\bigg)
\int_0^1  (1-r) \fd^2_{j,k} U^i \Bigl(t, \mu + r \Big(\frac{S}{N} -\mu\Big)\Bigr)  dr,
\end{align*}
and then, adding and subtracting  $\fd^2_{j,k} U^i(t,\mu)$ inside the integral, we obtain 
\begin{align*}
&\frac{1}{y^l}  N \EE  
\biggl[ u^{N,l}\biggl(t,\bm{x}, y^1\frac{S_{\mu^N_{\bm{x},\bm{y}}}[x^1]}{N\mu^N_{\bm{x},\bm{y}}[x^1]},\dots, y^N\frac{S_{\mu^N_{\bm{x},\bm{y}}}[x^N]}{N\mu^N_{\bm{x},\bm{y}}[x^N]}\biggr)
 -u^{N,l}(t,\bm{x},\bm{y}) \biggr]
 \\
 & = N \EE  
\biggl[ \frac{S^i}{N \mu^i} \biggl( U^i\bigl(t,\frac{S}{N}\bigr) - U^i(\mu) \biggr) \biggr]
\\
&=  N \EE  
\bigg[ \frac{S^i}{N \mu^i}  
\sum_{j\in\dd} \fd_{j} U^i(t,\mu) \bigg(\frac{S^j}{N} -\mu^j\bigg) \bigg] 
\\
&\quad + \frac12 N \EE  
\bigg[ \frac{S^i}{N \mu^i}  
\sum_{j,k\in\dd} \bigg(\frac{S^j}{N} -\mu^j\bigg) \bigg(\frac{S^k}{N} -\mu^k\bigg) \fd^2_{j,k} U^i(t,\mu) \bigg] 
\\
&\quad + N \EE  
\bigg[ \frac{S^i}{N \mu^i}  
\sum_{j,k\in\dd} \biggl\{  \bigg(\frac{S^j}{N} -\mu^j\bigg) \bigg(\frac{S^k}{N} -\mu^k\bigg)
\\
&\quad \quad \times
\int_0^1  (1-r) \biggl[ \fd^2_{j,k} U^i \Bigl(t, \mu + r \bigl(\frac{S}{N} -\mu\bigr)\Bigr) 
- \fd^2_{j,k} U^i(t,\mu) \biggr] dr
\biggr\}
\bigg] 
\\
& =: \mbox{I} + \mbox{II} + R^{N,l}.
\end{align*}
The first term becomes
\begin{align*}
\mbox{I}&=  N \sum_{j\in\dd}\EE 
\bigg[\frac{S^i- N\mu^i}{N\mu^i} \fd_{j} U^i(t,\mu) \Bigl(\frac{S^j}{N} -\mu^j\Bigr) 
+ \fd_{j} U^i(t,\mu) \Bigl(\frac{S^j}{N} -\mu^j\Bigr)
\bigg]
\\
&=   N \sum_{j\in\dd}
\frac{ \mu^i\delta_{j,i} -\mu^i\mu^j}{N \mu^i}  \fd_{j} U^i(t,\mu) + 0
\\
&= \sum_{j\in\dd}
 (\delta_{j,i} -\mu^j )\fd_{j} U^i(t,\mu),
\end{align*}
while the second term is 
\begin{align*}
\mbox{II}&= \frac{ N }{2}
\sum_{j,k\in\dd}  \fd^2_{j,k} U^i(t,\mu)
\EE\bigg[\bigg(\frac{S^j}{N} -\mu^j\bigg) \bigg(\frac{S^k}{N} -\mu^k\bigg) \bigg]
\\
&\quad + \frac{N }{2 N \mu^i}
\sum_{j,k\in\dd}  \fd^2_{j,k} U^i(t,\mu)
\EE\bigg[\left(S^i - N\mu^i\right) \bigg(\frac{S^j}{N} -\mu^j\bigg) \bigg(\frac{S^k}{N} -\mu^k\bigg) \bigg]
\\
&= \frac{1}{2}
\sum_{j,k\in\dd} \left(\mu^j \delta_{j,k} -\mu^j\mu^k\right)
 \fd^2_{j,k} U^i(t,\mu) +O\left(\frac1N\right),
\end{align*}
where in the latter term we used the property of the multinomial distribution
\[
\EE\bigl[\bigl(S_i-N\mu_i\bigr)\bigl(S^j-N\mu^j\bigr)\bigl(S^k-N\mu^k\bigr)\bigr] = \begin{cases}
N\bigl(2 (\mu^i)^3-3(\mu^i)^2+\mu^i\bigr)   & \mbox{ if } i=j=k
\\
N\mu^i\mu^j(2\mu^i-1)  & \mbox{ if } i=k\neq j
\\
2N \mu^i\mu^j\mu^k  & \mbox{ if } i\neq j \neq k
\end{cases},
\]
which is of order $N$.

It remains to estimate the rest $R^{N,l}$, which is of the form $R^{N,l} = \sum_{j,k\in \dd} R_{j,k}$ (we feel better to remove the superscripts $N,l$ in the notation). Thus, in the following, we fix $j$ and $k$ and we estimate $|R_{j,k}|$.
We will use several times Rosenthal's inequality, see 
 \eqref{eq:rosenthal}: for any real $p\geq 2$ and any $i\in \dd$, it yields
\be 
\EE \Big[ \bigl|S^i - N \mu^i \bigr|^p \Big]  \leq  N^{p/2} \mu^i. 
\ee
In order to use the properies of $[\mathcal{C}^{1+\gamma'/2,2+\gamma'}_{WF}([0,T] \times \mathcal{S}_{d-1})]^d$, we multiply and divide the integrand in $R_{j,k}$ by 
$$Q_{j,k}(r):=\sqrt{( \mu^j + r(S^j/N -\mu^j)) ( \mu^k + r(S^k/N -\mu^k) ) },$$ 
which is not zero since 
$(\bm{x},\bm{y}) \in {\mathcal T}_{N}$, 
so that we obtain 
\begin{align*}
R_{j,k} &= N \EE  
\bigg[ \frac{S^i}{N \mu^i}  
\bigg(\frac{S^j}{N} -\mu^j\bigg) \bigg(\frac{S^k}{N} -\mu^k\bigg) 
\\
&\quad \quad \times 
\int_0^1 \frac{ (1-r)}{Q_{j,k}(r)}   \biggl( Q_{j,k}(r)\fd^2_{j,k} U^i \Bigl(t, \mu + r \Bigl(\frac{S}{N} -\mu\Bigr)\Bigr) 
- \sqrt{\mu^j \mu^k} \fd^2_{j,k} U^i(t,\mu) 
\\
&\quad \quad \quad \quad -\Big(Q_{j,k}(r) - \sqrt{\mu^j \mu^k} \Big) \fd^2_{j,k} U^i(t,\mu) \biggr) dr
\Bigg] 
\\
&=: R^1_{j,k} + R^2_{j,k}.
\end{align*}
The denominator is bounded by $Q_{j,k}(r) \geq (1-r)\sqrt{\mu^j \mu^k}$, and thus applying H\"older and Rosenthal inequalities we get
\begin{align*}
|R^1_{j,k}|& \leq  
C N \EE  
\bigg[ \frac{S^i}{N \mu^i}  
\bigg|\frac{S^j}{N} -\mu^j\bigg| \bigg|\frac{S^k}{N} -\mu^k\bigg|  
\int_0^1 \frac{r}{\sqrt{\mu^j \mu^k}}  \sum_{e\in \dd} \bigg|\frac{S^e}{N} -\mu^e\bigg|^{\frac{\gamma'}2} dr \bigg]
\\
& \leq C N \frac{1}{ N \mu^i \sqrt{\mu^j \mu^k}}  \sum_{e\in \dd}
\EE  
\bigg[ |S^i -N\mu^i| 
\bigg|\frac{S^j}{N} -\mu^j\bigg| \bigg|\frac{S^k}{N} -\mu^k\bigg|  
 \bigg|\frac{S^e}{N} -\mu^e\bigg|^{\frac{\gamma'}2} \bigg] 
 \\
 &\quad \quad + CN \frac{1}{  \sqrt{\mu^j \mu^k}}  \sum_{e\in \dd}
\EE  
\bigg[ 
\bigg|\frac{S^j}{N} -\mu^j\bigg| \bigg|\frac{S^k}{N} -\mu^k\bigg|  
 \bigg|\frac{S^e}{N} -\mu^e\bigg|^{\frac{\gamma'}2} \bigg]
 \\
 &=  \frac{C}{ N^{2+\frac{\gamma'}2} \mu^i \sqrt{\mu^j \mu^k}}  \sum_{e\in \dd}
\EE  
\Big[ |S^i -N\mu^i| 
\big|S^j -N \mu^j\big| \big| S^k  - N \mu^k\big|  
 \big| S^e  -N \mu^e\big|^{\frac{\gamma'}2} \Big] 
 \\
 &\quad \quad +  \frac{C}{ N^{1+\frac{\gamma'}2} \sqrt{\mu^j \mu^k}}  \sum_{e\in \dd}
\EE  
\Big[ 
\big|S^j -N \mu^j\big| \big| S^k  - N \mu^k\big|  
 \big| S^e  -N \mu^e\big|^{\frac{\gamma'}2} \Big] 
 \\
 &\leq  \frac{C}{ N^{2+\frac{\gamma'}2} \mu^i \sqrt{\mu^j \mu^k}}  \sum_{e\in \dd}
\Big( \prod_{\ell=i,j,k} \EE  
\Big[ |S^\ell -N\mu^\ell|^{\frac{3}{1-\gamma'/6}} \Big] \Big)^{\frac{1-\gamma'/6}{3}}
 \Big( \EE | S^e  -N \mu^e \big| \Big)^{\frac{\gamma'}2} 
 \\
 &\quad \quad +  \frac{C}{ N^{1+\frac{\gamma'}2} \sqrt{\mu^j \mu^k}}  \sum_{e\in \dd}
\Big( \prod_{\ell=j,k} \EE  
\Big[ |S^\ell -N\mu^\ell|^{\frac{2}{1-\gamma'/4}} \Big] \Big)^{\frac{1-\gamma'/4}{2}}
 \Big( \EE | S^e  -N \mu^e \big| \Big)^{\frac{\gamma'}2} 
 \\
 &\leq \frac{C}{ N^{2+\frac{\gamma'}2} \mu^i \sqrt{\mu^j \mu^k}} N^{\frac{3}2 + \frac{\gamma'}4} (\mu^i \mu^j \mu^k)^{\frac{1-\gamma'/6}{3}} 
 +  \frac{C}{ N^{1+\frac{\gamma'}2} \sqrt{\mu^j \mu^k}} N^{1+ \frac{\gamma'}4}(\mu^j \mu^k)^{\frac{1-\gamma'/4}{2}} 
 \\
 &= \frac{C}{N^{\frac12 +\frac{\gamma'}{4}} (\mu_i)^{\frac{2+\gamma'/6}{3}} (\mu^j \mu^k)^{\frac{1+2\gamma'/6}{6}} } 
 + \frac{C}{N^{\frac{\gamma'}{4}} (\mu^j \mu^k)^{\frac{\gamma'}{8}}} 
 \\
 & \leq \frac{C}{N^{\frac{1+\gamma'/2}{2} - \epsilon\left(1+\frac{\gamma'}{6} \right)}}
 + \frac{C}{N^{\frac{\gamma'}{4} - \epsilon\gamma'/2}} 
 \leq \frac{C}{N^\eta},
\end{align*}
for
$\epsilon\in (0, {1/4} )$ and for $\eta>0$, depending on 
$\epsilon$ and $\gamma'$.  

The term $R^2_{j,k}$, using $\sqrt{\mu^j\mu^k} | \fd^2_{j,k} U^i(\mu) |  \leq C$, is bounded as
\begin{align*}
|R^2_{j,k}|& \!\leq  \!
N \EE  
\Bigg[ \frac{S^i}{N \mu^i}  
\bigg|\frac{S^j}{N} -\mu^j\bigg| \bigg|\frac{S^k}{N} -\mu^k\bigg|  \cdot\\
&\quad \cdot \int_0^1 \!\frac{1}{\sqrt{\mu^j \mu^k}}  
\sqrt{\bigg| \bigg( \mu^j + r\bigg(\frac{S^j}{N} -\mu^j \!\bigg) \bigg) \bigg( \mu^k + r\bigg(\frac{S^k}{N} -\mu^k \!\bigg)\bigg) - \!\mu^j\mu^k \bigg|} |\fd^2_{j,k} U^i(\mu) |
 dr \Bigg]\\
 &\leq \frac{CN}{\mu^j \mu^k}
 \EE  
\Bigg[ \frac{S^i}{N \mu^i}  
\bigg|\frac{S^j}{N} -\mu^j\bigg| \bigg|\frac{S^k}{N} -\mu^k\bigg|   \int_0^1  
\sqrt{\bigg|  r\bigg(\frac{S^j}{N} -\mu^j \!\bigg) \mu^k + r\bigg(\frac{S^k}{N} -\mu^k \!\bigg) \mu^j \bigg|}
 dr \Bigg] \\
 &\quad +  \frac{CN}{\mu^j \mu^k}
 \EE  
\Bigg[ \frac{S^i}{N \mu^i}  
\bigg|\frac{S^j}{N} -\mu^j\bigg| \bigg|\frac{S^k}{N} -\mu^k\bigg|   \int_0^1  
r \sqrt{ \bigg|\frac{S^j}{N} -\mu^j\bigg| \bigg|\frac{S^k}{N} -\mu^k\bigg| }
 dr \Bigg] \\
 &=: (A) +(B).
\end{align*}
To estimate $(A)$, we bound the term with $r(S^j/N -\mu^j) \mu^k$ inside the square root; the other term is analogous. We have
\begin{align*}
& \frac{CN}{\mu^j \mu^k}
 \EE  
\Bigg[ \frac{S^i}{N \mu^i}  
\bigg|\frac{S^j}{N} -\mu^j\bigg| \bigg|\frac{S^k}{N} -\mu^k\bigg|   \bigg|  \frac{S^j}{N} -\mu^j \bigg|^{\frac12} \sqrt{\mu^k}  \Bigg] \\
&\leq \frac{CN}{\mu^j \sqrt{\mu^k} N \mu^i }
\EE  
\Bigg[ |S^i -N\mu^i|
\bigg|\frac{S^j}{N} -\mu^j\bigg|^{\frac32} \bigg|\frac{S^k}{N} -\mu^k\bigg| \Bigg]
+ \frac{CN}{\mu^j \sqrt{\mu^k} }
\EE  
\Bigg[ 
\bigg|\frac{S^j}{N} -\mu^j\bigg|^{\frac32} \bigg|\frac{S^k}{N} -\mu^k\bigg| \Bigg] \\
&= \frac{C}{\mu^i \mu^j \sqrt{\mu^k} N^{\frac52} } 
\EE  
\Big[ |S^i -N\mu^i| |S^j -N\mu^j|^{\frac32} |S^k -N\mu^k| \Big] \\
& \quad
 + \frac{C}{ \mu^j \sqrt{\mu^k} N^{\frac32} } 
\EE  
\Big[ |S^j -N\mu^j|^{\frac32} |S^k -N\mu^k| \Big] \\
& \leq \frac{C}{\mu^i \mu^j \sqrt{\mu^k} N^{\frac52} } 
\Big( \EE  
\Big[ |S^i -N\mu^i|^{3} \Big] \EE  
\Big[|S^j -N\mu^j|^{\frac92} \Big]
\EE  
\Big[|S^k -N\mu^k|^3 \Big] \Big)^{\frac13}\\
&\quad + \frac{C}{ \mu^j \sqrt{\mu^k} N^{\frac32} } 
\Big( \EE  
\Big[ |S^j -N\mu^j|^{3} \Big] \EE \Big |S^k -N\mu^k|^2 \Big] \Big)^{\frac12} \\
&\leq \frac{C}{\mu^i \mu^j \sqrt{\mu^k} N^{\frac52} } N^{\frac74} (\mu^i\mu^j\mu^k)^{\frac13} 
+ \frac{C}{ \mu^j \sqrt{\mu^k} N^{\frac32} } N^{\frac54} \sqrt{\mu^j\mu^k} \\
&= \frac{C}{N^{\frac34} (\mu^i \mu^j)^{\frac23} (\mu^k)^{\frac16}} 
+ \frac{C}{N^\frac14 (\mu^j)^\frac12} \\
&\leq  \frac{C}{N^{\frac34 -  \frac32\epsilon } } 
+ \frac{C}{N^{\frac14 - \frac12\epsilon} } \leq \frac{C}{N^\eta},
\end{align*}
for $\epsilon \in (0, {1/4})$  and for a possibly new value of $\eta$. And, then, $(A)\leq C/N^\eta$.
The term $(B)$ is estimated in the same way:
\begin{align*}
(B) &\leq \frac{CN}{\mu^j \mu^k}
 \EE  
\Bigg[ \frac{S^i}{N \mu^i}  
\bigg|\frac{S^j}{N} -\mu^j\bigg|^{\frac32} \bigg|\frac{S^k}{N} -\mu^k\bigg|^{\frac32}   \Bigg] \\
&\leq \frac{CN}{N \mu^i \mu^j \mu^k N^3}
 \EE  
\Big[ |S^i- N \mu^i|  |S^j- N \mu^j|^\frac32 |S^k- N \mu^k|^\frac32 \Big] \\
& \quad + \frac{CN}{ \mu^j \mu^k N^3} \EE  
\Big[  |S^j- N \mu^j|^\frac32 |S^k- N \mu^k|^\frac32 \Big] \\
&\leq \frac{C}{ \mu^i \mu^j \mu^k N^3} N^2 (\mu^i \mu^j \mu^k)^\frac13
+  \frac{C}{ \mu^j \mu^k N^2} N^{\frac32} (\mu^j \mu^k)^{\frac12} \\
& \leq \frac{C}{N^{1-2\epsilon }} 
+  \frac{C}{N^{\frac12 -\epsilon }} \leq \frac{C}{ N^\eta},
\end{align*} 
which concludes the proof.
\end{proof}

\subsection{Proofs of Lemmas \ref{lem:exp:first:order:variation}, \ref{prop:nash:1st:order} and 
\ref{lem:expansion:H}}
\label{subse:7:2}

\begin{proof}[Proof of Lemma \ref{lem:exp:first:order:variation}]
We just prove \eqref{deltanl}; the proof of \eqref{deltall} is similar.
As in the proof of Lemma \ref{Nash:approximate:common:noise}, we fix $t \in [0,T]$ and $(\bm{x},\bm{y}) \in \ES^N \times \setQ$. We then let 
$\mu = \mu^N_{\bm{x},\bm{y}}$
and, fixing $l,n \in \NN$ with $m \not =l$, we denote $x^l$ by $i$ and $x^m$ by $k$, namely 
$i:=x^l$ and $k=x^m$.

We then have the following expansion:
\begin{align*}
&\frac{1}{y^l} \Delta^m u^{N,l}(t,\bm{x}, \bm{y})[j] 
\\
&= U^{i}\Big( t,\mu+ \frac{y^m}{N}(\delta_j - \delta_{k}) \Big) - U^i( t,\mu)
\\
&= \int_0^1 
\frac{y^m}{N} \Bigl[  \fd_{j} U \Big(t, \mu + s \frac{y^m}{N}(\delta_j - \delta_{k}) \Big) 
- 
\fd_{  k} U \Big(t, \mu + s \frac{y^m}{N}(\delta_j - \delta_{k}) \Big)
\Bigr]
 ds 
\\
&= \frac{y^m}{N} \Bigl[\fd_{j} U^{i}(t, \mu) - \fd_{{k}} U^{i}(t, \mu) \Bigr] 
\\
&\hspace{15pt} + \int_0^1 \frac{y^m}{N} \Bigl[\fd_{j}  U^i \Big(t, \mu + s \frac{y^m}{N}(\delta_j - \delta_{k}) \Big) -\fd_{j} U^{i}(t, \mu) 
 - \fd_{{k}}  U^i \Big( \mu + s \frac{y^m}{N}(\delta_j - \delta_{k}) \Big)  + \fd_{{k}} U^{i}(t, \mu) \Bigr]ds,
\end{align*}
 and the last two lines may be written in the form of a rest $\varrho^{N,l,m}(t,\bm{x},\bm{y})[j]$. Since $\fD U^i$ is $\gamma'$-H\"older continuous (for the Wright--Fisher distance), this remainder is bounded by
\[
\bigl|\varrho^{N,l,m}(t,\bm{x},\bm{y})[j]\bigr|
 \leq C \frac{(y^m)^{1+\gamma'/2}}{N^{1+\gamma'/2}},
\]
which completes the proof.
\end{proof}

\begin{proof}[Proof of Lemma \ref{prop:nash:1st:order}.]
We first address the second term 
in the top line of 
\eqref{eq:nash:1st:order}, namely the term containing $a^*$. As in the proofs of the previous two statements, 
we fix $t \in [0,T]$, $(\bm{x},\bm{y}) \in {\mathcal T}_{N}$
and $l \in \NN$, 
 and we let 
$\mu = \mu^N_{\bm{x},\bm{y}}$ and $i=x^l$. 

By the definition of $a^*$ (see \eqref{eq:def:Hamiltonian}), we have
\begin{equation*}
\frac{1}{y^l}  \sum_{m\neq l} a^*\Big(x^m, \frac{1}{y^m} u^{N,m}_{\bullet}\Big)\cdot \Delta^m u^{N,l} [\bullet]
= \sum_{m\neq l} \sum_{j\in\dd} \frac{1}{y^m}
\Bigl(- \Delta^m u^{N,m} [j] \Bigr)_+ \frac{1}{y^l} \Delta^m u^{N,l} [j] 
\end{equation*}
and then by Lemma \ref{lem:exp:first:order:variation},
\begin{align*}
&\frac{1}{y^l}  \sum_{m\neq l} a^*\Big(x^m, \frac{1}{y^m} u^{N,m}_{\bullet}\Big)\cdot \Delta^m u^{N,l} [\bullet]
\\
&= \sum_{m\neq l} \sum_{j \in \ES} 
\biggl[
 \Big( \bigl( U^{x^m}- U^j \bigr) (t,\mu)  -\frac{y^m}{N} \big(
 \fd_{j} U^{x^m} - \fd_{{x^m}} U^{x^m} \big)(t,\mu) -\varrho^{N,m,m}(t,\bm{x},\bm{y})[j]  \Big)_+
 \\
 &\hspace{90pt} \times 
 \frac{y^m}{N} \Big( \bigl( \fd_{j} U^{i} - \fd_{{x^m}} U^{i}\bigr)(t, \mu) + \varrho^{N,l,m}(t,\bm{x},\bm{y})[j] \Big) 
 \biggr]
 \\
 &= \sum_{m\neq l} \sum_{j,k\in\dd}\hspace{-3pt} \one_{\{x^m=k\}} 
 \biggl[ \Big( \bigl( U^{k} -U^j\bigr)(t,\mu) -\frac{y^m}{N} \big(
 \fd_{j} U^{k} - \fd_{k} U^{k} \big)(t,\mu) -\varrho^{N,m,m}(t,\bm{x},\bm{y})[j]   \Big)_+
 \\
 &\hspace{90pt} \times  
 \frac{y^m}{N} \Big( \bigl( \fd_{j} U^{i} - \fd_{k} U^{i}\bigr)(t,\mu) + \varrho^{N,l,m}(t,\bm{x},\bm{y})[j] \Big) \biggr].
 \end{align*}
And then,
 \begin{align*}
 &\frac{1}{y^l}  \sum_{m\neq l} a^*\Big(x^m, \frac{1}{y^m} u^{N,m}_{\bullet}\Big)\cdot \Delta^m u^{N,l}[\bullet] 
\\
 &= \sum_{m\neq l} \sum_{j,k\in\dd} \one_{\{x^m=k\}} 
 \bigl( (U^{k} -U^j)(t,\mu) \bigr)_+
 \frac{y^m}{N} \big(\fd_{j} U^{i} - \fd_{k} U^{i}  \big)(t,\mu) 
 \\
 &\hspace{1pt}+ \sum_{m\neq l} \sum_{j,k\in\dd} \biggl[ \one_{\{x^m=k\}} 
  \Big( \bigl( U^{k} -U^j\bigr)(t,\mu) -\frac{y^m}{N} \big(
 \fd_{j} U^{k} - \fd_{k} U^{k} \big)(t,\mu) -\varrho^{N,m,m}(t,\bm{x},\bm{y})[j]  \Big)_+
 \\
&\hspace{60pt} \times \varrho^{N,l,m}(t,\bm{x},\bm{y})[j] \biggr]
 \\
 &\hspace{1pt}+ \sum_{m\neq l} \sum_{j,k\in\dd} \one_{\{x^n=k\}} 
\bigg[ \Big( \bigl( U^{k} -U^j \bigr)(t,\mu) -\frac{y^m}{N} \big(
 \fd_{j} U^{k} - \fd_{k} U^{k} \big)(t,\mu) -\varrho^{N,m,m}(t,\bm{x},\bm{y})[j]  \Big)_+ 
 \\
&\hspace{60pt} - \bigl((U^k -U^j)(t,\mu)\bigr)_+ \biggr]
 \frac{y^m}{N} \big(\fd_{j} U^{i} - \fd_{k} U^{i} \big)(t,\mu) 
 \\
 &:= \sum_{k,j\in\dd} \mu^k \bigl( (U^{k} -U^j)(t,\mu) \bigr)_+
  \big(\fd_j U^{i} - \fd_{k} U^{i}  \big)(t,\mu) + R_1 +R_2.
\end{align*}
Using the boundedness of $U$ and $\fD U$,  the first remainder is bounded by 
\begin{align*}
|R_1|&\leq C\sum_{m\in \EN } 
\sup_{j \in \ES}
\bigl|\varrho^{N,l,m}(t,\bm{x},\bm{y})[j]\bigr|  +C \sum_{m \in \EN} \frac{y^m}{N}
\sup_{j \in \ES}\bigl|\varrho^{N,l,m}(t,\bm{x},\bm{y})[j] \bigr|   
\\
&\hspace{15pt}  + C \sum_{m \in \EN} 
\sup_{j \in \ES} \bigl| \bigl( \varrho^{N,m,m} \varrho^{N,l,m}\bigr) (t,\bm{x},\bm{y})[j]\bigr| 
\\
&\leq C \sum_{m \in \EN} \frac{(y^m)^{1+\gamma'/2}}{N^{1+\gamma'/2}} + 
C \sum_{m \in \EN} \frac{(y^m)^{2+\gamma'/2}}{N^{2+\gamma'/2}} 
+C \sum_{m \in \EN} \frac{(y^m)^{2+2\gamma'/2}}{N^{2+2\gamma'/2}},
\end{align*}
while the second remainder is bounded by
\begin{align*}
|R_2|\leq C\sum_{m \in \EN} \Big( \frac{y^m}{N} + \sup_{j \in \ES} \bigl|\varrho^{N,m,m}(t,\bm{x},\bm{y})[j]\bigr| \Big)  \frac{y^m}{N}
\leq C \sum_{m \in \EN} \frac{(y^m)^{2}}{N^{2}} + 
C \sum_{m \in \EN} \frac{(y^m)^{2+\gamma'/2}}{N^{2+\gamma'/2}} .
\end{align*}
Finally, since $y^m \leq N$, we deduce that $\vert R_{1} \vert+ \vert R_{2} \vert$ is less than 
$C \sum_{m \in \EN} (y^m)^{1+\gamma'/2}/N^{1+\gamma'/2}$. 
By definition of ${\mathcal T}_{N}$ in  
\eqref{eq:mathcal:TN}, we then recall that 
$\max_{m \in\NN} y^m \leq N^{1-\epsilon}$, which yields
\be
\sum_{m \in \EN} \frac{(y^m)^{1+\gamma'/2}}{N^{1+\gamma'/2}}
\leq 
N^{-\gamma' \epsilon/2}
\sum_{m \in \EN} \frac{y^m}{N}
= N^{-\gamma' \epsilon/2},
\ee
where we used in the last line the fact that $\sum_{m \in \EN} y^m=N$.
This shows that the second term on the top line of 
\eqref{eq:nash:1st:order}
satisfies the bound \eqref{eq:rn2l}.

The first term on the top line of 
\eqref{eq:nash:1st:order}
is handled in the same way, using the additional
expansion 
\eqref{deltall}
to treat the case when $m$ (the index in the sum) is equal to $l$. 
\end{proof}

  \begin{proof}[Proof of Lemma \ref{lem:expansion:H}.]
   	Applying
	\eqref{eq:def:Hamiltonian}
	and
	 \eqref{deltall}, we obtain (denote $\mu^N_{\bm{x},\bm{y}} =\mu$ and $x^l=i$)
   	\begin{align*}
   	H\Big(x^l, \frac{1}{y^l}u^{N,l}_\bullet \Big) 
	&= - \frac12 
   	\sum_{j\in\dd}  \Bigl(- \frac{1}{y^l} \Delta^l u^{N,l}(t,\bm{x}, \bm{y})[j]\Bigr)_+ ^2 
	\\
   	&= - \frac12 
   	\sum_{j\in\dd} \Bigl( \bigl(U^i-U^{j}\bigr)( t,\mu)  
   - \frac{y^l}{N} \bigl(\fd_j U^{i} - \fd_{i} U^{i}\bigr)(t,\mu)  - \varrho^{N,l,l}(t, \bm{x},\bm{y})[j] \Big)_+^2 \\
   &=- \frac12 
   \sum_{j\in\dd} \Bigl( \bigl(U^i -U^{j}\bigr)(t, \mu) \Bigr)_+^2 + r_3^{N,l}.
   	\end{align*}
   	The first term is the Hamiltonian $H(i,U(t,\mu))$. Using 
	\eqref{eq:conclusion:lem:exp:first:order:variation} together with 
	the fact that $U$ and $\fD U$ are bounded and that 
	$y^l \leq N$, we can estimate the remainder by 
   	\begin{align*}
   	| r_3^{N,l} | &\leq  C 
   	\sum_{j\in\dd} \Bigl( 1 + 
    \frac{y^l}{N} + \bigl\vert \varrho^{N,l,l}(t, \bm{x},\bm{y})[j] \bigr\vert
	 \Bigr) 
\Bigl(  \frac{y^l}{N} +
	 \bigl\vert \varrho^{N,l,l}(t, \bm{x},\bm{y})[j] \bigr\vert \Bigr) 
	\\
   &\leq C \frac{y^l}{N} + C \Big( \frac{y^l}{N} \Big)^{1+\gamma'/2} \leq C \frac{y^l}{N},
   	\end{align*}
	which completes the proof. 
   \end{proof}

\section{Further prospects}\label{sec:7}
It is fair to say that, despite our long analysis, we have left open several quite important questions. This is mostly 
our deliberate choice since we want to keep the paper of a reasonable length. 

As we already alluded to in \S \ref{sec:convres}, a first point 
would be to compare the particle system driven by 
the equilibrium feedback law $\Delta w^{N,n}$
by the same particle system but driven by 
the feedback strategy $\Delta z^{N,n}$.

Another problem
would be to prove propagation of chaos. Notice indeed that our convergence result 
Theorem \ref{thm:convergence}
does not permit to prove asymptotic conditional independence of the particles. 
The usual argument to do so is based on a standard result of Sznitman, see 
\cite[Proposition 2.2]{sznit}, 
but it does not apply here because the weights 
are not equal to $1/N$ but to $(Y^l_{t}/N)_{l \in \NN}$.
In order to prove that, asymptotically, the particles just interact through the common noise, a possible approach 
would consist in showing that the distance between
$(\mu^N_{t})_{0 \le t \le T}$ 
and $({\mathbb E}[ \mu^N_{t} \vert {\mathcal N}^0])_{0 \le t \le T}$
tends to $0$. In order to do so, we may think of 
comparing $(\mu^N_{t})_{0 \le t \le T}$
and $(\widehat{\mu}^N_{t})_{0 \le t \le T}$, 
 the latter being obtained by considering the same system as \eqref{tilde_processes},
driven by the same common noise but by independent copies 
$\widehat{\mathcal N}^1,\cdots,
\widehat{\mathcal N}^N$
of the Poisson measures 
${\mathcal N}^1,\cdots,{\mathcal N}^N$. 
Alternatively, we might think of a more direct coupling argument; we refer for instance to 
\cite{kurtz-xiong-01} or to \cite[Chapter 2]{CarmonaDelarue_book_II} for such a proof of propagation of chaos
for diffusive particle systems subjected to a common noise. In this regard, we could use the 
representation of the MFG equilibrium 
in the form $P_t^i = \E[ Y_t \one_{\{X_t=i\}} | \mathcal{F}_t^W]$, where $X_{t}$ is the 
state of a tagged player within the population and $W$ is the common noise in 
\eqref{eq:weak:sde:final:2:b}; we refer to 
\cite{BCCD-arxiv} for more details on this representation.
Still, this would obviously raise subtle questions about a suitable coupling between the Brownian noise $W$
and the Poisson measures ${\mathcal N}^0$.
%

Last, we may think of estimating the weak error in the convergence in law 
proved in 
Theorem 
\ref{thm:convergence}, at least for the one-dimensional 
marginals. 
It seems that, in order to do so, we can adapt 
Proposition 
\ref{prop:u}, choosing instead 
of $U$ itself, the solution $Z$ of the equation 
\begin{align}
&\partial_t Z    +
\sum_{j,k \in \ES} p_k\bigl[ \varphi(p_{j}) + (U^k-U^j)_+ \bigr] \left( \partial_{p_{j}} Z - \partial_{p_{k}} Z\right) 
  +\tfrac{\varepsilon^2}{2} \sum_{j,k \in \ES}(p_j \delta_{jk}-p_{j} p_{k}) \partial^2_{p_{j} p_{k}} Z =0,
  \nonumber
  \\
&Z(T,p)= h(p),
\label{eq:PDE:Z}
\end{align}
for $p \in \cP(\ES)$ and for a terminal boundary condition $h$. 
Provided $h$ is smooth enough, 
we know from \cite[Theorem 10.0.2]{EpsteinMazzeo}
that the above equation has a classical solution. 
Then, following the proof of 
Proposition 
\ref{prop:u}, we can show that 
the function $(t,\bm{x},\bm{y}) \mapsto Z(t,\mu^N_{\bm {x},\bm{y}})$
is nearly harmonic for the generator of $(\bm{X}^N,\bm{Y}^N)$ (which we merely denoted 
by $(\bm{X},\bm{Y})$ 
in 
\eqref{tilde_processes}
with $\iota=0$ therein). 
The fact that the latter mapping is nearly harmonic 
 makes it possible to 
prove (by It\^o's expansion) that $Z(0, \mu^N_{\bm {x},\bm{y}})$
and ${\mathbb E}[h(\mu^N_{\bm{X}_{T}^N,\bm{Y}_{T}^N})]$ 
get closer as $N$ tends to $\infty$, 
whenever $(\bm{X}^N,\bm{Y}^N)$ starts from $(\bm{x},\bm{y})$ 
at time $0$.  
This strategy has been reported 
within the more standard context of McKean--Vlasov equations in 
\cite[Subsection 5.7.4]{CarmonaDelarue_book_I}
and, in a more systematic manner, in 
the recent contribution 
\cite{chassagneux2019weak}.

Additionally, one can also construct an asymptotic equilibrium in the $N$-player game using the master equation; 
we refer for instance to 
\cite[Subsection 6.1.2]{CarmonaDelarue_book_II}
for more on this approach in the more standard diffusive setting.  
Specifically, under the conditions in Theorems \ref{thm:convergence_value} and \ref{thm:convergence}, for any sequence of initial conditions $\{(\blx^N,\bly^N)=((x^{N,l})_{l \in \NN},(1,\dots,1))\}_N$ satisfying \eqref{conv:init}, 
 the feedback strategy vector  
$\bm{\hat\alpha}^*=(\hat\alpha^{*,1},\dots,\hat\alpha^{*,N})$ given by
\be 
\begin{split}
\hat\alpha^{*,l}(t,\bm{x},\bm{y})[j]=a^*\bigg(x^l, \u^{N,l}_\bullet(t,\bm{x},\bm{y})\bigg)[j] :&= \Bigl(\u^{N,l}\big(t,\bm{x},\bm{y})- \u^{N,l}\bigl(t,(j,\bm{x}^{-l}),\bm{y}\bigr)\Big)_+
\\
&=\Bigl(- \Delta^l \u^{N,l}\big(t,\bm{x},\bm{y})[j]\Big)_+,
\end{split}
\ee 
for $l \in \EN$ such that $j \not = x^l$,
is expected to define an asymptotic Nash equilibrium in Markov feedback form for the $N$-player game. 
That means that  for any player $l\in\NN$ and any sequence of feedback controls $\{\beta^{N,l}\}_{N\ge l}$, one should have
\begin{align}\notag
\liminf_{N\to\infty}J^{N,l}\bigl(t,\blx,\bly,\beta^l,\hat{\bm \alpha}^{*,-l}\bigr)\ge \limsup_{N \to\infty}J^{N,l}\bigl(t,\blx,\bly,\hat{\bm \alpha}^{*}\bigr),
\end{align}
where $J^{N,l}(t,\bm{x},\bm{y},\bm{\alpha})$ denotes the cost when the process $(\bm{X}^{N},\bm{Y}^{N})$ starts at time $t$ with $(\bm{X}_t^{N},\bm{Y}_t^{N})= (\bm{x}^N,\bm{y}^N)$ (we here put a superscript $N$ in order to emphasize the dependence on $N$). Moreover, letting $\hat\mu^N=(\hat\mu^N_t)_{0\le t\le T}$ be empirical distribution under $\bm{\hat\alpha}^*$, $(\hat\mu^N_{t})_{0 \le t \le T}$  {is expected to} converge in the weak sense on ${\mathcal D}([0,T];\R^d)$, equipped with 
the $J1$ Skorokhod topology, to the solution $(P_{t})_{0 \le t \le T}$ of 
the SDE \eqref{eq:SDE:p}. Notice that, differently from the analysis carried out in the rest of the paper, this would force us to address the asymptotic behavior of 
the (conditional) mass of  a ``deviating player'' $(\widetilde X^{N,l},\widetilde Y^{N,l})$. The latter would read 
$(Q_{t}^{N,l}[i]:=\E[ \widetilde Y_t^{N,l} \one_{\{\widetilde X_t^{N,l}=i\}} | {\mathcal N}^0])_{0 \le t \le T}$; a peculiarity of it is that, similar to 
the process $(Q_{t})_{0 \le t \le T}$ 
in 
\eqref{eq:weak:sde:final:2:b}, it might not take values in the simplex. 

%
%
%
%
%
%
\appendix
\section{Nash equilibria of the $N$-player game}
\label{subse:heuristic:master equation}

\subsection{Proof of Proposition \ref{prop:verification:Nash}}
\label{subse:proof:verif}

{ \ }
\vskip 5pt

 \textit{Value of $v^{N,l}$ when $y^l=0$.}
We first prove that $v^{N,l}(t,\blx,\bly)=0$ when  $y^l=0$. To do so, 
given 
$\bm{\alpha}^*$ as in 
\eqref{optcon} and \eqref{optcon:2}, we call 
$(\bm{X},\bm{Y})$  the dynamics related to the strategy vector $(0,{\boldsymbol \alpha}^{*,-l})$
 and starting at time $t$ from $(\bm{X}_t,\bm{Y}_t)= (\bm{x},\bm{y})$. By a direct application of Dynkin's formula
to $(v^{N,l}(s,\bm{X}_{s},\bm{Y}_{s}))_{t \leq s \leq T}$, 
using in addition the fact that 
the process $Y^l$ remains equal to $0$ whenever starting from $0$,
we get that 
$v^{N,l}(t,\blx,\bly)=0$ when  $y^l=0$.
\vskip 5pt

\textit{Verification argument.}
Let $\beta(t,\bm{x},\bm{y})$ be another feedback control and then $(\bm{X},\bm{Y})$ now denote the dynamics related to the strategy vector $(\beta,\alpha^{*,-l})$ (still starting at time $t$ from $(\bm{X}_t,\bm{Y}_t)= (\bm{x},\bm{y})$). Fix $l\in\NN$. By  definition of the Hamiltonian and from the fact that $v^{N,l}(s,\bm{x},\bm{y})$ is zero if $y^l=0$, the Nash system gives 
\begin{align*}
&\frac{d}{dt}v^{N,l}(s,\blx,\bly)
+
\sum_{m \in \EN } 
\varphi(\mu^N_{\bm{x},\bm{y}}[\bullet])
\cdot \Delta^m v^{N,l}[\bullet] 
\\
&\qquad
+\sum_{m\neq l} \alpha^{*,m}\big(s,\bm{x}, \bm{y})\cdot \Delta^m v^{N,l}[\bullet] 
 +  \beta(s,\bm{x},\bm{y})\cdot
\Delta^lv^{N,l}
%
\\
& \qquad 
+\varepsilon N \EE 
\biggl[ v^{N,l}\biggl(s,\bm{x}, y^1\frac{S_{\mu^N_{\bm{x},\bm{y}}}[x^1]}{N\mu^N_{\bm{x},\bm{y}}[x^1]},\dots, y^N\frac{S_{\mu^N_{\bm{x},\bm{y}}}[x^N]}{N\mu^N_{\bm{x},\bm{y}}[x^N]}\biggr)
 -v^{N,l}(s,\bm{x},\bm{y})
\biggr]
\\
&=  \frac12 \sum_{j \not = x^l} y^l \bigl\vert \alpha^{*,l}(s,\blx,\bly)[j] \bigr\vert^2 
+ 
\one_{\{y^l \not =0\}} \beta(s,\bm{x},\bm{y}) 
\cdot
\Delta^lv^{N,l}[\bullet] - y^lf(x_l, \mu^N_{\bm{x},\bm{y}})
\\
&\geq 
\frac12 \sum_{j \not = x^l} y^l \bigl\vert \alpha^{*,l}(s,\blx,\bly)[j] \bigr\vert^2 
-
\one_{\{y^l \not =0\}} \beta(s,\bm{x},\bm{y}) 
\cdot
\bigl( - \Delta^lv^{N,l} \bigr)_{+}[\bullet] - y^lf(x_l, \mu^N_{\bm{x},\bm{y}})
\\
&= \frac12 \sum_{j \not = x^l} y^l \bigl\vert \alpha^{*,l}(s,\blx,\bly)[j] \bigr\vert^2 
-
y^l \beta(s,\bm{x},\bm{y}) 
\cdot
\alpha^{*,l} (s,\bm{x},\bm{y}) 
 - y^lf(x_l, \mu^N_{\bm{x},\bm{y}})
 \\
 &=
  \frac12 \sum_{j \not = x^l} y^l \bigl\vert \alpha^{*,l}(s,\blx,\bly)[j] - \beta(s,\blx,\bly)[j] \bigr\vert^2 
- \frac12 \sum_{j \not = x^l} y^l \bigl\vert \beta(s,\blx,\bly)[j] \bigr\vert^2 
 - y^lf(x_l, \mu^N_{\bm{x},\bm{y}}), 
\end{align*}
for any $s,\bm{x},\bm{y}$. Applying Dynkin's formula, and the above inequality, we obtain
\begin{align*}
&v^{N,l}(t,\bm{x},\bm{y})= \E \bigg[ v^{N,l}(T,\bm{X}_T,\bm{Y}_T) - \int_t^T \bigg(
 \frac{d}{dt}v^{N,l} (s,\bm{X}_s,\bm{Y}_s)\\
 &\quad+\sum_{m\neq l} \alpha^{*,m}\Big(s,{\boldsymbol X}_s,\bm{Y}_s\Big)\cdot \Delta^m v^{N,l}(s,\bm{X}_s,\bm{Y}_s)+ \beta(s,\bm{X}_s,\bm{Y}_s)\cdot
\Delta^lv^{N,l}(s,\bm{X}_s,\bm{Y}_s)
 \\
&\quad  +\varepsilon N \EE 
\biggl[ v^{N,l}\biggl(s,\bm{X}_s, Y_s^1\frac{S_{\mu^N_s}[X^1_s]}{N\mu^N_{s}[X_s^1]},\dots, 
 Y_s^N\frac{S_{\mu^N_s}[X^N_s]}{N\mu^N_{s}[X_s^N]}\biggr)
 -v^{N,l}(s,\bm{X}_s,\bm{Y}_s)
\biggr]\bigg) ds \bigg]
\\
&\leq \E \biggl[ Y^l_T g(X^l_T, \mu^N_T) + \int_t^T Y_s^l \Big(
\ell(X_s^l,\beta(s,\bm{X}_s,\bm{Y}_s)) 
+ f(X_s^l, \mu^N_{s})\Big) ds \biggr]
\\
&\hspace{15pt}- \frac12 \E \biggl[  \int_t^T Y_s^l
\sum_{j \not = X^l_{s}}
\bigl\vert \alpha^{*,l}(s,\bm{X}_{s},\bm{Y}_{s})[j] - \beta(s,\bm{X}_{s},\bm{Y}_{s})[j] \bigr\vert^2 
 ds \biggr], 
\end{align*}
which shows that 
$v^{N,l}(t,\bm{x},\bm{y})
\leq J^l(t,\bm{x},\bm{y},\beta,\bm{\alpha}^{*,-l})$ (the latter being defined as the cost to $l$ 
when the system is driven by $[\beta,\bm{\alpha}^{*,-l}]$ and starts from $(\blx,\bly)$ 
at time $t$). Replacing $\beta$ by $\alpha^{*,l}$, we obtain 
$v^{N,l}(t,\bm{x},\bm{y})
=J^l(t,\bm{x},\bm{y},\bm{\alpha})$.
\qed

\subsection{Proof of insensitivity and of uniqueness of the Nash equilibrium}
We now prove 
the claims of 
Proposition 
\ref{prop:existence:!:continuous:eq}
that we left aside in Section \ref{sec:3}. It remains to prove the fact that the equilibrium is uniquely determined (in the sense 
of \eqref{eq:uniqueness:nash}) and satisfies the insensitivity property 
\eqref{eq:insensitivity}. In fact, we prove both at the same time. 

In order to proceed, we assume that we are given a strategy, say $\widehat{\bm{\alpha}}$, that defines a 
bounded equilibrium in Markov feedback form. 
\vskip 4pt

\textit{First step.} Fix a player $l$ and then identify $\widehat \alpha^{l}$ with the best response 
of the cost functional \eqref{cost} when all the feedback functions, except the $l$th one, are fixed. 
In order to do so, we may first solve the equation (which is directly inspired from
\eqref{newnash})
\begin{align}
&\frac{d}{dt}\widehat w^{l}
+\sum_{m \in \EN } 
\varphi(\mu^N_{\bm{x},\bm{y}}[\bullet])
\cdot \Delta^m\widehat w^{l}[\bullet] 
+\sum_{m\neq l} \widehat \alpha^{m}\big(t,\blx,\bly\big)\cdot \Delta^m \widehat w^{l} 
\nonumber
\\
& + H\Big(x^l, \widehat w^{l}_\bullet \Big) + f\bigl(x^l, \mu^N_{\bm{x},\bm{y}}\bigr)
\nonumber
\\
& +\varepsilon N \EE 
\biggl[ \frac{S_{\mu^N_{\bm{x},\bm{y}}}[x^l]}{N\mu^N_{\bm{x},\bm{y}}[x^l]}
\biggl(
\widehat w^{l}\biggl(t,\bm{x}, y^1\frac{S_{\mu^N_{\bm{x},\bm{y}}}[x^1]}{N\mu^N_{\bm{x},\bm{y}}[x^1]},\dots, y^N\frac{S_{\mu^N_{\bm{x},\bm{y}}}[x^N]}{N\mu^N_{\bm{x},\bm{y}}[x^N]}\biggr)
 - \widehat w^{l}\bigl(t,\bm{x},\bm{y}\bigr)
\biggr)\biggr]=0, \nonumber
\end{align}
with the terminal boundary condition
\be
\widehat w^{l}(T,\bm{x},\bm{y})= g(x^l, \mu^N_{\bm{x},\bm{y}}).
\ee
Notice that we solve one equation only (and not a system). 
We then let 
$\widehat v^{l}(t,\blx,\bly):=y^l \widehat w^{l}(t,\blx,\bly)$. 

Very much in the spirit of the proof of Proposition 
\ref{prop:verification:Nash} in Subsection \ref{subse:proof:verif}
(with $(\bm{X},\bm{Y})$ denoting the same dynamics, 
for some strategy vector $(\beta,\widehat \alpha^{-l})$
except that $\widehat{\bm{\alpha}}$ is now given as some equilibrium), we get 
that 
\begin{equation*}
\begin{split}
&\widehat v^{l}(t,\bm{x},\bm{y})
+
\frac12 \E \bigg[\int_t^T Y_s^l \sum_{j \not= X_{s}^l} 
\bigl\vert a^*\bigl({X}_{s}^l,\widehat w^{l}_{\bullet}(s,\bm{X}_{s},\bm{Y}_{s})\bigr)[j] - \beta(s,\bm{X}_{s},\bm{Y}_{s})[j]
\bigr\vert^2ds 
 \biggr]
\leq J^l\bigl(t,\bm{x},\bm{y},\beta,\widehat{\bm{\alpha}}^{-l}\bigr),
\end{split}
\end{equation*}
from which we deduce that $\widehat{\alpha}^{l}(t,\blx,\bly) = a^*(x^l,\widehat{w}_{\bullet}^{l}(t,\blx,\bly))$ whenever $y^l >0$. 
This prompts us to introduce the notation $\Y^l = \{ \bly \in \Y : y^l >0\}$. 

Following the proof of Proposition 
\ref{prop:nash:system:posedness}, we notice that 
\begin{equation*}
\sup_{t \in [0,T]} \max_{\blx \in \ES^N} \sup_{\bly \in \Y}
\vert \widehat{w}^{l}(t,\blx,\bly) \vert < \infty. 
\end{equation*}

\textit{Second step.}
Inspired by the proof of Proposition 
\ref{prop:nash:system:posedness}, we are to regard 
$(\widehat w^{l})_{l \in \NN}$ as the fixed point of some mapping $\widehat \Phi$. 
However, for reasons that will be made clear below, 
we construct the mapping $\widehat \Phi$ in a slightly different manner 
than the function $\Phi$ in 
the proof
Proposition 
\ref{prop:nash:system:posedness}. In order to proceed, 
we introduce a smooth bounded cut-off function $\vartheta$ from ${\mathbb R}$ into itself
with the property that 
\begin{equation*}
\vartheta \bigl( 
\widehat w^{l}(t,\blx,\bly) 
\bigr) =
\widehat w^{l}(t,\blx,\bly), \quad t \in [0,T], \ \blx \in \ES^N, \ \bly \in \Y. 
\end{equation*}
Noticing from the first step that 
$\widehat \alpha^{l}(t,\blx,\bly) 
\cdot \Delta^m \widehat w^{l}$ always writes in the form
\begin{equation*}
\widehat \alpha^{l}(t,\blx,\bly) 
\cdot \Delta^m \widehat w^{l}[\bullet] 
= 
\Bigl(\widehat \alpha^{l}(t,\blx,\bly) 
{\mathbf 1}_{\{y^l=0\}}
+
a^*\bigl(x^l,\widehat w^{l}(t,\blx,\bly)\bigr)
{\mathbf 1}_{\{y^l>0\}}
\Bigr)
\cdot \Delta^m \widehat w^{l}[\bullet],
\end{equation*}
we hence define $\widehat \Phi$ as the mapping 
 that sends an 
input $(w^{l})_{l \in \NN}$ onto the solution 
$(\widetilde w^l)_{l \in \NN}$ 
of the system
\begin{align}
&\frac{d}{dt} \widetilde w^{l}
+\sum_{m \in \EN } 
\varphi(\mu^N_{\bm{x},\bm{y}}[\bullet])
\cdot \Delta^m \widetilde w^{l}[\bullet]
 + H\Big(x^l, \vartheta\bigl(\widetilde w^{l}\bigr)_\bullet \Big) + f\bigl(x^l, \mu^N_{\bm{x},\bm{y}}\bigr) 
\nonumber
\\
&\qquad +\sum_{m\neq l}
\biggl(\widehat \alpha^{m}(t,\blx,\bly) 
\one_{\{y^m=0\}}
+
a^*\Bigl(x^m,\vartheta(w^{m})_{\bullet}\Bigr)
\one_{\{y^m>0\}}
\biggr)
\cdot \Delta^m \vartheta( w^{l} )[\bullet]
\label{newnash2}
\\
&\qquad +\varepsilon N \EE 
\biggl[ \frac{S_{\mu^N_{\bm{x},\bm{y}}}[x^l]}{N\mu^N_{\bm{x},\bm{y}}[x^l]}
\biggl(
w^{l}\biggl(t,\bm{x}, y^1\frac{S_{\mu^N_{\bm{x},\bm{y}}}[x^1]}{N\mu^N_{\bm{x},\bm{y}}[x^1]},\dots, y^N\frac{S_{\mu^N_{\bm{x},\bm{y}}}[x^N]}{N\mu^N_{\bm{x},\bm{y}}[x^N]}\biggr)
 -w^{l}\bigl(t,\bm{x},\bm{y}\bigr)
\biggr)\biggr]=0, \nonumber
\end{align}
with the obvious notation 
that $\vartheta(w^{l})_{\bullet}$
denotes 
the vector in ${\mathbb R}^d$
given by 
$(\vartheta(w^{l} )(t, \bm{x},\bm{y})[j]= \vartheta(w^{l} )(t, (j, \bm{x}^{-l}),\bm{y}))_{j \in \ES}$.
In other words, $(w^{l})_{l \in \NN}$ can be approximated by 
iterating the function $\widehat \Phi$. 

We then consider an input 
$(w^{l})_{l \in \NN}$ with the property that, for any 
$l,n \in \NN$, any $\bly \in \Y$ such that $y^l >0$ and 
$y^n=0$ and any $(t,\blx) \in [0,T] \times \ES^N$,
\begin{equation}
\label{eq:hyp:induction}
\Delta^n w^l(t,\bm{x},\bm{y}) [j] = 0, \quad j \in \ES, 
\end{equation}
or, equivalently,
for any $\blx \in \ES^N$, 
\begin{equation*}
w^l\bigl(t,\blx,\bm{y}\bigr) 
=
w^l\bigl(t,(j,\blx^{-n}),\bm{y}\bigr), \quad j \in \ES. 
\end{equation*}
We hence check what the last three  terms in 
\eqref{newnash2} become, whenever computed at
such a  point $\bly \in \Y$ with $y^l>0$ and $y^n=0$ form some $l,n \in \NN$.  
The first point is to notice that, for any 
$\blx,\blx' \in \ES^N$ with 
$\blx^{-n}=\blx^{\prime,-n}$,
\begin{equation*}
\mu^N_{\blx,\bly}
= \frac1N \sum_{m=1}^N y^m \delta_{x^m}
= \frac1N \sum_{m \not = n} y^m \delta_{x^m}
= \mu^N_{\blx',\bly},
\end{equation*}
which shows that 
\begin{equation}
\label{eq:x:xprime:1}
f \bigl( x^l , \mu^N_{\blx,\bly} \bigr) 
= 
f \bigl( x^l , \mu^N_{\blx',\bly} \bigr). 
\end{equation}
As for the term on the second line of 
\eqref{newnash2}, we have that 
$\Delta^m \vartheta(w^l(t,\blx,\bly))[\bullet]=0$
for any $m$ such that $y^m=0$, 
 and then  
\begin{equation*}
\begin{split}
&\sum_{m\neq l}
\biggl(\widehat \alpha^{m}(t,\blx,\bly) 
\one_{\{y^m=0\}}
+
a^*\Bigl(x^m,\vartheta\bigl(w^{m}(t,\blx,\bly)\bigr)_{\bullet}\Bigr)
\one_{\{y^m>0\}}
\biggr)
\cdot \Delta^m \vartheta\bigl( w^{l}(t,\blx,\bly)\bigr)[\bullet]
\\
&= \sum_{m\neq l}
\biggl(
a^*\Bigl(x^m,\vartheta\bigl(w^{m}(t,\blx,\bly)\bigr)_{\bullet}\Bigr)
\one_{\{y^m>0\}}
\biggr)
\cdot \Delta^m \vartheta\bigl( w^{l}(t,\blx,\bly)\bigr)[\bullet].
\end{split}
\end{equation*}
By assumption, we have that 
$\vartheta\bigl(w^{m}(t,\blx,\bly)\bigr)_{\bullet}
=\vartheta\bigl(w^{m}(t,\blx',\bly)\bigr)_{\bullet}$
whenever $y^m>0$ (it suffices to replace $l$ by $m$ in 
\eqref{eq:hyp:induction}). 
By the same argument, 
$\Delta^m \vartheta\bigl( w^{l}(t,\blx,\bly)\bigr)[\bullet]
= 
\Delta^m \vartheta\bigl( w^{l}(t,\blx',\bly)\bigr)[\bullet]$. 
Therefore, 
\begin{equation}
\label{eq:x:xprime:2}
\begin{split}
&\sum_{m\neq l}
\biggl(\widehat \alpha^{m}(t,\blx,\bly) 
\one_{\{y^m=0\}}
+
a^*\Bigl(x^m,\vartheta\bigl(w^{m}(t,\blx,\bly)\bigr)_{\bullet}\Bigr)
\one_{\{y^m>0\}}
\biggr)
\cdot \Delta^m \vartheta\bigl( w^{l}(t,\blx,\bly)\bigr)[\bullet]
\\
&= \sum_{m\neq l}
\biggl(\widehat \alpha^{m}(t,\blx',\bly) 
\one_{\{y^m=0\}}
+
a^*\Bigl(x^m,\vartheta\bigl(w^{m}(t,\blx',\bly)\bigr)_{\bullet}\Bigr)
\one_{\{y^m>0\}}
\biggr)
\cdot \Delta^m \vartheta\bigl( w^{l}(t,\blx',\bly)\bigr)[\bullet].
\end{split}
\end{equation}
We then proceed in a similar manner with the last term in 
\eqref{newnash2}. 
Importantly, we recall that, in the expectation therein, the ratio 
$S_{\mu^N_{\bm{x},\bm{y}}}[x^l]/(N\mu^N_{\bm{x},\bm{y}}[x^l])$
is understood as $1$ when 
$\mu^N_{\bm{x},\bm{y}}[x^l]=0$. In particular, 
as long as $y^l$ itself cannot be zero, we always have that 
$y^l S_{\mu^N_{\bm{x},\bm{y}}}[x^l]/(N\mu^N_{\bm{x},\bm{y}}[x^l]) >0$. 
Hence, 
\begin{equation*}
w^{l}\biggl(t,\bm{x}, y^1\frac{S_{\mu^N_{\bm{x},\bm{y}}}[x^1]}{N\mu^N_{\bm{x},\bm{y}}[x^1]},\dots, y^N\frac{S_{\mu^N_{\bm{x},\bm{y}}}[x^N]}{N\mu^N_{\bm{x},\bm{y}}[x^N]}\biggr)
=
w^{l}\biggl(t,\bm{x}', y^1\frac{S_{\mu^N_{\bm{x}',\bm{y}}}[x^{\prime,1}]}{N\mu^N_{\bm{x}',\bm{y}}[x^{\prime,1}]},\dots, y^N\frac{S_{\mu^N_{\bm{x}',\bm{y}}}[x^N]}{N\mu^N_{\bm{x}',\bm{y}}[x^{\prime,N}]}\biggr),
\end{equation*}
from which we deduce that 
\begin{equation}
\label{eq:x:xprime:3}
\begin{split}
&\varepsilon N \EE 
\biggl[ \frac{S_{\mu^N_{\bm{x},\bm{y}}}[x^l]}{N\mu^N_{\bm{x},\bm{y}}[x^l]}
\biggl(
w^{l}\biggl(t,\bm{x}, y^1\frac{S_{\mu^N_{\bm{x},\bm{y}}}[x^1]}{N\mu^N_{\bm{x},\bm{y}}[x^1]},\dots, y^N\frac{S_{\mu^N_{\bm{x},\bm{y}}}[x^N]}{N\mu^N_{\bm{x},\bm{y}}[x^N]}\biggr)
 -w^{l}\bigl(t,\bm{x},\bm{y}\bigr)
\biggr)\biggr]
\\
&= 
\varepsilon N \EE 
\biggl[ \frac{S_{\mu^N_{\bm{x'},\bm{y}}}[x^{\prime,l}]}{N\mu^N_{\bm{x'},\bm{y}}[x^{\prime,l}]}
\biggl(
w^{l}\biggl(t,\bm{x}', y^1\frac{S_{\mu^N_{\bm{x}',\bm{y}}}[x^{\prime,1}]}{N\mu^N_{\bm{x}',\bm{y}}[x^{\prime,1}]},\dots, y^N\frac{S_{\mu^N_{\bm{x}',\bm{y}}}[x^N]}{N\mu^N_{\bm{x}',\bm{y}}[x^{\prime,N}]}\biggr)
 -w^{l}\bigl(t,\bm{x}',\bm{y}\bigr)
\biggr)\biggr].
\end{split}
\end{equation}
Collecting 
\eqref{eq:x:xprime:1},
\eqref{eq:x:xprime:2}
and
\eqref{eq:x:xprime:3},
 we deduce that, 
for any $l,n \in \NN$ and $\bly \in \Y$ with 
$y^l>0$ and $y^n=0$, 
the last three terms in 
\eqref{newnash2}
are the same when evaluated at $(t,\blx,\bly)$ 
and $(t,\blx',\bly)$ with $\blx^{-n}=\blx^{\prime,-n}$. 
We finally observe that, for $l,n \in \NN$, 
and for
$(\blx,\bly) \in \ES^N \times \Y$ fixed with $y^l>0$ and $y^n=0$,
\eqref{newnash2}
may be regarded as an ordinary differential equation with 
$\widetilde w^l(\cdot,\blx,\bly)$ as unique solution. 
And thus
$\widetilde w^l(\cdot,\blx,\bly)
= 
\widetilde w^l(\cdot,\blx',\bly)$. Put it differently, 
$\widetilde w^l$ satisfies 
\eqref{eq:hyp:induction}, which shows that 
\eqref{eq:hyp:induction}
is stable by $\widehat \Phi$. Therefore, 
the fixed point 
$\widehat w$ of $\widehat \Phi$ also satisfies 
\eqref{eq:hyp:induction} (for any $l \in \NN$). 
\vskip 5pt

\textit{Third step.}
We eventually identify 
$(\widehat w^l)_{l \in \NN}$
with 
the solution $(w^{N,l})_{l \in \NN}$ 
given by Proposition 
\ref{prop:nash:system:posedness}. 
The proof mostly follows from 
the argument developed in the previous step. Indeed, we now know that both solutions 
can be approximated by 
iterating the mapping $\widehat{\Phi}$
associated with 
\eqref{newnash2}, with the special feature that 
all the inputs therein are required to satisfy 
\eqref{eq:hyp:induction}. 

The analysis achieved in the second step says that, in order to 
compute $\widetilde w^l$ at tuples $(t,\blx,\bly) \in [0,T] \times \ES^N \times \Y^l$, 
the precise values of $\widehat{\alpha}^n(t,\blx,\bly)$ at indices $n$ for which 
$y^n=0$ do not matter and, moreover, 
only the knowledge of each input $w^m$
at points $(t,\blx,\bly) \in [0,T] \times \ES^N \times \Y^m$, for $m \in \NN$, does matter. 
As a result, for any $l \in \NN$, 
the approximating sequences (as given by the iteration of the mapping 
$\widehat{\Phi}$ defined through the system 
\eqref{newnash2}) 
of the two solutions 
$\widehat w^l$
and $w^{N,l}$ 
coincide at any point $(t,\blx,\bly) 
\in [0,T] \times \ES^N \times \Y^l$. 
Therefore,
\begin{equation*}
\widehat w^l(t,\blx,\bly) 
=
w^{N,l}(t,\blx,\bly), 
\quad 
(t,\blx,\bly) 
\in [0,T] \times \ES^N \times \Y^l,
\end{equation*} 
which completes the proof. 
\section*{Acknowledgement}

We thank the anonymous AE and referee for their suggestions, which helped us to improve our paper.

\bibliographystyle{abbrv}
\bibliography{references}

\begin{thebibliography}{10}

\bibitem{BCCD-arxiv}
E.~Bayraktar, A.~Cecchin, A.~Cohen, and F.~Delarue.
\newblock Finite state mean field games with {W}right-{F}isher common noise.
\newblock {\em J. Math. Pures Appl. (9)}, 147:98--162, 2021.

\bibitem{bay-coh2019}
E.~Bayraktar and A.~Cohen.
\newblock Analysis of a finite state many player game using its master
  equation.
\newblock {\em SIAM Journal on Control and Optimization}, 56(5):3538--3568,
  2018.

\bibitem{bayzhang2019}
E.~Bayraktar and X.~Zhang.
\newblock On non-uniqueness in mean field games.
\newblock {\em Proc. Amer. Math. Soc.}, 148(9):4091--4106, 2020.

\bibitem{belak2019continuous}
C.~Belak, D.~Hoffmann, and F.~T. Seifried.
\newblock Continuous-time mean field games with finite state space and common
  noise.
\newblock {\em To appear in Applied Mathematics and Optimization}, 2021.

\bibitem{bertucciJMPA}
C.~Bertucci.
\newblock Optimal stopping in mean field games, an obstacle problem approach.
\newblock {\em J. Math. Pures Appl.}, 120:165--194, 2018.

\bibitem{BertucciLasryLions}
C.~Bertucci, J.-M. Lasry, and P.-L. Lions.
\newblock Some remarks on mean field games.
\newblock {\em Communications in Partial Differential Equations},
  44(3):205--227, 2019.

\bibitem{cam-fis2018}
L.~Campi and M.~Fischer.
\newblock {$N$}-player games and mean-field games with absorption.
\newblock {\em Ann. Appl. Probab.}, 28(4):2188--2242, 2018.

\bibitem{CardaliaguetDelarueLasryLions}
P.~Cardaliaguet, F.~Delarue, J.-M. Lasry, and P.-L. Lions.
\newblock {\em The master equation and the convergence problem in mean field
  games}, volume 201 of {\em Annals of Mathematics Studies}.
\newblock Princeton University Press, Princeton, NJ, 2019.

\bibitem{CarmonaDelarue_book_I}
R.~Carmona and F.~Delarue.
\newblock {\em Probabilistic theory of mean field games with applications.
  {I}}, volume~83 of {\em Probability Theory and Stochastic Modelling}.
\newblock Springer, Cham, 2018.
\newblock Mean field FBSDEs, control, and games.

\bibitem{CarmonaDelarue_book_II}
R.~Carmona and F.~Delarue.
\newblock {\em Probabilistic theory of mean field games with applications.
  {II}}, volume~84 of {\em Probability Theory and Stochastic Modelling}.
\newblock Springer, Cham, 2018.
\newblock Mean field games with common noise and master equations.

\bibitem{cecdaifispel}
A.~Cecchin, P.~Dai~Pra, M.~Fischer, and G.~Pelino.
\newblock On the convergence problem in mean field games: a two state model
  without uniqueness.
\newblock {\em SIAM J. Control Optim.}, 57(4):2443--2466, 2019.

\bibitem{cec-del2020}
A.~Cecchin and F.~{Delarue}.
\newblock {Selection by vanishing common noise for potential finite state mean
  field games}.
\newblock {\em To appear in Communications in PDE}, 2021.

\bibitem{Cecchin2017}
A.~Cecchin and M.~Fischer.
\newblock Probabilistic approach to finite state mean field games.
\newblock {\em Applied Mathematics {\&} Optimization}, 81(2):253--300, 2020.

\bibitem{cec-pel2019}
A.~Cecchin and G.~Pelino.
\newblock Convergence, fluctuations and large deviations for finite state mean
  field games via the master equation.
\newblock {\em Stochastic Processes and their Applications}, 129(11):4510 --
  4555, 2019.

\bibitem{chassagneux2019weak}
J.-F. Chassagneux, L.~Szpruch, and A.~Tse.
\newblock Weak quantitative propagation of chaos via differential calculus on
  the space of measures.
\newblock {\em arXiv preprint arXiv:1901.02556}, 2019.

\bibitem{cla-ren-tan2019}
J.~{Claisse}, Z.~{Ren}, and X.~{Tan}.
\newblock {Mean Field Games with Branching}.
\newblock {\em arXiv e-prints}, page arXiv:1912.11893, Dec. 2019.

\bibitem{delarue-cetraro}
F.~Delarue.
\newblock Master equation for finite state mean field games with additive
  common noise.
\newblock In {\em Mean Field Games, Cetraro, Italy 2019, Cardaliaguet, Pierre,
  Porretta, Alessio (Eds.)}, LNM 2281, pages 203--248. Springer, 2021.

\bibitem{del-lac-ram2019}
F.~Delarue, D.~Lacker, and K.~Ramanan.
\newblock From the master equation to mean field game limit theory: a central
  limit theorem.
\newblock {\em Electron. J. Probab.}, 24:Paper No. 51, 54, 2019.

\bibitem{DelarueLackerRamananLDP}
F.~Delarue, D.~Lacker, and K.~Ramanan.
\newblock From the master equation to mean field game limit theory: large
  deviations and concentration of measure.
\newblock {\em Ann. Probab.}, 48(1):211--263, 2020.

\bibitem{EpsteinMazzeo}
C.~L. Epstein and R.~Mazzeo.
\newblock {\em Degenerate diffusion operators arising in population biology},
  volume 185 of {\em Annals of Mathematics Studies}.
\newblock Princeton University Press, Princeton, NJ, 2013.

\bibitem{EthierKurtz}
S.~N. Ethier and T.~G. Kurtz.
\newblock {\em Markov processes}.
\newblock Wiley Series in Probability and Mathematical Statistics: Probability
  and Mathematical Statistics. John Wiley \& Sons, Inc., New York, 1986.
\newblock Characterization and convergence.

\bibitem{Fischer2017}
M.~Fischer.
\newblock On the connection between symmetric $ n $-player games and mean field
  games.
\newblock {\em Ann. Appl. Probab.}, 127(2):757--810, 2017.

\bibitem{fis1999}
R.~A. Fisher.
\newblock {\em The genetical theory of natural selection}.
\newblock Oxford University Press, Oxford, variorum edition, 1999.
\newblock Revised reprint of the 1930 original, Edited, with a foreword and
  notes, by J. H. Bennett.

\bibitem{Huang2007}
M.~Huang, P.~E. Caines, and R.~P. Malham{\'e}.
\newblock The {N}ash certainty equivalence principle and {M}ckean-{V}lasov
  systems: {A}n invariance principle and entry adaptation.
\newblock In {\em Decision and Control, 2007 46th IEEE Conference on}, pages
  121--126. IEEE, 2007.

\bibitem{Huang2006}
M.~Huang, R.~P. Malham{\'e}, and P.~E. Caines.
\newblock Large population stochastic dynamic games: Closed-loop
  {M}c{K}ean-{V}lasov systems and the {N}ash certainty equivalence principle.
\newblock {\em Commun. Inf. Syst.}, 6(3):221--251, 2006.

\bibitem{kurtz-xiong-01}
T.~G. Kurtz and J.~Xiong.
\newblock Numerical solutions for a class of {SPDE}s with application to
  filtering.
\newblock In {\em Stochastics in finite and infinite dimensions}, Trends Math.,
  pages 233--258. Birkh\"{a}user Boston, Boston, MA, 2001.

\bibitem{Lacker2015general}
D.~Lacker.
\newblock A general characterization of the mean field limit for stochastic
  differential games.
\newblock {\em Probab. Theory Related Fields}, pages 1--68, 2015.

\bibitem{Lacker2017}
D.~Lacker.
\newblock Limit theory for controlled mckean-vlasov dynamics.
\newblock {\em SIAM J. Control Optim.}, 55:1641--1672, 2017.

\bibitem{lac2020closed}
D.~Lacker.
\newblock On the convergence of closed-loop {N}ash equilibria to the mean field
  game limit.
\newblock {\em Ann. Appl. Probab.}, 30(4):1693--1761, 2020.

\bibitem{Lasry2006}
J.-M. Lasry and P.-L. Lions.
\newblock Jeux \`a champ moyen. {I}. {L}e cas stationnaire.
\newblock {\em C. R. Math. Acad. Sci. Paris}, 343(9):619--625, 2006.

\bibitem{LasryLions2}
J.-M. Lasry and P.-L. Lions.
\newblock Jeux \`a champ moyen. {II}. {H}orizon fini et contr\^ole optimal.
\newblock {\em C. R. Math. Acad. Sci. Paris}, 343(10):679--684, 2006.

\bibitem{nut2018}
M.~Nutz.
\newblock A mean field game of optimal stopping.
\newblock {\em SIAM J. Control Optim.}, 56(2):1206--1221, 2018.

\bibitem{petrov}
V.~V. Petrov.
\newblock {\em Limit theorems of probability theory}, volume~4 of {\em Oxford
  Studies in Probability}.
\newblock The Clarendon Press, Oxford University Press, New York, 1995.
\newblock Sequences of independent random variables, Oxford Science
  Publications.

\bibitem{sznit}
A.-S. Sznitman.
\newblock Topics in propagation of chaos.
\newblock In {\em Ecole d'{\'e}t{\'e} de probabilit{\'e}s de Saint-Flour
  XIX—1989}, pages 165--251. Springer, 1991.

\bibitem{wri1931}
S.~Wright.
\newblock Evolution in {M}endelian populations.
\newblock {\em Genetics}, 16(2):97, 1931.

\end{thebibliography}

\end{document}